\documentclass{siamart190516}
\newcommand{\FL}[1]{{\color{blue}#1}}

\usepackage{amsmath,amssymb,amsfonts,latexsym,stmaryrd}
\usepackage{graphicx,float,epsfig} 
\usepackage{mathrsfs} 
\usepackage{color}
\usepackage{setspace} 
\usepackage{lipsum}
\usepackage{graphicx,float}
\usepackage{color}
\usepackage{epstopdf}
\usepackage{multirow}
\usepackage{svg}
\usepackage{url}
\usepackage{diagbox}
\ifpdf
  \DeclareGraphicsExtensions{.eps,.pdf,.png,.jpg}
\else
  \DeclareGraphicsExtensions{.eps}
\fi

\def\ds{\displaystyle}

\def\O{\Omega}

\def\e{\varepsilon}

\def\l{\lambda}

\def\E{K}

\renewcommand\sp{\mathop{\mathrm{Sp}}\nolimits}

\newcommand{\jump}[1]{\llbracket #1 \rrbracket}

\newcommand{\bb}{\boldsymbol}
\usepackage{booktabs}
\usepackage{float}

\newcommand\bu{\boldsymbol{u}}
\newcommand\bv{\boldsymbol{v}}
\newcommand\bw{\boldsymbol{w}}

\newcommand\bn{\boldsymbol{n}}

\newcommand\curl{\mathop{\mathbf{curl}}\nolimits}

\newcommand\bT{\boldsymbol{T}}

\newcommand\cT{\mathcal{T}}
\def\CT{{\mathcal T}}

\newcommand\btau{\boldsymbol{\tau}}

\newcommand\bphi{\boldsymbol{\varphi}}

\renewcommand\H{\mathrm{H}}

\renewcommand\L{\mathrm{L}}

\renewcommand\O{\Omega}

\renewcommand\div{\mathop{\mathrm{div}}\nolimits}
\newcommand\rot{\mathop{\mathrm{rot}}\nolimits}

\renewcommand\sp{\mathop{\mathrm{sp}}\nolimits}

\newcommand{\vertiii}[1]{{\left\vert\kern-0.25ex\left\vert\kern-0.25ex\left\vert #1 
    \right\vert\kern-0.25ex\right\vert\kern-0.25ex\right\vert}}

\newsiamremark{remark}{Remark}
\newsiamremark{hypothesis}{Hypothesis}
\crefname{hypothesis}{Hypothesis}{Hypotheses}
\newsiamthm{problem}{Problem}
\newsiamthm{claim}{Claim}

\headers{ NCVEM for the  Oseen eigenvalue problem}{D. Adak, F. Lepe and  G. Rivera}

\title{A noncoforming virtual element approximation  for the  Oseen eigenvalue problem\thanks{Submitted to the editors DATE.
\funding{DA and FL were partially supported by DIUBB through project 2120173 GI/C Universidad del B\'io-B\'io. FL was partially supported by 
ANID-Chile through FONDECYT project 11200529 (Chile). GR was supported by Universidad de Los Lagos Regular R02/21 and ANID-Chile through FONDECYT project 1231619 (Chile).
}}}

\author{Dibyendu Adak\thanks{GIMNAP-Departamento de Matem\'atica, Universidad del B\'io-B\'io, Casilla 5-C, Concepci\'on, Chile. 
\texttt{dadak@ubiobio.cl}.}
\and
Felipe Lepe\thanks{GIMNAP-Departamento de Matem\'atica, Universidad del B\'io-B\'io, Casilla 5-C, Concepci\'on, Chile. 
\texttt{flepe@ubiobio.cl}.}
\and
Gonzalo Rivera\thanks{Departamento de Ciencias Exactas, Universidad de Los Lagos, Osorno, Chile.
\texttt{gonzalo.rivera@ulagos.cl}}}

\usepackage{amsopn}

\ifpdf
\hypersetup{
  pdftitle={Virtual element approximation  for the  Oseen eigenvalue problem},
  pdfauthor={Danilo Amigo  and Felipe Lepe}
}
\fi
\def\CE{\mathcal{E}}

\def\CT{{\mathcal T}}
\def\CS{{\mathcal S}}
\def\CI{{\mathcal I}}

\newcommand\bt{\boldsymbol{t}}

\newcommand\Bpsi{\boldsymbol{\psi}}

\newcommand\BH{\bold{H}}
\newcommand\BL{\bold{L}}
\def\CH{{\mathcal H}}
\newcommand\bq{\bold{q}}
\newcommand\BU{\boldsymbol{\mathcal{U}}}
\newcommand\BPI{\boldsymbol{\Pi}^{\nabla}_K}
\newcommand\BPIZ{\boldsymbol{\Pi}^{0}_K}
\newcommand\BPIO{\boldsymbol{\Pi}^{0}_{1,K}}
\newcommand\BPIvec{\boldsymbol{\Pi}^{0}_{k-1,K}}
\newcommand\BPL{\boldsymbol{\mathcal{P}}}
\newcommand\BBeta{\boldsymbol{\beta}}
\newcommand\calN{\boldsymbol{\mathcal{N}}}
\newcommand\Bgama{\boldsymbol{\gamma}}

\def\E{K}

\newcommand\Q{\mathrm{Q}}
\def\CP{{\mathcal P}}
\newtheorem{assumption}{Assumption}
\begin{document}

\maketitle

\begin{abstract}
In this paper we analyze a  nonconforming virtual element method to approximate the eigenfunctions and eigenvalues of the two dimensional Oseen eigenvalue problem. The spaces under consideration lead to a divergence-free method which is capable to capture properly the divergence at discrete level and the eigenvalues and eigenfunctions. Under the compact theory for operators  we prove convergence and error estimates for the method. By employing the theory of compact operators   we recover the double order of convergence of the spectrum. Finally, we present numerical tests to assess the performance of the proposed numerical scheme.
\end{abstract}



\section{Introduction}
\label{sec:intro}
 The numerical approximation of partial differential equations, and the analysis of schemes to approximate the solution of classical models in the pure and applied sciences, is a well-established topic. In particular, the numerical analysis for eigenvalue problems arising from fluid mechanics has paid the attention for researchers from several years, and the literature attending this topic is abundant. We mention \cite{adak2022vem, MR3197278, MR3864690,MR4071826,  MR4022421,MR4229296, LRVSISC, MR2473688, MR3095260} as some references on this topic.

The common aspect of the above references of the mentioned eigenvalue problems are related to the Stokes equations, where the particularity is that the resulting eigenvalue problem results to be selfadjoint and hence, symmetric. This is a desirable feature since we deal with real eigenvalues and eigenfunctions. Now the task is different, since our research program is devoted to the study of non-selfadjoint eigenvalue problems in fluid mechanics, in particular the Oseen eigenvalue problem and hence, the well-developed theory for the Stokes eigenvalue problem must be extended.

The Oseen equations are a linearization of the Navier-Stokes equations and a complete analysis of the source problem for the Oseen system is available in \cite{John2016}.  Here is presented the motivation on the need to study the Oseen system, since to solve the time-dependent Navier-Stokes equations, it is necessary to solve a linear system in each step of time which, precisely is an Oseen type of system.   With this motivation at hand, our task is to analyze numerically the Oseen eigenvalue problem with the aid of a virtual element method (VEM).

The VEM possesses many remarkable features that make  it an attractive numerical strategy for engineering and mathematical communities in order to solve  different model problems. In a general view, the most important features of the VEM are a solid mathematical background, the capability of combine elements
irrespective of geometric shapes, including nonconvex and oddly shaped elements,  arbitrary orders of accuracy and regularity, the easy extension to higher dimensions, among others. A recent state of art of the VEM and its applications is available in \cite{MR4510898} and, more particularly related to the Oseen equations, we mention \cite{MR4369815,MR4419353,MR4510906}.

In the present work we are interested in the application of a nonconforming virtual element method (NCVEM) to solve the nonsymmetric Oseen eigenvalue problem. The NCVEM, introduced in \cite{MR3507277},  has been applied in different elliptic problems such as \cite{MR3741104,MR4287400,MR3576570,MR3646370,MR3529253,MR4032861} and in particular  for  eigenvalue problems we mention \cite{MR4664360,MR4512466,MR3959470} as interesting references with 
excellent results for the discretization of the corresponding spectrums.

For the Oseen eigenvalue problem, we need an inf-sup stable  NCVEM for the Stokes source problem which is available in \cite{MR4032861}. This family of NCVEM has also the capability of holding the incompressibility condition at discrete level, a desirable feature  already available for the conforming VEM  \cite{MR3626409}. Further in \cite{frerichs2022divergence}, a divergence preserving pressure robust conforming  VEM scheme is developed for Stokes problem which is a very appealing feature of conforming approach.  

Recently in \cite{lepe2023finite} and for the best of the author's knowledge, a finite element approximation is proposed  for the Oseen eigenvalue problem as a novel effort to solve numerically this problem. Since the problem is non-symmetric, the ad-hoc strategy for the analysis is the 
introduction of the dual eigenvalue problem in order to obtain error estimates for the method, following the well known theory of \cite{MR1115235}. This approach has been also utilized for a conforming VEM proposed in \cite{danilooseen} for the Oseen eigenproblem. Clearly for the NCVEM approach the strategy is similar but not exactly the same, since the lack of conformity carries  extra terms due to the  variational crime that a non-conforming method naturally involves and  must be correctly controlled. Clearly this analysis  must be done for both, the primal and dual eigenvalue problems.

The formulation under consideration on this paper is the classic velocity-pressure formulation which has the advantage of using the simplest virtual spaces for the approximation. In fact, the implementation of the NCVEM implies a reduction of computational costs compared with \cite{danilooseen}, since the VEM considered on this reference requieres the addition of more degrees of freedom locally for the approximation. However, the cost that we paid on the mathematical analysis is the presence of the term associated to the lack of conformity which, to be correctly controlled and bounded requires additional ingredients such as convexity of the domain. On the other hand, despite of the fact that the method is non-conforming, the solution operator that we define for our work is defined from $\BL^2$ to $\BL^2$ and allows us to utilize the classic theory for compact operators \cite{MR1115235} to carry out the convergence and error analysis of the method similarly as in 
\cite{MR3959470}. Let us remark that this choice for the solution operator is that allows us to consider the compact operators theory, because if otherwise we choose the operator defined from $\mathbf{H}^1$ into $\mathbf{H}^1$ as in \cite{danilooseen}, the theory of non-compact operators must be considered to obtain convergence and error estimates. Moreover, in our contribution we derive an $\BL^2$ error estimate for the velocity via a duality argument, delivering an improvement on the error estimates for this variable under the NCVEM approach. Based on the previous observations, we summarize our contributions to the development of nonconforming VEM for the approximation of the Oseen eigenproblem as follows:
\begin{itemize}
\item Oseen eigenvalue problem is typically a non-self-adjoint eigenvalue problem which requires the convergence analysis of both primal and the corresponding adjoint problems.
\item The $L^2$-convergence of the solution operators associated with the primal and adjoint problems are performed by deﬁning  suitable dual problems associated with the original problems.
\item The solution operator is deﬁned on the $L^2$ space to exploit the well-established spectral theory for compact operator.
\item The nonconforming virtual element space imposes weak continuity of the discrete functions on inter-element edges which needs comparatively less degrees of freedom compared to the conforming approximation.
\end{itemize}

 Theoretically, we are capable to prove that the proposed NCVEM is spurious free according to the theory of  \cite{MR0203473}, which is a consequence of the convergence in norm for compact operators. However, in the numerical section, we report that similarly as in the continuous VEM framework (see \cite{danilooseen,MR4253143,MR4229296} for instance), the stabilization terms of the NCVEM may also introduce spurious eigenvalues and must be avoided. 

The paper is organized as follows: In Section \ref{sec:model_problem} we introduce the Oseen eigenvalue problem and associated weak formulation. We present the functional framework in which the papers is based, namely Hilbert spaces, norms, the variational formulation, regularity of the source and spectral problems, and the solution operator in the same section. All this must be defined for the primal and dual eigenvalue problems. In Section \ref{sec:vem}, we have recollected the divergence-free nonconforming VEM space and discrete formulation of the weak form. The discrete solution operator is also defined in the same section. The a priori error estimates for the source problem in $\L^2$, and broken $\H^1$ norms are defined in Section~\ref{apriori:error}. Eventually, in Section~\ref{spectral:approx}, we have proved the double order of convergence of the spectrum. In Section~\ref{sec:numerical-experiments}, we have assessed some numerical experiments as an evidence of the theoretical estimates.

\subsection{Notation and Preliminaries}
Given any Hilbert space $\mathrm{X}$, we define  $\boldsymbol{\mathrm{X}}:=\mathrm{X}^2$, the space of vectors with entries in $\mathrm{X}$.
For any scalar field $\varphi$ and vector field $\bu$, we introduce the following differential operators: the $\textbf{curl}$ of $\varphi$, defined as $\textbf{curl} \varphi = (\partial_2 \varphi, - \partial_1 \varphi)^\texttt{t}$ where $\texttt{t}$ represents the transpose operator; the gradient of $\bu$, defined as the matrix $(\nabla \bu)=(\partial_j u_i)_{i,j=1,2} $; the rotor of $\bu$, defined as $\rot \bu = \partial_2 u_1- \partial_1 u_2 $; the divergence of $\bu$, defined as $\div \bu = \partial_1 u_1+ \partial_2 u_2 $. Given
$\bold{A}:= (A_{ij}) \in \mathbb{C}^{2 \times 2}$, we define
$\bold{A}: \bold{B}:= \sum_{i,j=1}^2 A_{ij} \overline{B_{ij}}$
as the tensorial product between $\bold{A}$ and $\bold{B}$. The entry  $\overline{B_{ij}}$ represent the complex conjugate of $B_{ij}$. Similarly, given two vectors $\bold{s}= (s_i), \bold{r}= (r_i) \in \mathbb{C}^{2}$, we define the
products
\begin{equation*}
\bold{s} \cdot \bold{r}:= \sum_{i=1}^2 s_i \overline{r_i} \qquad \bold{s} \otimes  \bold{r}:=\bold{s} \overline{\bold{r}}^\texttt{t}=  (s_i \overline{r_j})_{1 \leq i,j\leq 2},
\end{equation*}
as the dot and dyadic product in $\mathbb{C}$. Further, we recollect the definition $\div(\bold{A}):=(\sum_{j=1}^2 \partial_j A_{ij})_{i=1,2}$.

\section{The variational formulation}
\label{sec:model_problem}
Let us describe the model of our study. From now and on,  $\O\subset\mathbb{R}^2$ represents an open bounded polygonal/polyhedral domain with Lipschitz boundary $\partial\O$.  The  equations of the Oseen eigenvalue problem are  given as follows:
\begin{equation}\label{def:oseen-eigenvalue}
\left\{
\begin{array}{rcll}
-\nu\Delta \bu + (\boldsymbol{\beta}\cdot\nabla)\bu + \nabla p&=&\lambda\bu&\text{in}\,\O,\\
\div \bu&=&0&\text{in}\,\O,\\
\displaystyle\int_{\O} p &=&0, &\text{in}\,\O,\\
\bu &=&\boldsymbol{0},&\text{on}\,\partial\O,
\end{array}
\right.
\end{equation}
where $\bu$ is the displacement, $p$ is the pressure and $\boldsymbol{\beta}$ is a given vector field, representing a \textit{steady flow velocity}  and $\nu>0$ is the kinematic viscosity.

Through our paper,  we assume  the existence of  two positive  numbers $\nu^+$ and $\nu^{-}$ such that  $\nu^{-}< \nu< \nu^{+}$. On the other hand, we assume that  $\boldsymbol{\beta}\in \BL^{\infty}(\O,\mathbb{C})$ and solenoidal (i.e. $\div\boldsymbol{\beta}=0$). For the kinematic viscosity and the steady flow velocity we assume the  following standard assumptions  (see \cite{John2016}):
\begin{itemize}
\item $\|\boldsymbol{\beta}\|_{\infty,\O}\sim 1$ if $\nu\leq \|\boldsymbol{\beta}\|_{\infty,\O}$,
\item $\nu\sim 1$ if $\|\boldsymbol{\beta}\|_{\infty,\O}<\nu$.
\end{itemize}

Regarding the convective term,  let us assume that there exists a constant  $\varepsilon_1>0$ such that $\boldsymbol{\beta}\in\BL^{2+\varepsilon_1}(\O,\mathbb{C})$ that leads to  the skew-symmetry of the convective term (see \cite[Remark 5.6]{John2016}) which claims that for all $\bv\in\BH_0^1(\O,\mathbb{C})$, there holds
\begin{equation}
\label{eq:skew}
\int_{\O}(\boldsymbol{\beta}\cdot\nabla)\bv\cdot\bv=0\quad\forall\bv\in\BH_0^1(\O,\mathbb{C}).
\end{equation}

Now we introduce the functional spaces and norms for our analysis.  Let us define the spaces  $\mathcal{X}:=\BH_0^1(\O,\mathbb{C})\times \L_0^2(\O,\mathbb{C})$ together with the space $\mathcal{Y}:=\BH_0^1(\O,\mathbb{C})\times\BH_0^1(\O,\mathbb{C})$ where
$$
\L_0^2(\Omega,\mathbb{C}):=\left\{q\in \L^2(\O,\mathbb{C})\,:\,\int_{\O} q =0 \right\}.
$$
For the space $\mathcal{X}$  we define the norm $\|\cdot\|_{\mathcal{X}}^2:=\|\cdot\|_{1,\O}^2+\|\cdot\|_{0,\O}^2$ whereas for $\mathcal{Y}$ the norm will be $\|(\bv,\bw)\|_{\mathcal{Y}}^2=\|\bv\|_{1,\O}^2+\|\bw\|_{1,\O}^2$, for all $(\bv,\bw)\in\mathcal{Y}$. 

Let us introduce the following sesquilinear forms $a:\mathcal{Y}\rightarrow\mathbb{C}$
and $b:\mathcal{X}\rightarrow\mathbb{C}$ defined by 
\begin{equation*}
a(\bw,\bv):=a_{\text{sym}}(\bw,\bv) + a_{\text{skew}}(\bw,\bv)\quad
\text{and}
\quad
b(\bv,q):=-\int_{\O}q\,\div\bv,
\end{equation*}
where $a_{\text{sym}},a_{\text{skew}}:\mathcal{Y}\rightarrow\mathbb{C}$ are two sesquilinear forms  defined by
\begin{equation*}
a_{\text{sym}}(\bw,\bv):=\int_{\Omega}\nu\nabla\bw :\nabla\bv\quad
\text{and}
\quad
a_{\text{skew}}(\bw,\bv):=\frac{1}{2} \Big(a^{\BBeta}(\bw,\bv)-a^{\BBeta}(\bv,\bw) \Big),
\end{equation*}
where, $a^{\BBeta}(\bw,\bv):=\int_{\Omega}(\boldsymbol{\beta}\cdot\nabla)\bw\cdot\bv$. On the other hand we define the following sesquilinear form   $c(\bw,\bv) := (\bw,\bv)_{0,\O}$ as the standard inner product in $\BL^{2}(\O,\mathbb{C})$. 
With these sesquilinear forms at hand, we write the following weak formulation for \eqref{def:oseen-eigenvalue}: Find $\lambda\in\mathbb{C}$ and $(\boldsymbol{0},0)\neq(\bu,p)\in \mathcal{X}$ such that 
\begin{equation}\label{def:oseen_system_weak}
	\left\{
	\begin{array}{rcll}
a(\bu,\bv) + b(\bv,p)&=&\lambda c(\bu,\bv)&\forall \bv\in \BH_0^1(\O,\mathbb{C}),\\
b(\bu,q)&=&0&\forall q\in \L_0^2(\O,\mathbb{C}).
\end{array}
	\right.
\end{equation}

Observe that the resulting eigenvalue problem is  non-symmetric due the presence of the  sesquilinear form $a^{\boldsymbol{\beta}}(\cdot,\cdot)$. 
Let us define the kernel $\mathcal{K}$ of $b(\cdot,\cdot)$ as follows
\begin{equation*}
\mathcal{K}:=\{\bv\in\BH_0^1(\O,\mathbb{C})\,:\,  b(\bv, q)=0\,\,\,\,\forall q\in\L_0^2(\O,\mathbb{C})\}.
\end{equation*}
With this space available, it is  straightforward to verify using \eqref{eq:skew} that  $a(\cdot,\cdot)$ is $\mathcal{K}$-coercive. Moreover, the bilinear form $b(\cdot,\cdot)$ satisfies the following inf-sup condition
\begin{equation}
\label{ec:inf-sup_cont}
\displaystyle\sup_{\btau\in\BH_0^1(\O,\mathbb{C})}\frac{b(\btau,q)}{\|\btau\|_{1,\O}}\geq\beta\|q\|_{0,\O}\quad\forall q\in\L^2_0(\O,\mathbb{C}),
\end{equation}
where $\beta>0$. Let us introduce  the solution operator, which we denote by $\bT$ and is defined as follows
\begin{equation}\label{eq:operador_solucion_u}
	\bT:\BL^2(\O,\mathbb{C}) \rightarrow \BL^2(\O,\mathbb{C}),\qquad \boldsymbol{f}\mapsto \bT\boldsymbol{f}:=\widehat{\bu}, 
\end{equation}
where the pair  $(\widehat{\bu}, \widehat{p})\in\mathcal{X}$ is the solution of the following well-posed source problem
\begin{equation}\label{def:oseen_system_weak_source}
	\left\{
	\begin{array}{rcll}
a(\widehat{\bu}, \bv)+b(\bv,\widehat{p})&=&c(\boldsymbol{f},\bv)&\forall \bv\in \BH_0^1(\O,\mathbb{C}),\\
b(\widehat{\bu},q)&=&0&\forall q\in \L_0^2(\O,\mathbb{C}),
\end{array}
	\right.
\end{equation}
implying  that $\bT$ is well defined due to the Babu\v{s}ka-Brezzi theory. Moreover, from \cite[Lemma 5.8]{John2016} we have the following estimates for the velocity and pressure, respectively
\begin{equation*}\label{eq:estimatefuente_velocity}
\ds\|\nabla\widehat{\bu}\|_{0,\O}\leq\frac{C_{pf}}{\nu}\|\boldsymbol{f}\|_{0,\O},
\end{equation*}
\begin{equation*}
\label{eq:estimatefuente_pressure}
\ds \|\widehat{p}\|_{0,\O}^2\leq \frac{1}{\beta}\left( \|\boldsymbol{f}\|_{0,\O}+\nu^{1/2}
\|\nabla\widehat{\bu}\|_{0,\O}\left(\nu^{1/2}+C_{pf}\frac{\|\boldsymbol{\beta}\|_{\infty,\O}}{\nu^{1/2}}\right)\right),
\end{equation*}
where $C_{pf}>0$ represents the constant of the  Poincar\'e-Friedrichs inequality and $\beta>0$ is the inf-sup constant given in \eqref{ec:inf-sup_cont}.

It is easy to check that $(\lambda , (\bu , p)) \in  \mathbb{C}  \times \mathcal{X}$  solves \eqref{def:oseen_system_weak} if and only if $(\kappa , \bu )$ is an eigenpair of $\bT , i.e., \bT \bu  = \kappa \bu$  with $\kappa  := 1/\lambda$ and $\lambda\neq 0$.

A key point for the analysis is the additional regularity of the solution. To obtain this, the assumptions on $\boldsymbol{\beta}$ are important,. To make matters precise, if   the convective term is well defined, it is possible to  resort to  the classic Stokes regularity results available on the literature (see \cite{MR1600081} for instance). Hence, the  following additional regularity result for the solutions of the Oseen system holds.
\begin{theorem}\label{th:regularidadfuente}
There exists $s>0$  such that for all $\boldsymbol{f} \in \BL^2(\O,\mathbb{C})$, the solution $(\widehat{\bu},\widehat{p})\in\mathcal{X}$ of problem \eqref{def:oseen_system_weak_source}, satisfies for the velocity $\widehat{\bu}\in  \BH^{1+s}(\O,\mathbb{C})$, for the pressure $\widehat{p}\in \H^s(\Omega,\mathbb{C})$, and
 \begin{equation*}
\|\widehat{\bu}\|_{1+s,\O}+\|\widehat{p}\|_{s,\O}\leq C\|\boldsymbol{f}\|_{0,\O},
 .\end{equation*}
where $C := \dfrac{C_{pf}}{\beta}\max\left\lbrace 1, \dfrac{C_{pf}\|\bb{\beta}\|_{\infty,\O}}{\nu}\right\rbrace$ and $\beta > 0$ is the constant associated to the inf-sup condition \eqref{ec:inf-sup_cont}. Further, if $(\bu,p)$ is an eigenfunction satisfying \eqref{def:oseen_system_weak}, then there exists $r>0$, not necessarily equal to $s$, such that $(\bu,p) \in \mathcal{X}\cap (\BH^{1+r}(\Omega,\mathbb{C}) \times \H^r(\Omega,\mathbb{C}))$ and the following bound holds
\begin{equation*}
\|\widehat{\bu}\|_{1+r,\O}+\|\widehat{p}\|_{r,\O}\leq C\|\widehat{\bu}\|_{0,\O}.
 \end{equation*}
\end{theorem}

Observe that the following compact inclusion $\BH^{1+s}(\O,\mathbb{C})\hookrightarrow \BL^{2}(\O,\mathbb{C})$,  implying directly the  compactness of $\bT$. Finally, we have the following spectral characterization for $\bT$.
\begin{lemma}(Spectral Characterization of $\bT$).
The spectrum of $\bT$ is such that $\sp(\bT)=\{0\}\cup\{\kappa_{k}\}_{k\in{N}}$ where $\{\kappa_{k}\}_{k\in\mathbf{N}}$ is a sequence of complex eigenvalues that converge to zero, according to their respective multiplicities. 
\end{lemma}

We conclude this section by redefining the spectral problem \eqref{def:oseen_system_weak} in order to  simplify the notations for the forthcoming analysis. With this in mind, let us introduce the sesquilinear form $A:\mathcal{X}\times\mathcal{X}\rightarrow\mathbb{C}$ defined by 
\begin{equation*}
A((\bu,p);(\bv,q)):= a(\bu,\bv)+ b(\bv,p)- b(\bu,q),\quad\forall (\bv,q)\in\mathcal{X}, 
\end{equation*}
which allows us to rewrite problem \eqref{def:oseen_system_weak} as follows: Find $\lambda\in\mathbb{C}$ and $(\boldsymbol{0},0)\neq(\bu,p)\in\mathcal{X}$ such that
\begin{equation}
\label{eq:eigenA}
A((\bu,p),(\bv,q))=\lambda c(\bu,\bv)\quad\forall (\bv,q)\in\mathcal{X}.
\end{equation}

Since the problem is non-selfadjoint, it is necessary to introduce the adjoint eigenvalue problem, which  reads as follows (see \cite[eq. (2.17)]{lepe2023finite}): Find $\lambda^*\in\mathbb{C}$ and a pair $(\boldsymbol{0},0)\neq(\bu^*,p^*)\in\mathcal{X}$ such that  
\begin{equation}\label{def:oseen_system_weak_dual_eigen}
	\left\{
	\begin{array}{rcll}
a(\bv,\bu^*)-b(\bv,p^*)&=&\overline{\lambda}c(\bv,\bu^*)&\forall \bv\in \BH_0^{1}(\O,\mathbb{C}),\\
-b( \bu^*,q)&=&0&\forall q\in \L_0^2(\O,\mathbb{C}).
\end{array}
	\right.
\end{equation}
Now we introduce the adjoint of \eqref{eq:operador_solucion_u} defined  by 
\begin{equation*}\label{eq:operador_adjunto_solucion_u}
	\bT^*:\BL^{2}(\O,\mathbb{C})\rightarrow \BL^{2}(\O,\mathbb{C}),\qquad \boldsymbol{f}\mapsto \bT^*\boldsymbol{f}:=\widehat{\bu}^*, 
\end{equation*} 
where $\widehat{\bu}^*\in\BH_0^{1}(\O,\mathbb{C})$ is the adjoint velocity of $\widehat{\bu}$ and solves the following adjoint source  problem: Find  
$(\widehat{\bu}^*, \widehat{p}^*)\in\mathcal{X}$ such that 
\begin{equation}\label{def:oseen_system_weak_dual_source}
	\left\{
	\begin{array}{rcll}
	a(\bv,\widehat{\bu}^*)-b(\bv,\widehat{p}^*)&=&c(\bv,\boldsymbol{f})&\forall \bv\in \BH_0^{1}(\O,\mathbb{C}),\\
-b( \widehat{\bu}^*,q)&=&0&\forall q\in \L_0^2(\O,\mathbb{C}).
\end{array}
	\right.
\end{equation}
Similar to Theorem \ref{th:regularidadfuente}, let us assume that the dual source and eigenvalue problems are such that the following estimate holds.

\begin{theorem}\label{th:regularidadfuente_dual}
There exist $s^*>0$ such that for all $\boldsymbol{f} \in \BL^2(\O,\mathbb{C})$, the solution $(\widehat{\bu}^*,\widehat{p}^*)$ of problem \eqref{def:oseen_system_weak_dual_source}, satisfies $\widehat{\bu}^*\in  \BH^{1+s^*}(\Omega,\mathbb{C} )$ and $\widehat{p}^*\in \H^{s^*}(\Omega,\mathbb{C} )$, and
 \begin{equation*}
\|\widehat{\bu}^*\|_{1+s^*,\O}+\|\widehat{p}^*\|_{s^*,\O}\leq C\|\boldsymbol{f}\|_{0,\O},
\end{equation*}
where $C>0$ is defined in Theorem \ref{th:regularidadfuente}. Further, if $(\bu^*,p^*) $ is an eigenfunction satisfying \eqref{def:oseen_system_weak_dual_eigen}, then  there exists $r^*>0$, not necessarily equal to $s^*$, such that  $(\bu^*,p^*) \in \mathcal{X}\cap ((\BH^{1+r^*}(\Omega,\mathbb{C}) \times \H^{r^*}(\Omega,\mathbb{C})))$ and the following bound holds
\begin{equation*}
\|\widehat{\bu}^*\|_{1+r^*,\O}+\|\widehat{p}^*\|_{r^*,\O}\leq C\|\widehat{\bu}^*\|_{0,\O},
 .\end{equation*}
\end{theorem}

Finally the spectral characterization of $\bT^*$ is given as follows.
\begin{lemma}(Spectral Characterization of $\bT^*$).
The spectrum of $\bT^*$ is such that $\sp(\bT^*)=\{0\}\cup\{\kappa_{k}^*\}_{k\in{N}}$ where $\{\kappa_{k}^*\}_{k\in\mathbf{N}}$ is a sequence of complex eigenvalues that converge to zero, according to their respective multiplicities. 
\end{lemma}
It is easy to prove that if $\kappa$ is an eigenvalue of $\bT$ with multiplicity $m$, $\overline{\kappa^*}$ is an eigenvalue of $\bT^*$ with the same multiplicity $m$.

Let us define the sesquilinear form $\widetilde{A}:\mathcal{X}\times\mathcal{X}\rightarrow\mathbb{C}$ by
\begin{equation*}
\widetilde{A}((\bv,q),(\bu^*,p^*)):=a(\bv,\bu^*)-b(\bv,p^*)+b(\bu^*,q),
\end{equation*}
which allows us to rewrite the dual eigenvalue problem  \eqref{def:oseen_system_weak_dual_eigen} as follows:  Find $\lambda^*\in\mathbb{C}$ and the pair $(\boldsymbol{0},0)\neq(\bu^*,p^*)\in\mathcal{X}$ such that  
\begin{equation*}
\label{eq:eigen_A_dual}
\widetilde{A}((\bv,q),(\bu^*,p^*))=\lambda^*c(\bv,\bu^*)\quad\forall (\bv,q)\in\mathcal{X}.
\end{equation*}

\section{The virtual  element method}
\label{sec:vem}
In order to discretize the Oseen eigenvalue problem, we first go over nonconforming virtual element space in this section. The original purpose of this space's development was to approximate the Stokes equation numerically. In our research, we utilise the improved version on Ref. \cite{MR4032861}. 
\subsection{Mesh notation and mesh regularity}
We consider the family of meshes $\{ \CT_h\}_{h>0}$ such that each mesh $\CT_h$ is a partition of the domain $\Omega$ into a finite collection of non-overlapping, polygonal elements $K$ with mesh diameter $h_K$, and boundary $\partial K$. As usual, we define $h:= \max_{K \in\CT_h } h_K$. Furthermore, $\CE:=\CE_{\text{int}} \cup \CE_{\text{bdy}}$ denotes the set of mesh edges of $\CT_h$ where $\CE_{\text{int}}$ and $\CE_{\text{bdy}}$ denotes respectively the subsets of the interior and boundary mesh edges.

Consider the polygonal element $K \in \CT_h$. We denote the outward pointing normal and the tangent unit vector to the polygonal boundary $\partial K$ by $\bn_K$ and $\bt_K$, respectively. For every edge $e \subset \partial K$, we denote by  $\bn_e$,and $\bt_e$ the normal and tangent unit vectors to $e$, respectively. Conventionally, we assume that $\bn_e$ points out of $\Omega$ if $e$ is a boundary edge, and $\bn_e$ and $\bt_e$ form an anti-clockwise oriented pair along every internal edge $e$. Accordingly, it holds that $\bn_e:= (t_2,-t_1)$ whenever $\bt_e:=(t_1,t_2)$. 

We define the space of piecewise polynomials of degree $k \geq 0$ by
\begin{equation*}
\CP_k(\CT_h):=\{q \in \L^2(\Omega): q|_K \in \CP_k(K) \quad \forall K \in  \CT_h \}.
\end{equation*}
Similarly, for all integers $l > 0$, we define the broken Sobolev space of degree $l$ on $\CT_h$ of  vector-valued fields, whose components are in $\BH^l(K)$ for all mesh elements $K$, as
\begin{equation*}
\BH^l(\CT_h):=\{\bphi \in  \BL^2(\Omega) : \bphi|_K \in \BH^l(K) \quad \forall K \in  \CT_h \}.
\end{equation*}
We endow this functional space with the broken semi-norm
\begin{equation*}
|\bphi_h|_{\FL{\ell},h}:=\Big ( \sum_{K \in \CT_h} |\bphi|_{\FL{\ell},K}^2 \Big)^{1/2}.
\end{equation*} 
Consider the internal edge $e \subset \partial K^{+} \cap \partial K^{-}$, $\text{where}\, K^{+}, K^{-} \in \CT_h$, and $\bn_e$ points from $K^{+}$ to $K^{-}$. We define the jump of a function $\bv$ through $e$ by $\jump{\bv}|_e:= \bv|_{K^{+}}-\bv|_K^{-}$ and, 
for boundary edges, we define $\jump{\bv}|_e:= \bv|_e$. 
For the a priori error analysis, we need the following regularity assumptions on the mesh family $\{ \CT_h \}_{h>0}$.
\begin{assumption}(Mesh Regularity) There exists a positive constant $\sigma >0$ such that for all $K \in \CT_h$ it holds that
\begin{itemize}
\item $(M1)$ the ratio between every edge length and the diameter $h_K$ is bigger than $\sigma$;
\item $(M2)$ $K$ is star-shaped with respect to a ball of radius $\rho_K$ satisfying $\rho_K > \sigma h_K$. 
\end{itemize}
\end{assumption}
These mesh assumptions impose some constraints that are admissible for the formulation of the method discussed in the next subsection. 
In view of the following analysis, it is helpful to define the continuous bilinear forms $a(\cdot,\cdot)$, $b(\cdot,\cdot)$ and $c(\cdot,\cdot)$ on the discrete space $\BH^1(\CT_h)$ as a sum of local contributions.
\begin{align*}
a(\bw,\bv)&:=\sum_{\E\in\CT_{h}}a_{\text{sym}}^{\E}(\bw,\bv)+a_{\text{skew}}^{\E}(\bw,\bv)\qquad \forall \bw, \bv\in \BH^1(\CT_h),\\
b(\bv,q)&:=\sum_{\E\in\CT_{h}}b^{\E}(\bv,q) \qquad \forall \bv\in  \BH^1(\CT_h)\; \text{and }q\in \L_0^2(\O,\mathbb{C}),\\
c(\bw,\bv)&:=\sum_{\E\in\CT_{h}}c^{\E}(\bw,\bv) \qquad \forall \bw,\bv\in \BL^{2}_0(\O,\mathbb{C}),\\
A((\bu,p),(\bv,q))&:=\sum_{K \in \CT_h} A^K((\bu,p),(\bv,q)) \qquad \forall (\bu,p), (\bv,q) \in \mathcal{X},
\end{align*}
where the local forms $a_{\text{sym}}^{\E}(\cdot,\cdot)$, $a_{\text{skew}}^{\E}(\cdot,\cdot)$, $b^K(\cdot,\cdot)$,  $c^K(\cdot,\cdot)$, and $A^K(\cdot,\cdot)$ corresponds to the restrictions of 
$a_{\text{sym}}(\cdot,\cdot)$, $a_{\text{skew}}(\cdot,\cdot)$, $b(\cdot,\cdot)$, $c(\cdot,\cdot)$,  and $A(\cdot,\cdot)$, respectively, from $\O$ to each element $K$ of $\CT_h$.

In the same way, we split  elementwise the norm $\L^2(\O,\mathbb{C})$ by 
\begin{equation*}
\label{eq:normas}
 \|q\|_{0,\O}:=\left(\sum_{K\in\mathcal{T}_h}\|q\|_{0,K}^2\right)^{1/2}\quad\forall q\in \L^2(\O,\mathbb{C}).
\end{equation*}

\subsection{Local and global discrete space}
In what follows we summarize the key ingredients for the discrete analysis, given by \cite{MR4032861}. For $K \in \CT_h$, we define for $k\geq 1$ the following auxiliary finite dimensional space
\begin{equation}
\label{aux_space}
\widetilde{\boldsymbol{\CS}}(K):=\{\bv \in \BH^1(K): \div \bv \in \CP_{k-1}(K), \rot \bv \in \CP_{k-1}(K), \bv \cdot \bn_e \in \CP_{k}(e) \forall e \subset \partial K\}.
\end{equation}
We decompose the space $\widetilde{\boldsymbol{\CS}}(K)$ in \eqref{aux_space} into the direct sum of two subspaces as follows
\begin{equation*}
\widetilde{\boldsymbol{\CS}}(K)= \widetilde{\boldsymbol{\CS}}_1(K) \oplus \widetilde{\boldsymbol{\CS}}_0(K),
\end{equation*}
where $\widetilde{\boldsymbol{\CS}}_1(K):= \{ \bv \in \widetilde{\boldsymbol{\CS}}(K): \div \bv=0, \bv \cdot \bn_{K}|_{\partial K}=0\}$
and 
\begin{equation}
\label{Loc:space:0}
\widetilde{\boldsymbol{\CS}}_0(K):= \{ \bv \in \widetilde{\boldsymbol{\CS}}(K): \rot \bv=0\}.
\end{equation}
Additionally, we introduce the   space 
\begin{equation}
\label{Loc:space:H}
\widetilde{\CH}:= \{ \phi \in \H^2(K), \Delta^2 \phi \in \CP_{k-1}(K), \phi|_{e}=0, \Delta \phi|_{e} \in \CP_{k-1}(e) \forall e \subset \partial K\}.
\end{equation}
The local space is constructed as sum of \eqref{Loc:space:0}, and curl of \eqref{Loc:space:H} as follows
\begin{equation*}
\widetilde{\BU}= \widetilde{\boldsymbol{\CS}}_1(K) \oplus \curl\widetilde{\CH}.
\end{equation*}
We define the following operators:
\begin{itemize}
\item (H1) the edge polynomial moments:
\begin{equation*}
\frac{1}{|e|} \int_e \bv \cdot \bn_e q_k \quad  \forall q_k \in  \CP_k(e), \forall e \subset \partial K;
\end{equation*}

\item  (H2) the edge polynomial moments:
\begin{equation*}
\frac{1}{|e|} \int_e \bv \cdot \bt_e q_{k-1} \quad  \forall q_{k-1} \in  \CP_{k-1}(e), \forall e \subset \partial K;
\end{equation*}

\item (H3) the elemental polynomial moments:

\begin{equation*}
\frac{1}{|K|} \int_K \bv \cdot \bq_{k-2} \quad  \forall \bq_{k-2} \in  \nabla \CP_{k-1}(K);
\end{equation*}

\item  (H4) the elemental polynomial moments:

\begin{equation*}
\frac{1}{|K|} \int_K \bv \cdot \bq_{k}^{\perp} \quad  \forall \bq_{k}^{\perp} \in  (\nabla \CP_{k+1}(K))^{\perp}.
\end{equation*}
\end{itemize}
Here, $(\nabla \CP_{k+1}(K))^{\perp}$ is the $\BL^2$-orthogonal complement of $\nabla \CP_{k+1}(K)$ in $\boldsymbol{\CP}_k(K)$, where $\boldsymbol{\CP}_k(K)$ is vector-valued polynomial space on $K$ of order $k$.  Following \cite{MR4032861}, we deduce that the set of operators above  provides a set of the degrees of freedom of the discrete space $\widetilde{\BU}$. 
Based on the computational aspect, we introduce the elliptic projection operator $\BPI:\widetilde{\BU} \rightarrow \BPL_k(K)$:
\begin{equation}
\label{def:projection:ell}
\begin{split}
 a^K_{\text{sym}}(\BPI \bu, \bq)&=a^K_{\text{sym}}(\bu,\bq)\quad \forall \bq \in \boldsymbol{\CP}_k(K), \\
 \int_{\partial K} \BPI \bu-\bu & =0.
\end{split}
\end{equation}
From the definition of the projection operator $\BPI$, we deduce that the right-hand side of \eqref{def:projection:ell} is  computable from (H1)-(H4). By employing the projection operator $\BPI$, we define a local computational space which is subspace of $\widetilde{\BU}$ as follows:
\begin{equation*}
\begin{split}
\BU(K)&:=  \{\bv \in \widetilde{\BU}: \int_{K} (\bv-\BPI \bv) \cdot \bq_k=0 \quad \forall \bq_k \in (\nabla \CP_{k+1}(K))^{\perp} / (\nabla \CP_{k-1}(K))^{\perp} \quad  \\
& \text{and} \quad 
\int_{e} (\bv-\BPI \bv) \cdot \bn_e q_k=0 \quad \forall q_k \in \CP_k(e)/\CP_{k-1}(e), \quad \forall e \subset \partial K \},
\end{split}
\end{equation*}
where the symbol $\mathcal{V}/ \mathcal{V}_1$ denotes the subspace of space $\mathcal{V}$ consisting of polynomials that are $\BL^2(K)$-orthogonal to space $\mathcal{V}_1$.
Since the projector $\BPI$ is invariant on polynomial function space  $\BPL_k(K)$, we deduce that $\BPL_k(K) \subset \BU(K)$. Furthermore, (H1) and (H3) are a set of degrees of freedom for $\BU(K)$. For $K \in \CT_h$, the local space $\BU(K)$ is unisolvent with respect to a certain set of bounded linear operators, which are defined as follows:
\begin{itemize}
\item the edge polynomial moments:
\begin{equation*}
\frac{1}{|e|} \int_e \bv \cdot \bq_{k-1} \quad  \forall \bq_{k-1} \in  \BPL_{k-1}(e), \forall e \subset \partial K;
\end{equation*}

\item the elemental polynomial moments

\begin{equation*}
\frac{1}{|K|} \int_K \bv \cdot \bq_{k-2} \quad  \forall \bq_{k-2} \in  \BPL_{k-2}(K);
\end{equation*}

\end{itemize}
According to the definition of the virtual space $\BU(K)$, the term $\BPI \bv$ is computable for all $\bv \in \BU(K) $. 
Now we define the global nonconforming virtual space by
\begin{multline}
\label{NCVEM:space}
\BU_h:=\left\{\bv \in \BL^2(\O,\mathbb{C}) : \bv|_K \in \BU(K) \forall K \in \CT_h,\right.\\
\left. \int_e \jump{\bv}_e \cdot \bq_{k-1}=0 \,\ \forall  \bq_{k-1} \in \BPL_{k-1}(e)  \forall e \in  \CE \right\}.
\end{multline}
Clearly the space $\BU_h$ is not continuous over $\Omega$ since $\BU_h\not\subset \BH^1(\Omega)$. In the next lemma, we summarize two technical results that will be helpful in the derivation of the a priori estimates of the next sections. Further, we highlight that the $\BL^2$ projection operator $\BPIZ$ is computable on $\BU(K)$ \cite{zhang2023divergence}. To define the interpolation operator $\mathcal{I}$ on the space $\BU_h$, for each element $K \in \CT_h$, we denote by $\Sigma_i$, the operator associated with the i-th local degree of freedom, $i = 1,2,\ldots,N^{\text{dof}}$. From the above construction, it is easily seen that for every smooth enough function $\bv$, there exists an unique element $\mathcal{I}_K \bv \in \BU_h(K)$ such that $\Sigma_i(\bv-\mathcal{I}_K \bv)=0$ ,  $\forall$ $i = 1,2,\ldots,N^{\text{dof}}$.
\begin{remark}\label{equivalencia}
From the Poincar\'e-Friedrichs inequalit (see \cite{BrennerSIAM}) it is straightforward to check that $|\cdot|_{1,h}$ is equivalent to $\|\cdot\|_{1,h}$ on $\BU_h$, i.e there exists constant $C_{1,p} >0$ such that
\begin{equation}\label{eq:equivalencia}
C_{1,p}\|\bv_h\|_{1,h}\leq |\bv_h|_{1,h}\leq \|\bv_h\|_{1,h},\qquad \forall \bv_h\in \BU_h.
\end{equation}
\end{remark}Then, we define the global interpolation $\mathcal{I}$ for $\BU_h$ by setting $\mathcal{I}|_K = \mathcal{I}_K$ $\forall K \in \CT_h $. Two technical conclusions that will be useful in deriving the a priori estimates of the following sections are summarized in the next lemma.  
\begin{lemma}
\label{Int_poly_err}
The following statements hold:
\begin{itemize}
\item For each polygon $K \in \CT_h$ and any t such that $1 \leq t \leq k+1$, it holds that  
\begin{equation}
\|\bv - \mathcal{I}_K \bv \|_{m,K} \leq C h^{t-m} |\bv|_{t,D(K)} \qquad m=0,1. 
\end{equation}
\item For each polygon $K \in \CT_h$ and any  $t$ such that $1 \leq t \leq k+1$, there exists a polynomial $\bv_{\pi} \in \BPL_{k}(K) $, such that 
\begin{equation}
\|\bv - \bv_{\pi} \|_{m,K} \leq C h^{t-m} |\bv|_{t,K} \qquad m=0,1, 
\end{equation}
\end{itemize}
where $D(K)$ denotes the union of polygons in $\CT_h$ intersecting $K$.
\end{lemma}
On the other hand, the discrete pressure space is given by
\begin{equation*}
\Q_h:=\{q_h \in \L^{2}(\O,\mathbb{C})\,:\, q_h|_K\in\mathcal{P}_{k-1}(K),\,\quad\forall K\in \CT_h\},
\end{equation*}
We also introduce the $\L^2$-orthogonal projection $\mathcal{R}_h:\L^2(\O)\to \Q_h$
and the following approximation result  holds for  $0\leq t\leq 1$ (see \cite{MR3507271} for instance)
\begin{equation}\label{eq:estimateL2}
\|q-\mathcal{R}_hq\|_{0,\O}\leq Ch^t\|q\|_{\FL{t},\O},\qquad\forall q\in\H^{t}(\O).
\end{equation}
Let us introduce the operator $\div_h(\cdot)$ which corresponds to the discretized global form of the divergence operator, i.e., $(\div_h \bv)|_K = \div(\bv|_K)$ for all $K\in\CT_h$ (and sufficiently regular $\bv$). From the above construction, we deduce that $\div_h \BU_h \subset\Q_h$, and the relation between the virtual interpolation operator and $\mathcal{R}_{h}$ is as follows $\div_h \mathcal{I}\bv=\mathcal{R}_h\div_h\bv$ for all  $\bv\in\BH^1(\O)$.
Now, let $S^K(\cdot,\cdot)$ be any symmetric positive definite bilinear form chosen to satisfy
\begin{equation}
\label{eq:S_bounded}
c_0 a_{\text{sym}}^K(\bv_h,\bv_h)\leq S^K(\bv_h,\bv_h)\leq c_1 a_{\text{sym}}^K(\bv_h,\bv_h),
\end{equation}
for some positive constants $c_0$ and $c_1$ depending only on the constant $\sigma$ from the mesh assumptions $M1$ and $M2$.
Then, for all $\bw_h,\bv_h\in\BU_h$, we introduce on each element $K$ the local (and computable) bilinear forms
\begin{align*}
a_{h,\text{sym}}^K(\bw_h,\bv_h)&:=a_{\text{sym}}^K(\BPI\bw_h,\BPI\bv_h)+S^K(\bw_h-\BPI\bw_h,\bv_h-\BPI\bv_h);\\
a_{h,\text{skew}}^K(\bw_h,\bv_h)&:=\dfrac{1}{2}\int_K\left(\left(\boldsymbol{\beta}\cdot\BPIvec \nabla \right)\bw_{h}\cdot\BPIZ\bv_h-\left(\boldsymbol{\beta}\cdot\BPIvec \nabla \right)\bv_{h}\cdot\BPIZ\bw_h\right);\\
c_h^K(\bw_h,\bv_h)&:=c^K(\BPIZ\bw_h,\BPIZ\bv_h).
\end{align*}
The construction of $a_{h,\text{sym}}^K(\cdot, \cdot)$ and $c_h^K(\cdot,\cdot)$ guarantees the usual consistency and stability properties of the VEM. With this considerations at hand,  the following result holds true which is direct from \cite{MR2997471}.
\begin{lemma}
The local bilinear forms $a_{h,\text{sym}}^K(\cdot, \cdot)$ and $c_h^K(\cdot,\cdot)$ on each element $K$ satisfy:
\begin{itemize}
\item Consistency: for all $h>0$ and for all $K\in\CT_h$ we have that
\begin{align*}
a_{h,\text{sym}}^K(\bv_h, \bq_k)&=a_{\text{sym}}^K(\bv_h, \bq_k)\qquad\forall \bq_k\in \BPL_{k}(K), \\
c_{h}^K(\bv_h, \bq_k)&=c^K(\bv_h, \bq_k)\qquad\forall \bq_k\in \BPL_{k}(K).
\end{align*}
\item Stability:  for all $K\in\CT_h$, there exist positive constants $c_*$, $c^*$ and $d^*$, independent of $h$, such that
\begin{align*}
c_*a_{\text{sym}}^K(\bv_h, \bv_h)\leq& a_{h,\text{sym}}^K(\bv_h, \bv_h)\leq c^*a_{\text{sym}}^K(\bv_h, \bv_h) \qquad\forall \bv_h\in \BU_h, \\
&c_{h}^K(\bv_h, \bv_h)\leq d^*c^K(\bv_h, \bv_h)\qquad\forall \bv_h\in \BU_h.
\end{align*}
\end{itemize}
\end{lemma}
For the bilinear form $b_h(\cdot,\cdot)$, we do not introduce any approximation and simply set
$$b_h(\bv_h,q_h):=\sum_{K\in\CT_h}b^K(\bv_h,q_h)=-\sum_{K\in\CT_h}\int_Kq_h\div\bv_h,\qquad \forall  \bv_h\in\BU_h,q_h\in\Q_h.$$
Since $b_h(\bv_h,q_h)$ is computable in each element $K \in\CT_h$ with the aid of the degrees of freedom defined on $\boldsymbol{\mathcal{U}}(K)$.
Naturally for all $\bw_h,\bv_h\in\BU_h$ we can introduce the following bilinear form 
$$a_h(\bw_h,\bv_h):=\sum_{K\in\CT_h}a_h^K(\bw_h,\bv_h)=\sum_{K\in\CT_h}a_{h,\text{sym}}^K(\bw_h,\bv_h)+a_{h,\text{skew}}^K(\bw_h,\bv_h).$$
It is easy to check that $a_h(\cdot,\cdot)$ and $b_h(\cdot,\cdot)$ are continuous sesquilinear forms. Indeed, for $a_h(\cdot,\cdot)$ we have
\begin{equation}
\label{eq:bound_ah}
|a_h(\bu_h,\bv_h)|\leq |a_{h,\text{sym}}(\bu_h,\bv_h)|+ |a_{h,\text{skew}}(\bu_h,\bv_h)|,
\end{equation}
where we need to estimate each contribution on the right hand side of the inequality above. For the symmetric part we have
\begin{multline}
\label{eq:bound_sym}
|a_{h,\text{sym}}(\bu_h,\bv_h)|=\left| \sum_{K\in\CT_h}a_{\text{sym}}^K(\bu_h,\bv_h)+S^K(\bu_h-\boldsymbol{\Pi}_{K}^{\nabla}\bu_h,\bv_h-\boldsymbol{\Pi}_{K}^{\nabla}\bv_h)\right|\\
\leq \sum_{K\in\CT_h}\nu\|\nabla\boldsymbol{\Pi}^{\nabla}_K\bu_h\|_{0,K}\|\nabla\boldsymbol{\Pi}^{\nabla}_K\bv_h\|_{0,K}+c_1\nu\|\nabla(\boldsymbol{\Pi}^{\nabla}_K\bu_h-\bu_h)\|_{0,K}\|\nabla(\boldsymbol{\Pi}_{K}^{\nabla}\bv_h-\bv_h)\|_{0,K}\\
\leq \nu\max\{\widetilde{c}_1,1 \}\|\bu_h\|_{1,h}\|\bv_h\|_{1,h},
\end{multline}
where the constant $\widetilde{c}_1$ is the sum of all the constants $c_1$ involved in \eqref{eq:S_bounded} for each element $K\in\CT_h$. Now, for the skew-symmetric part we have 
\begin{multline}
\label{eq:bound_skew}
|a_{h,\text{skew}}(\bu_h,\bv_h)|\leq \frac{1}{2}\left|\sum_{K\in\CT_h}\int_K(\boldsymbol{\beta}\cdot\boldsymbol{\Pi}^0_{k-1,K})\bu_h\boldsymbol{\Pi}_{k,K}^0\bv_h \right|\\
+
\frac{1}{2}\left|\sum_{K\in\CT_h}\int_K(\boldsymbol{\beta}\cdot\boldsymbol{\Pi}^0_{k-1,K})\bv_h\boldsymbol{\Pi}_{k,K}^0\bu_h \right|\\
\leq \frac{1}{2}\sum_{K\in\CT_h}\|\boldsymbol{\beta}\|_{\infty,K}\|\boldsymbol{\Pi}_{k-1,K}^0\nabla\bu_h\|_{0,\e}\|\boldsymbol{\Pi}_{k,K}^0\bv_h\|_{0,K}\\
+
\frac{1}{2}\sum_{K\in\CT_h}\|\boldsymbol{\beta}\|_{\infty,K}\|\boldsymbol{\Pi}_{k-1,K}^0\nabla\bv_h\|_{0,K}\|\boldsymbol{\Pi}_{k,K}^0\bu_h\|_{0,K}\\
\leq \|\boldsymbol{\beta}\|_{\infty,\O} \|\bu_h\|_{1,h}\|\bv_h\|_{1,h},
\end{multline}
where   we have used the stability of the projector $\boldsymbol{\Pi}_{k-1,K}^0$ and $\boldsymbol{\Pi}_{k,K}^0$, respectively. Hence, replacing \eqref{eq:bound_sym} and \eqref{eq:bound_skew} in \eqref{eq:bound_ah} we have that
\begin{equation*}
|a_h(\bu_h,\bv_h)|\leq\max\{\nu\max\{\widetilde{c}_1,1\},\|\boldsymbol{\beta}\|_{\infty,\O}C_{\textrm{I}}C_{\textrm{II}}\}\|\bu_h\|_{1,h}\|\bv_h\|_{1,h},
\end{equation*}
which proves the boundedness of $a_h(\cdot,\cdot)$. On the other hand, for $b_h(\cdot,\cdot)$ we have
\begin{multline*}
|b_h(\bv_h,q_h)|=\left|\sum_{K\in\cT_h}\int_{K}q_h\div\bv_h \right|\leq \sum_{K\in\CT_h}\|q_h\|_{0,K}\|\div\bv_h\|_{0,K}\\
\leq \sum_{K\in\CT_h}\|q_h\|_{0,K}\|\nabla\bv_h\|_{0,K}\leq \|q_h\|_{0,\O}\|\bv_h\|_{1,h},
\end{multline*}
proving that $b_h(\cdot,\cdot)$ is also continuous.
\subsection{The discrete eigenvalue problem}
The nonconforming virtual element discretization of the variational formulation \eqref{def:oseen_system_weak} reads as follows.
Find $\lambda_h\in\mathbb{C}$ and $(\boldsymbol{0},0)\neq(\bu_h,p_h)\in \mathcal{X}_h$ such that 
\begin{equation}\label{def:oseen_system_weak_discrete}
	\left\{
	\begin{array}{rcll}
a_h(\bu_h,\bv_h) + b_h(\bv_h,p_h)&=&\lambda_h c_h(\bu_h,\bv_h)&\forall \bv_h\in \BU_h,\\
b_h(\bu_h,q_h)&=&0&\forall q_h\in \Q_h,
\end{array}
	\right.
\end{equation}
where $\mathcal{X}_h:= \BU_h \times \Q_h$.
Thanks to the stability of the bilinear form $a_{h,\text{sym}}^K(\cdot,\cdot)$ and the definition of the bilinear form $a_{h,\text{skew}}^K(\cdot,\cdot)$, it is easy to check that $a_h(\cdot,\cdot)$ is coercive, i.e. 
$$c_*\nu|\bv_h|_{1,h}^2\leq a_h(\bv_h,\bv_h)\qquad \forall\bv_h\in\BU_h.$$
On the other hand, given the discrete spaces $\BU_h$ and $\Q_h$ , satisfy that $\div_h \BU_h \subset\Q_h$, standard arguments (see \cite{MR851383}) guarantee that there exists a positive constant $\beta_0$, independent of $h$, such that
\begin{equation}
\label{ec:inf-sup_disc}
\displaystyle\sup_{\bv_h\in\BU_h}\frac{b_h(\bv_h,q_h)}{|\bv_h|_{1,h}}\geq\beta_0\|q_h\|_{0,\O}\quad\forall q_h\in\Q_h.
\end{equation}
The next step is to introduce the  discrete solution operator $\bT_h: \BL^2(\Omega) \rightarrow \BU_h \subset \BL^2(\Omega)$, defined by $\bT_h\boldsymbol{f}:=\widehat{\bu}_h$, where $\widehat{\bu}_h$ is the  solution of the corresponding discrete source problem:
\begin{equation}\label{def:oseen_system_sourcediscretproblem}
	\left\{
	\begin{array}{rcll}
a_h(\widehat{\bu}_h,\bv_h) + b_h(\bv_h,\widehat{p}_h)&=& c_h(\boldsymbol{f},\bv_h)&\forall \bv_h \in \BU_h,\\
b_h(\widehat{\bu}_h,q_h)&=&0&\forall q_h\in \Q_h.
\end{array}
	\right.
\end{equation}
Since the discrete inf-sup condition is satisfied, the operator $\bT_h$ is well defined. Moreover, we have the following stability result.
$$c_*\nu|\widehat{\bu}_h|_{1,h}\leq \dfrac{1}{C_{1,p}}\|\boldsymbol{f}\|_{0,\O},$$ whereas for the pressure we have
$$\|\widehat{p}_h\|_{0,\O}\leq\dfrac{1}{\beta_0}\left(\dfrac{1}{C_{1,p}}\|\boldsymbol{f}\|_{0,\O}+\nu^{1/2}|\widehat{\bu}_h|_{1,h}\left(\nu^{1/2}\max\{\widetilde{c}_1,1 \}+\dfrac{\|\boldsymbol{\beta}\|_{\infty,\O}}{\nu^{1/2}}\right)\right).$$
where $C_{1,p}$ and $\widetilde{c}_1$  are the constants appearing in\eqref{eq:equivalencia} and \eqref{eq:bound_sym} respectively.
%
%
As in the continuous case, we have the following relation between the discrete spectral problem and its source problem, i.e., $(\lambda_h,(\bu_h,p_h))$ is a solution of Problem \eqref{def:oseen_system_weak_discrete} if and only if $(\kappa_h,\bu_h)$ is an eigenpair of $\bT_h$, i.e., $\bT_h\bu_h=\kappa_h\bu_h$ with $\kappa_h=1/\lambda_h$ and $\lambda_h\neq 0$.
The discrete version of the spectral problem \eqref{eq:eigenA} is written as
\begin{problem}
\label{dis_original}
Find $(\lambda_h, \bu_h, p_h) \in \mathbb{R} \times \BU_h \times \Q_h$ such that $\|\bu_h\|_{0,\Omega}+\|p_h\|_{0,\Omega} > 0$, and 
\begin{equation*}
A_h((\bu_h, p_h), (\bv_h, q_h)) = \lambda_h c_h(\bu_h,\bv_h), \quad \forall (\bv_h, q_h) \in \BU_h\times \Q_h,
\end{equation*}
\end{problem}
where 
$$A_h((\bu_h, p_h), (\bv_h, q_h))=a_h(\bu_h,\bv_h)+b_h(\bv_h,p_h)-b_h(\bu_h,q_h).$$
In Problem~\ref{dis_original}, $A_h(\cdot,\cdot)$, and $c_h(\cdot,\cdot)$ are the virtual element discretization of $A(\cdot,\cdot)$, and $c(\cdot,\cdot)$ respectively, whereas $(\lambda_h, (\bu_h, p_h))$ is the virtual element approximation of the continuous solution  $(\lambda, (\bu, p))$. 
Likewise, we define the discrete formulation corresponding to the adjoint problem, i.e. Eqn.~\eqref{def:oseen_system_weak_dual_eigen}. The identical arguments as for the primal formulation imply the well-posedness of the discrete formulation. 
\begin{remark}
The discrete bilinear form $c_h(\cdot,\cdot)$ is defined neglecting the corresponding stabilizer. We emphasize that we define the solution operator on $\BL^2$ which does not guarantee the existence of the trace on the boundary, and consequently, the edge moments will not be well-defined. This does not guarantee the existence of the associated stabilizer. However, the proposed  definition $c_h(\cdot,\cdot)$ needs only $\BL^2$ regularity and hence suitable for our strategy. 
\end{remark}
As the case continues, it is now necessary to define the adjoint discrete problem, which consists in: Find $\lambda_h^*\in\mathbb{C}$ and $(\boldsymbol{0},0)\neq(\bu_h^*,p_h^*)\in \mathcal{X}_h$ such that 
\begin{equation}\label{def:oseen_system_weak_discretedual}
	\left\{
	\begin{array}{rcll}
a_h(\bv_h,\bu_h^*) - b_h(\bv_h,p_h^*)&=&\overline{\lambda_h^*} c_h(\bv_h,\bu_h^*)&\forall \bv_h\in \BU_h,\\
-b_h(\bu_h^*,q_h)&=&0&\forall q_h\in \Q_h,
\end{array}
	\right.
\end{equation}
Now we define the discrete version of the operator $\bT^*$ is then given by $\bT_h^*: \BL^2(\Omega) \rightarrow \BU_h \subset \BL^2(\Omega)$, defined by $\bT_h^*\boldsymbol{f}:=\widehat{\bu}_h^*$, where $\widehat{\bu}_h^*$ is the  solution of the corresponding discrete source problem:
\begin{equation}\label{def:oseen_system_sourcediscretproblemdual}
	\left\{
	\begin{array}{rcll}
a_h(\bv_h,\widehat{\bu}_h^*) - b_h(\bv_h,\widehat{p}_h^*)&=& c_h(\bv_h,\boldsymbol{f})&\forall \bv_h \in \BU_h,\\
-b_h(\widehat{\bu}_h,q_h)&=&0&\forall q_h\in \Q_h.
\end{array}
	\right.
\end{equation}
\section{A priori error estimates for the source problem}
\label{apriori:error}
We are now in a position to be able to show that $\bT_h$ converges to $\bT$ as $h$ becomes zero in the broken norm. This is contained in the following result
\begin{theorem}
Let $\boldsymbol{f}\in\BL^2(\O,\mathbb{C})$ be such that $\widehat{\bu}:=\bT\boldsymbol{f}$ and $\widehat{\bu}_h:=\bT_h\boldsymbol{f}$ with $\widehat{\bu} \in \BH^{1+s}(\O,\mathbb{C})$, $s\geq 1$. Then, there exists a positive constant $C$, independent of $h$, such that
\begin{equation*}
\|(\bT-\bT_h)\boldsymbol{f}\|_{1,h}=\|\widehat{\bu}-\widehat{\bu}_h\|_{1,h} \leq C h^{\min\{k,s\}} \Big(|\widehat{\bu}|_{1+s,\Omega}+|\widehat{p}|_{s,\Omega}+ \|\boldsymbol{f}\|_{s-1,\Omega} \Big).
\end{equation*}
\end{theorem}

\begin{proof}
By employing the interpolation operator on the discrete space, i.e., $\mathcal{I}$, we split the difference $\widehat{\bu}-\widehat{\bu}_h=\widehat{\bu}-\mathcal{I} \widehat{\bu} +\mathcal{I} \widehat{\bu}-\widehat{\bu}_h$. An application of the approximation properties of the interpolation operator yields the bound of $\boldsymbol{\eta}_h=\widehat{\bu}-\mathcal{I} \widehat{\bu}$. To estimate the other term, i.e., $\boldsymbol{\delta}_h:=\mathcal{I} \widehat{\bu}-\widehat{\bu}_h$, we apply the coercivity  and the fact that $\div(\mathcal{I} \widehat{\bu}-\widehat{\bu}_h)=0$ (see \cite{MR4032861}) in order to obtain
\begin{multline}
\label{main_H1}
C_{1,p}c_*\nu~\|\boldsymbol{\delta}_h\|_{1,h}^{2} \leq a_h(\mathcal{I} \widehat{\bu},\boldsymbol{\delta}_h )-a_h(\widehat{\bu}_h,\boldsymbol{\delta}_h )\\
 =a_h(\mathcal{I} \widehat{\bu},\boldsymbol{\delta}_h )-c_h(\boldsymbol{f},\boldsymbol{\delta}_h)+b_h(\boldsymbol{\delta}_h,\widehat{p}_h)-b_h(\widehat{\bu}_h,q_h)
= a_h(\mathcal{I} \widehat{\bu},\boldsymbol{\delta}_h )-c_h(\boldsymbol{f},\boldsymbol{\delta}_h)\\
=a_{h,\text{skew}}(\mathcal{I} \widehat{\bu},\boldsymbol{\delta}_h)+a_{h,\text{sym}}(\mathcal{I} \widehat{\bu}-\bu_{\pi},\boldsymbol{\delta}_h)+a_{\text{sym}}(\widehat{\bu}_{\pi}-\widehat{\bu},\boldsymbol{\delta}_h)+a_{\text{sym}}(\widehat{\bu},\boldsymbol{\delta}_h)-c_h(\boldsymbol{f},\boldsymbol{\delta}_h)\\
=\underbrace{a_{h,\text{skew}}(\mathcal{I} \widehat{\bu},\boldsymbol{\delta}_h)-a_{\text{skew}}(\widehat{\bu},\boldsymbol{\delta}_h)}_{A_1}+\underbrace{a_{h,\text{sym}}(\mathcal{I} \widehat{\bu}-\bu_{\pi},\boldsymbol{\delta}_h)+a_{\text{sym}}(\widehat{\bu}_{\pi}-\bu,\boldsymbol{\delta}_h)}_{A_2}\\
+\underbrace{c(\boldsymbol{f},\boldsymbol{\delta}_h)-c_h(\boldsymbol{f},\boldsymbol{\delta}_h)}_{A_3}+\underbrace{\calN_h((\widehat{\bu},\widehat{p}), \boldsymbol{\delta}_h)}_{A_4},
\end{multline}
where the consistency term $\calN_h(\cdot,\cdot)$ is given by 
\begin{multline}
\label{Consistency:Err}
\calN_h((\widehat{\bu},\widehat{p}),\boldsymbol{\delta}_h):=\sum_{ K \in \CT_h} a_{\text{skew}}^K(\widehat{\bu}, \boldsymbol{\delta}_h)+a_{\text{sym}}^K(\widehat{\bu}, \boldsymbol{\delta}_h)+b_h(\boldsymbol{\delta}_h,\widehat{p})-c(\boldsymbol{f}, \boldsymbol{\delta}_h) \\
= \sum_{ e \in \CE_{\text{int}}} \int_e \left(\nabla \widehat{\bu}-\frac{1}{2}  ( \widehat{\bu} \otimes \BBeta) -\widehat{p} \bold{I}\right) \bn_e \cdot \jump{\boldsymbol{\delta}_h},
\end{multline} 
where $\bold{I}$ is identity matrix of size $2 \times 2$. This term is obtained using similar argument of those in \cite{MR3507277,MR3959470}.
By using the approximation properties of the interpolation operator and polynomial representative, we bound  each term in  \eqref{main_H1}. $\calN_h((\widehat{\bu},\widehat{p}), \boldsymbol{\delta}_h)$ is the consistency error appeared due to non-conforming approximation of the discrete space. In order to estimate the Term $A_1$, first we note that
\begin{equation}
A_1=\sum_{K\in\CT_h}\left(\underbrace{a_{\text{skew}}^K(\mathcal{I} \widehat{\bu}-\widehat{\bu},\boldsymbol{\delta}_h)}_{B_1}+\underbrace{a_{h,\text{skew}}^K(\mathcal{I} \widehat{\bu},\boldsymbol{\delta}_h)-a_{\text{skew}}^K(\mathcal{I} \widehat{\bu},\boldsymbol{\delta}_h)}_{B_2}\right).
\end{equation}
Now, we estimate term $B_1$ as follows:
\begin{multline*}
a_{\text{skew}}^K(\mathcal{I} \widehat{\bu}-\widehat{\bu},\boldsymbol{\delta}_h)\leq C\left(\|\boldsymbol{\beta}\|_{\infty,K}\|\mathcal{I} \widehat{\bu}-\widehat{\bu}\|_{1,K}\|\boldsymbol{\delta}_h\|_{1,K}\right)\\\leq Ch_K^{\min\{s,k\}}\|\boldsymbol{\beta}\|_{\infty,K}|\bu|_{1+s,D(K)}\|\boldsymbol{\delta}_h\|_{1,K}.
\end{multline*}
 To estimate $B_2$,  is necessary to note that for each $K\in\CT_h$, $\bu,\bv\in\BH^1(K)$ and $\boldsymbol{\beta}\in  \BL^{\infty}(K)$, we have:
$$\int_K(\boldsymbol{\beta}\cdot\nabla)\bu\cdot\bv=\int_K(\nabla\bu)\boldsymbol{\beta}\cdot\bv=\int_K\nabla\bu:(\boldsymbol{\beta}\otimes\bv)^\texttt{t}.$$
 For each polygon $K\in\CT_h$, employing the orthogonality property of the $\BL^2$ projection operator, we obtain
\begin{equation*}
\begin{split}
&\int_K \Big ( (\BPIvec (\nabla \mathcal{I} \widehat{\bu}))\BBeta  \cdot \BPIZ \boldsymbol{\delta}_h-    (\nabla \mathcal{I} \widehat{\bu}) \BBeta\cdot  \boldsymbol{\delta}_h \Big )\\
&=\int_K \Big ( \left(\BPIvec (\nabla \mathcal{I} \widehat{\bu})-\nabla \mathcal{I} \widehat{\bu}\right): (\BBeta\otimes(\BPIZ \boldsymbol{\delta}_h-\boldsymbol{\delta}_h))^\texttt{t}\\
& \quad +\int_K  \left(\BPIvec (\nabla \mathcal{I} \widehat{\bu})-\nabla \mathcal{I} \widehat{\bu}\right): \left( (\BBeta\otimes\boldsymbol{\delta}_h)^\texttt{t}-\BPIvec ((\BBeta\otimes\boldsymbol{\delta}_h)^\texttt{t})\right)\\
& \quad +\int_K\left((\nabla \mathcal{I} \widehat{\bu}) \BBeta-\BPIZ ((\nabla \mathcal{I} \widehat{\bu}) \BBeta)\right)\cdot  (\BPIZ \boldsymbol{\delta}_h-\boldsymbol{\delta}_h).
\end{split}
\end{equation*}
Now assuming that $\nabla \widehat{\bu} \in \bold{H}^{s}(K)$, $  \BBeta \in \bold{W}^{1,\infty}(K)$ and $\boldsymbol{\delta}_h \in \bold{H}^1(K)$ and approximation properties of $\BPIZ $, continuity of $\BL^2$ inner product, it follows that :
\begin{equation}
\label{skew_1}
\begin{split}
&\int_K \Big ( (\BPIvec (\nabla \mathcal{I} \widehat{\bu}))\BBeta  \cdot \BPIZ \boldsymbol{\delta}_h-    (\nabla \mathcal{I} \widehat{\bu}) \BBeta\cdot  \boldsymbol{\delta}_h \Big )\\
& \leq  \|\BPIvec (\nabla \mathcal{I} \widehat{\bu})-\nabla \mathcal{I} \widehat{\bu}\|_{0,K}\|\BBeta\otimes(\BPIZ \boldsymbol{\delta}_h-\boldsymbol{\delta}_h)^\texttt{t}\|_{0,K} \\
& \quad  +\|\BPIvec (\nabla \mathcal{I} \widehat{\bu})-\nabla \mathcal{I} \widehat{\bu}\|_{0,K}\|(\BBeta\otimes\boldsymbol{\delta}_h)^\texttt{t}-\BPIvec ((\BBeta\otimes\boldsymbol{\delta}_h)^\texttt{t})\|_{0,K} \\
& \quad  +\|(\nabla \mathcal{I} \widehat{\bu}) \BBeta-\BPIZ ((\nabla \mathcal{I} \widehat{\bu}) \BBeta)\|_{0,K}\|\BPIZ \boldsymbol{\delta}_h-\boldsymbol{\delta}_h\|_{0,K} \\
& \leq Ch_K^{\min \{s,k \}}\|\BBeta\|_{\bold{W}^{1,\infty}(K)}|\widehat{\bu}|_{1+s,K}|\boldsymbol{\delta}_h|_{1,K}.
\end{split}
\end{equation}
Borrowing the analogous arguments as previous estimate, we obtain: 
\begin{equation}
\label{skew_2}
\int_K \Big ( \BPIvec (\nabla \boldsymbol{\delta}_h)\BBeta  \cdot \BPIZ \mathcal{I} \widehat{\bu}-  (\nabla \boldsymbol{\delta}_h)\BBeta  \cdot  \mathcal{I} \widehat{\bu} \Big) 
\leq C(\BBeta) h_K^{\min\{s,k\}} |\widehat{\bu}|_{1+s,K}|\boldsymbol{\delta}_h|_{1,K}.
\end{equation}
Thus, from the two estimates above (\eqref{skew_1}, \eqref{skew_2}), it is obtained that 
$$B_2\leq Ch_K^{\min\{s,k\}}\|\BBeta\|_{\bold{W}^{1,\infty}(K)}|\widehat{\bu}|_{1+s,K}|\boldsymbol{\delta}_h|_{1,K},$$
and finally considering the sum over all elements $K$
\begin{equation}
\label{Estimate:A1}
A_1\leq Ch^{\min\{s,k\}}\|\BBeta\|_{\bold{W}^{1,\infty}(\Omega)}|\widehat{\bu}|_{1+s,\Omega}|\boldsymbol{\delta}_h|_{1,h}.
\end{equation}
Now our task is to estimate the term $A_2$. To do this task, we begin with the first part of this term by using the approximation properties of the interpolation operator  and polynomial representative (cf. Lemma~\ref{Int_poly_err}) in the following way
\begin{equation} 
\label{Err:dis}
 \begin{split}
\sum_{K \in \CT_h} a_{h,\text{sym}}^K(\mathcal{I} \widehat{\bu}-\widehat{\bu}_{\pi},\boldsymbol{\delta}_h)& \leq \sum_{K \in \CT_h} a_{h,\text{sym}}^K(\mathcal{I} \widehat{\bu}-\widehat{\bu},\boldsymbol{\delta}_h)+\sum_{K \in \CT_h} a_{h,\text{sym}}^K(\widehat{\bu}-\widehat{\bu}_{\pi},\boldsymbol{\delta}_h) \\
& \leq C h^{\min\{s,k\}} |\widehat{\bu}|_{1+s,\Omega} |\boldsymbol{\delta}_h|_{1,h}. 
\end{split}
\end{equation}

 Now for the second part of $A_2$,  we invoke the polynomial approximation property given  in Lemma~\ref{Int_poly_err} in order to obtain
\begin{equation}
\label{Err:Con}
\sum_{K \in \CT_h} a^K_{h,\text{sym}}(\widehat{\bu}_{\pi}-\widehat{\bu}, \boldsymbol{\delta}_h) \leq C h^{\min\{s,k\}} |\widehat{\bu}|_{1+s,\Omega} |\boldsymbol{\delta}_h|_{1,h}.
\end{equation}
Hence, gathering \eqref{Err:dis} and \eqref{Err:Con} we have $A_2\leq Ch^{\min\{s,k\}} |\widehat{\bu}|_{1+s,\Omega} |\boldsymbol{\delta}_h|_{1,h}$.
%
%
To bound $A_3$, we use the approximation properties of the projection operator $\BL^2$ and, following  the arguments of \cite{MR4032861} we obtain
\begin{equation}
\label{Err:load}
A_3=c(\boldsymbol{f},\boldsymbol{\delta}_h)-c_h(\boldsymbol{f},\boldsymbol{\delta}_h) \leq C h^{\min\{s,k\}} |\boldsymbol{f}|_{s-1,\Omega} |\boldsymbol{\delta}_h|_{1,h}.
\end{equation}
%
%
%
Now, we focus to bound the consistency error $\calN_h(\cdot,\cdot)$. 
 For a better representation of the analysis, we define
$	\boldsymbol{\Bgama}:= \nabla \widehat{\bu}- \frac{1}{2}  (\BBeta \otimes \widehat{\bu})^T -\widehat{p} \bold{I}$. Consequently, we rewrite \eqref{Consistency:Err} as
\begin{equation*}
\calN_h((\widehat{\bu},\widehat{p}),\boldsymbol{\delta}_h)= \sum_{ e \in \CE_{\text{int}}} \int_e \Bgama \bn_e \cdot \jump{\boldsymbol{\delta}_h}. 
\end{equation*}
Further, following the definition of $\BU_h$ \eqref{NCVEM:space}, we have
\begin{equation}
\int_e \jump{\boldsymbol{\delta}_h}_e \cdot \bq_{k-1}=0 \,\ \forall  \bq_{k-1} \in \BPL_{k-1}(e),  \ \ \forall e \in  \CE_{\text{int}}.
\end{equation}
 By employing orthogonality of the polynomial projection operator, we rewrite the term as follows
\begin{multline*}
\calN_h((\widehat{\bu},\widehat{p}),\boldsymbol{\delta}_h)  = \sum_{ e \in \CE_{\text{int}}} \int_e (\Bgama-\mathcal{R}_h^K\Bgama) \bn_e \cdot \jump{\boldsymbol{\delta}_h-\boldsymbol{\mathcal{P}}_0^e \boldsymbol{\delta}_h} \\
 \leq \sum_{ e \in \CE_{\text{int}}} \| \Bgama-\mathcal{R}_h^K \Bgama \|_{e,0}  ~\| \boldsymbol{\delta}_h-\boldsymbol{\mathcal{P}}_0^e \boldsymbol{\delta}_h\|_{e,0},
\end{multline*}
where $\boldsymbol{\mathcal{P}}_0^e \boldsymbol{\delta}_h$ is the average value of $\boldsymbol{\delta}_h$  on edge $e$ and $\mathcal{R}_h|_K=\mathcal{R}_h^K$ \eqref{eq:estimateL2}. By using trace inequality and approximation properties of the $\BL^2$ projection operator, we derive as 
\begin{equation}
\label{jump:1}
\| \Bgama-\mathcal{R}_h^K \Bgama \|_{e,0} \leq C h^{\min\{s,k\}-\frac{1}{2}} |\Bgama|_{s,K} \quad  e \subset \partial K .
\end{equation}
By using the approximation property of the projection operator $\boldsymbol{\mathcal{P}}_0^e$, we derive the bound as follows: \cite{ciarlet2002finite}
\begin{equation}
\label{jump:2}
\|\jump{\boldsymbol{\delta}_h} \|_{0,e} \leq C h^{1/2} (|\boldsymbol{\delta}_h|_{1,K^+}+|\boldsymbol{\delta}_h|_{1,K^-}) \quad e = \partial K^+ \cap \partial K^-.
\end{equation}
By employing inequalities $\eqref{jump:1}$, and $\eqref{jump:2}$, we bound the consistency error as follows
\begin{equation}
\label{Vcrime}
\calN_h((\widehat{\bu},\widehat{p}),\boldsymbol{\delta}_h) \leq C h^{\min\{s,k\}} \left( |\widehat{\bu}|_{1+s,\Omega}+|\widehat{p}|_{s,\Omega}\right) |\boldsymbol{\delta}_h|_{1,h}.
\end{equation}
Upon inserting estimates \eqref{Estimate:A1},\eqref{Err:dis}, \eqref{Err:Con}, \eqref{Err:load}, and \eqref{Vcrime} into \eqref{main_H1}, we obtain the bound 
\begin{equation}
\label{Err:H1}
\|\widehat{\bu}-\widehat{\bu}_h\|_{1,h} \leq C h^{\min\{s,k\}} \left(|\widehat{\bu}|_{1+s,\Omega}+|\widehat{p}|_{s,\Omega} +|\boldsymbol{f}|_{s-1,\Omega} \right).
\end{equation}
Further following the analogous arguments as \cite[Theorem~13]{MR4032861}, and the bound of polynomial consistency error for the convective term, we derive the estimate for pressure variable, i.e.,
\begin{equation}
\label{Err:P}
\|\widehat{p}-\widehat{p}_h\|_{0,\Omega} \leq C h^{\min\{s,k\}} \left(|\widehat{\bu}|_{1+s,\Omega}+|\widehat{p}|_{s,\Omega} +|\boldsymbol{f}|_{s-1,\Omega} \right).
\end{equation} 
Upon using \eqref{Err:H1} and \eqref{Err:P} we obtain  the desire result.

\subsection{$\BL^2$ Error estimates for the velocity}
In this part, we would like to bound the error in $\BL^2$ norm. To achieve the goal, let $\O$ be an open, bounded, and convex domain. We first define the dual problem as follows: Find $(\Bpsi, \xi) \in \mathcal{X}$ such that
\begin{align}
- \nu \Delta \Bpsi-\div (\Bpsi \otimes \BBeta)-\nabla \xi & =(\widehat{\bu}-\widehat{\bu}_h)  \quad \text{in} \  \Omega,  \label{L2:1}\\
 \text{div} \Bpsi & =0 \quad \text{in} \ \Omega,  \label{L2:2}\\
 (\xi,1) & = 0 \quad \text{in} \ \Omega,  \label{L2:3}\\
 \Bpsi& =0 \quad \text{on} \ \partial \Omega. \label{L2:4}
\end{align} 
The model problem \eqref{L2:1}-\eqref{L2:4} is well posed. By applying the classical regularity theorem, we derive that
\begin{equation}
\label{classic_reg}
\| \Bpsi \|_{2, \Omega}+ \| \xi \|_{1, \Omega} \leq \|\widehat{\bu}-\widehat{\bu}_h\|_{0,\Omega}.
\end{equation} 
By multiplying $\bv_h=\widehat{\bu}-\widehat{\bu}_h$ in \eqref{L2:1}, we derive that
\begin{equation}
\label{weak:L2}
\int_{\Omega} \Big ( - \nu \Delta \Bpsi-\div(\Bpsi \otimes \BBeta)-\nabla \xi \Big) (\widehat{\bu}-\widehat{\bu}_h)=\| \widehat{\bu}-\widehat{\bu}_h\|_{0,\Omega}^2.
\end{equation}
Now, since  $\div\BBeta=0$, by employing integration by parts, we rewrite \eqref{weak:L2} as follows
\begin{multline}
\label{L2:main:2}
\| \widehat{\bu}-\widehat{\bu}_h \|_{0,\Omega}^2= \widehat{a}(\widehat{\bu}-\widehat{\bu}_h,\Bpsi)-b_h(\widehat{\bu}-\widehat{\bu}_h,\xi)\\
+ \sum_{ e \in \CE} \int_e \left(-\nabla \Bpsi-\frac{1}{2} (\Bpsi \otimes \BBeta) -\xi \bold{I}\right) \bn_e \cdot \jump{\widehat{\bu}-\widehat{\bu}_h}.
\end{multline} 
By employing the arguments as \eqref{Vcrime}, and classical regularity result \eqref{classic_reg}, we bound the following term as follows
\begin{multline}
\label{L2:main:3}
\sum_{ e \in \CE} \int_e \left(-\nabla \Bpsi-\frac{1}{2} (\Bpsi \otimes \BBeta) - \xi\bold{I}\right) \bn_e \cdot \jump{\widehat{\bu}-\widehat{\bu}_h}  \leq C h (|\Bpsi|_{2,\Omega}+|\xi|_{1,\Omega}) \|\widehat{\bu}-\widehat{\bu}_h\|_{1,h} \\
 \leq C h^{\min\{s,k\}+1} (|\widehat{\bu}|_{1+s,\Omega}+|\widehat{p}|_{s,\Omega}+|\boldsymbol{f}|_{s-1,\Omega}) \| \widehat{\bu}-\widehat{\bu}_h\|_{0,\Omega}.
\end{multline}
Further, using  the fact that $ b_h(\widehat{\bu}-\widehat{\bu}_h,\mathcal{R}_h\xi)=0$, we rewrite the terms as follows
\begin{multline}
\label{L2:interior}
a(\widehat{\bu}-\widehat{\bu}_h,\Bpsi)-b_h(\widehat{\bu}-\widetilde{\bu}_h,\xi)\\
=a(\widehat{\bu}-\widehat{\bu}_h,\Bpsi-\mathcal{I}\Bpsi)-b_h(\widehat{\bu}-\widehat{\bu}_h,\xi-\mathcal{R}_h\xi)+a(\widehat{\bu}-\widehat{\bu}_h,\mathcal{I}\Bpsi). 
\end{multline}
By employing the estimate~\eqref{Err:H1}, approximation properties of the interpolation operator, and regularity result (Eqn~\eqref{classic_reg}) we find
\begin{multline}
\label{L2:main:4}
 a(\widehat{\bu}-\widehat{\bu}_h,\Bpsi-\mathcal{I}\Bpsi)-b_h(\widehat{\bu}-\widehat{\bu}_h,\xi-\mathcal{R}_h\xi) \\
 \leq  C\| \BBeta \|_{\infty,\Omega} \| \widehat{\bu}-\widehat{\bu}_h \|_{1,h} \|\Bpsi-\mathcal{I}\Bpsi \|_{1,h}+  \| \widehat{\bu}-\widehat{\bu}_h \|_{1,h} \| \xi-\mathcal{R}_h\xi \|_{0,\Omega} \\
\leq  C\| \BBeta \|_{\infty,\Omega}  h^{\min\{s,k\}}  (|\widehat{\bu}|_{1+s,\O}+|\widehat{p}|_{s,\O}+|\boldsymbol{f}|_{s-1,\O})  (|\Bpsi|_{2,\Omega}+\| \xi \|_{1,\Omega}) h \\
  \leq C\| \BBeta \|_{\infty,\Omega} h^{\min\{s,k\}+1}  (|\widehat{\bu}|_{1+s,\O}+|\widehat{p}|_{s,\O}+|\boldsymbol{f}|_{s-1,\O}) \| \widehat{\bu}-\widehat{\bu}_h\|_{0,\Omega}.
\end{multline}
Further, with the estimate $b(\Bpsi,\widehat{p}-\widehat{p}_h)=0$, we rewrite the last term of \eqref{L2:interior} as follows
\begin{equation}
\label{L2:main}
\begin{split}
a(\widehat{\bu}-\widehat{\bu}_h,\mathcal{I}\Bpsi) &=a(\widehat{\bu},\mathcal{I}\Bpsi)-a(\widehat{\bu}_h,\mathcal{I}\Bpsi) \\
& = \Big ( a(\widehat{\bu},\mathcal{I}\Bpsi)+b(\mathcal{I}\Bpsi ,\widehat{p})-c(\boldsymbol{f},\mathcal{I}\Bpsi ) \Big) + \Big(a_h(\widehat{\bu}_h, \mathcal{I}\Bpsi)-a(\widehat{\bu}_h, \mathcal{I}\Bpsi) \Big) \\
&+ \Big ( c(\boldsymbol{f},\mathcal{I}\Bpsi )- c_h(\boldsymbol{f},\mathcal{I}\Bpsi )\Big)+b(\Bpsi-\mathcal{I}\Bpsi, \widehat{p}-\widehat{p}_h).
\end{split}
\end{equation}
Since $\mathcal{I}\Bpsi \in \BU_h$, the term $a(\widehat{\bu},\mathcal{I}\Bpsi)+b(\mathcal{I}\Bpsi ,\widehat{p})-c(\boldsymbol{f},\mathcal{I}\Bpsi )$ measures the inconsistency due to non-conforming property of the discrete space.
By using analogous arguments as \eqref{Consistency:Err}, we bound the term 
\begin{equation}
\label{L2:Variational_Crime}
a(\widehat{\bu},\mathcal{I}\Bpsi)+b(\mathcal{I}\Bpsi ,\widehat{p})-c(\boldsymbol{f},\mathcal{I}\Bpsi ) \leq C h^{\min\{s,k\}+1} ( |\widehat{\bu}|_{1+s,\O} + |\widehat{p}|_{s,\Omega})\|\widehat{\bu}-\widehat{\bu}_h \|_{0,\Omega}.
\end{equation}
Upon employing the boundedness of the $\BL^2$ projection operator, result \eqref{classic_reg}, we bound the discrete load term as follows
\begin{equation}
\label{L2: load}
c(\boldsymbol{f},\mathcal{I}\Bpsi )- c_h(\boldsymbol{f},\mathcal{I}\Bpsi ) \leq C h^{\min\{s,k\}+1} |\boldsymbol{f}|_{s-1,\O} \|\widehat{\bu}-\widehat{\bu}_h\|_{0,\Omega}.
\end{equation}
Using the approximation properties of the interpolation operator and estimate~\eqref{Err:H1}, we derive that
\begin{equation}
\label{L2: pressure}
b( \Bpsi-\mathcal{I}\Bpsi,\widehat{p}-\widehat{p}_h)\leq C h^{\min\{s,k\}+1} |\widehat{\bu}|_{1+s,\O} \|\widehat{\bu}-\widehat{\bu}_h\|_{0,\Omega}.
\end{equation} 
Now, we focus to bound the term $(a_h(\widehat{\bu}_h, \mathcal{I}\Bpsi)-a(\widehat{\bu}_h, \mathcal{I}\Bpsi))$ as follows
\begin{equation}
\begin{split}
\label{Diff_Discrete}
a_h(\widehat{\bu}_h, \mathcal{I}\Bpsi)-&a(\widehat{\bu}_h, \mathcal{I}\Bpsi) =\sum_{K \in \CT_h } \Big [ a_h^K(\widehat{\bu}_h-\BPIZ \widehat{\bu}, \mathcal{I}\Bpsi- \BPIO \Bpsi  ) \\
& -a^K(\widehat{\bu}_h-\BPIZ \widehat{\bu}, \mathcal{I}\Bpsi- \BPIO \Bpsi ) 
+a_h^K(\BPIZ \widehat{\bu}, \mathcal{I}\Bpsi)-a^K(\BPIZ \widehat{\bu}, \mathcal{I}\Bpsi) \\
&
+ a_h^K(\widehat{\bu}_h, \BPIO \Bpsi)-a^K(\widehat{\bu}_h, \BPIO \Bpsi) \Big ].
\end{split}
\end{equation}
By using the approximation properties of the projection operator and interpolation operator, we bound the term as follows
\begin{equation}
\label{L2:diff:con:dis}
\sum_{K \in \CT_h }a_h^K(\widehat{\bu}_h, \mathcal{I}\Bpsi)-a^K(\widehat{\bu}_h, \mathcal{I}\Bpsi) \leq C h^{\min\{s,k\}+1} |\widehat{\bu}|_{1+s,\Omega} \|\widehat{\bu}-\widehat{\bu}_h\|_{0,\Omega}.
\end{equation}
By inserting the estimates \eqref{L2:Variational_Crime}, \eqref{L2: load},\eqref{L2: pressure}, \eqref{Diff_Discrete}, \eqref{L2:diff:con:dis} into \eqref{L2:main}, we obtain 
\begin{equation}
\label{part:L2}
 a(\widehat{\bu}-\widehat{\bu}_h, \mathcal{I} \psi) \leq C h^{\min\{s,k\}+1} |\widehat{\bu}|_{1+s,\Omega} \|\widehat{\bu}-\widehat{\bu}_h\|_{0,\Omega}.
\end{equation}
Using the estimates \eqref{L2:main:3}, \eqref{L2:main:4}, and \eqref{part:L2} into \eqref{L2:main:2}, we derive 
\begin{equation*}
\|\widehat{\bu}-\widehat{\bu}_h\|_{0,\Omega} \leq C h^{\min\{s,k\}+1} (|\widehat{\bu}|_{1+s,\O}+|\widehat{p}|_{s,\O}+|\boldsymbol{f}|_{s-1,\O}).
\end{equation*}
\end{proof}
We have the following consequence
\begin{lemma}
\label{convergence:T}
There exists a constant $C>0$ independent of mesh size $h$ such that 
\begin{equation*}
\|(\bT-\bT_h)\boldsymbol{f}\|_{0,\Omega} \leq C h^{\min\{s,k\}+1} (|\widehat{\bu}|_{1+s,\O}+|\widehat{p}|_{s,\O}+|\boldsymbol{f}|_{s-1,\O}).
\end{equation*}
\end{lemma}
The above statement is state forward due to previous result, and Theorem~\ref{th:regularidadfuente}. The next results establish the convergence of the operator $\bT^{*}_h$ to $\bT^{*}$ as $h$ goes to zero in broken norm and in the $\BL^2$ norm. The proof can be obtained repeating the same arguments as those used in the previous section.
\begin{theorem}
\label{Convergence_dual}
Let $\boldsymbol{f}\in\BL^2(\O,\mathbb{C}) \cap \BH^{s^{*}-1}(\Omega)$ be such that $\widehat{\bu}^{*}:=\bT^{*}\boldsymbol{f}$ and $\widehat{\bu}_h^{*}:=\bT_h^{*}\boldsymbol{f}$. Then, there exists a positive constant $C$, independent of $h$, such that
\begin{equation*}
\|(\bT^{*}-\bT_h^{*})\boldsymbol{f}\|_{1,h}=\|\widehat{\bu}^{*}-\widehat{\bu}_h^{*}\|_{1,h} \leq C h^{\min\{k,s^{*}\}} \Big(|\widehat{\bu}^*|_{1+s^{*},\Omega}+|\widehat{p}^{*}|_{s^*,\Omega}+ \|\boldsymbol{f}\|_{s^{*}-1,\Omega} \Big).
\end{equation*}
\begin{equation*}
\|(\bT^{*} -\bT_h^{\ast})\boldsymbol{f}\|_{0,\Omega}=\|\widehat{\bu}^{\ast}-\widehat{\bu}_h^{\ast}\|_{0,\Omega} \leq C h^{\min\{k,s^*\}+1} \Big(|\widehat{\bu}^*|_{1+s^*,\Omega}+|\widehat{p}^*|_{s^*,\Omega}+ \|\boldsymbol{f}\|_{s^*-1,\Omega} \Big).
\end{equation*}
where $C$ is a positive constant independent of $h$.
\end{theorem}

As a  consequence of the previous results is that, according to the  theory of  \cite{MR0203473}, we are in a position   to conclude that  our numerical method does not introduce spurious eigenvalues. This is stated in the following theorem.
\begin{theorem}
	\label{thm:spurious_free}
	Let $V\subset\mathbb{C}$ be an open set containing $\sp(\bT)$. Then, there exists $h_0>0$ such that $\sp(\bT_h)\subset V$ for all $h<h_0$.
\end{theorem}
\section{Spectral approximation and error estimates:}
\label{spectral:approx}
We will obtain convergence and error estimates for the suggested nonconforming VEM discretization for the Oseen eigenvalue problem in this section. More precisely, we shall prove that $\bT_h$ gives a valid spectral approximation of $\bT$ by using the classical theory for compact operators (see \cite{MR1115235}). The equivalent adjoint operators $\bT_h^{\ast}$ and $\bT^{\ast}$ of $\bT_h$ and $\bT$, respectively, will then have a comparable convergence result established.  

Let us recall some preliminary definitions and notations (see \cite[page 683]{MR1115235} for further details). If $\mathbf{J}:X\rightarrow X$ is a compact operator and $X$ is a Banach space, 
the spectrum of $\mathbf{J}$ is defined by $\sp(\mathbf{J}):=\mathbb{C}\setminus\rho(\mathbf{J})$, where $\rho(\mathbf{J})$ corresponds to the resolvent of $\mathbf{J}$. For any $z\in\rho(\mathbf{J})$ we define the resolvent operator associated to $z$ as $\mathbf{R}_z(\mathbf{J}):=(z-\mathbf{J})^{-1}$.

 Let $\mu$ be a non-zero eigenvalue of $\bT$ with algebraic multiplicity $m$. $C$ sets a circle with a centre at $\mu$ in the complex plane, ensuring that no other eigenvalue is contained inside $C$. Furthermore, think about the spectral projections $E$ and $E^{\ast}$ in the manner described below:
\begin{equation*}
\boldsymbol{E}:= (2 \pi i)^{-1} \int_{C} (z-\bT)^{-1} dz \qquad \boldsymbol{E}^{\ast}:= (2 \pi i)^{-1} \int_{C} (z-\bT^{\ast})^{-1} dz,
\end{equation*}
where $\boldsymbol{E}$ and $\boldsymbol{E}^{\ast}$ are projections onto the space of generalized eigenvectors, i.e.,  $R(\boldsymbol{E}):=\ker(\mu-\bT)^\alpha$ and $R(\boldsymbol{E}^{\ast}):=\ker(\mu-\bT^*)^{\alpha}$, respectively, where $R$ denotes de range and  $\alpha$ is the ascent of $\mu-\bT$ and $\mu-\bT^{\ast}$.  Now, it is easy to prove that $R(\boldsymbol{E}),R(\boldsymbol{E}^{\ast}) \in \BH^{r+1}(\O) \times \BH^{r}(\O)$, and  $R(\boldsymbol{E}^{\ast}) \in \BH^{r^*+1}(\O) \times \BH^{r^*}(\O)$ (see Theorem~\ref{th:regularidadfuente} and \ref{th:regularidadfuente_dual}). Next, since $\bT_h$ converges to $\bT$,
it means that there exist $m$ eigenvalues (which lie in $C$) $\mu(1),\ldots,\mu(m)$ of $\bT_h$ (repeated according to their
respective multiplicities) which will converge to $\mu$ as $h$ goes to zero. In the same sense, we introduce the following spectral projector $\boldsymbol{E}_h := (2 \pi i)^{-1} \int_{C} (z-\bT_h)^{-1}dz$, which is a projector onto the invariant subspace $R(\boldsymbol{E}_h)$ of $\bT_h$ spanned by the generalized eigenvectors of $\bT_h$ corresponding to $\mu(1), \ldots, \mu(m)$. We also recall the definition of gap $\widehat{\delta}$ between the closed subspaces  $\boldsymbol{\mathcal{X}}$, and $\boldsymbol{\mathcal{Y}}$ of $\BL^{2}$.
\begin{equation*}
\widehat{\delta}(\boldsymbol{\mathcal{X}},\boldsymbol{\mathcal{Y}}):= \max \{ \delta(\boldsymbol{\mathcal{X}},\boldsymbol{\mathcal{Y}}), \delta(\boldsymbol{\mathcal{Y}}, \boldsymbol{\mathcal{X}}) \},
\end{equation*}
where
\begin{equation*}
\delta(\boldsymbol{\mathcal{X}},\boldsymbol{\mathcal{Y}})= \underset{\bold{x} \in \boldsymbol{\mathcal{X}}; \| \bold{x} \|_{\BL^2}=1 }{\sup} \delta(\bold{x},\boldsymbol{\mathcal{Y}}), \quad \text{with} \; \delta(\boldsymbol{x},\boldsymbol{\mathcal{Y}})= \underset{\bold{y} \in \boldsymbol{\mathcal{Y}}; \| \bold{y} \|_{\BL^2}=1 }{\inf} \|\bold{x}-\bold{y} \|_{\BL^2}.
\end{equation*}
The following error estimates for the approximation of eigenvalues and eigenfunctions hold true.
\begin{theorem}
There exists a strictly positive constant $C$ such that 
\begin{align}
&\widehat{\delta}(R(\boldsymbol{E}),R(\boldsymbol{E}_h)) \leq C h^{\min \{r,k \} +1}; \label{Spectral:1}\\
& |\mu-\widehat{\mu}_h|\leq C h^{\min \{r,k \}+\min \{r^*,k \}} \label{Spectral:2},
\end{align}
where $\widehat{\mu}_h:= \frac{1}{m} \sum_{j=1}^m \mu_h^{j}$, where $r \geq 1$, and $r^* \geq 1$ are the orders of regularity of the eigenfunctions of primal and dual problems.
\end{theorem}
\begin{proof}
The estimate~\eqref{Spectral:1} follows from \cite[Theorem~7.1]{MR1115235}, and the fact that $\|\bT_h-\bT\|_{0,\O} \approx O(h^{\min \{r,k\} +1})$ (Lemma~\ref{convergence:T}). In what follows we will prove \eqref{Spectral:2}: assume that $\bT(\bu_j)=\mu \bu_j$, for $j=1,2, \ldots, m$. Let $(\bu_j^{\ast})$ be a  dual basis for $R(E^{\ast})$ satisfying
\begin{equation}
\langle \bu_j,\bu_l^{\ast} \rangle:= A((\bu_j,p),(\bu_l^{\ast},p^{\ast}))=\delta_{jl},
\end{equation}
where $\langle \cdot,\cdot \rangle$ denotes the corresponding duality pairing. Now, from \cite[Theorem~7.2]{MR1115235}, we have that
\begin{equation}
\label{Spectral_both}
|\mu-\widehat{\mu}_h| \leq \frac{1}{m} \sum_{k=1}^m \Big|\langle (\bT-\bT_h)\bu_k,\bu_k^{\ast} \rangle \Big|+ \| (\bT-\bT_h)|_{R(E)} \|_{0,\Omega} \| (\bT^{\ast}-\bT_h^{\ast})|_{R(E)} \|_{0,\Omega},
\end{equation}
where $\langle\cdot,\cdot\rangle$ denotes the corresponding duality pairing. The estimate of the second term of \eqref{Spectral_both} is quite obvious. In this direction, we need bound of $\| (\bT-\bT_h) \|_{0,\Omega}$, and $\| (\bT^{\ast}-\bT_h^{\ast}) \|_{0,\Omega}$ which are achieved from Lemma~\ref{convergence:T}, and Theorem~\ref{Convergence_dual}. However, the estimate of $\langle (\bT-\bT_h)\bu_k,\bu_k^{\ast} \rangle$ is not straightforward, and it needs arguments same as \cite{MR4664360}.
\begin{equation}
\label{main_Spectral}
\begin{split}
\langle(\bT-\bT_h) \bu_k, \bu_k^{\ast} \rangle & =A ((\bT-\bT_h) \bu_k, p_k-p_{k,h}); (\bu_k^{\ast}, p_k^{\ast} )) \\
&= A ((\bT-\bT_h) \bu_k, p_k-p_{k,h}); (\bu_k^{\ast}, p_k^{\ast} )-(\bv_h,\eta_h))  \\
& \quad  + A ((\bT \bu_k, p_k); (\bv_h,\eta_h))-A((\bT_h \bu_k, p_{k,h}); (\bv_h,\eta_h)) \\
& = A ((\bT-\bT_h) \bu_k, p_k-p_{k,h}); (\bu_k^{\ast}, p_k^{\ast} )-(\bv_h,\eta_h)) + c(\bu_k,\bv_h) \\
& +\calN_h((\bT \bu_k,p_k),\bv_h)-A((\bT_h u_k,p_{k,h}),(\bv_h,\eta_h))\\
& + A_h((\bT_h u_k,p_{k,h}),(\bv_h,\eta_h))-c_h(\bu_k,\bv_h).
\end{split}
\end{equation}  
In the above estimate, the consistency error $\calN_h(\cdot,\cdot)$ appears since $\BU_h \not \subset \BH^1(\Omega)$. Now, we proceed to bound the terms appeared in \eqref{main_Spectral}. In \eqref{main_Spectral}, we have mentioned that $(\bv_h,\eta_h) \in \BU_h \times \Q_h$ is any discrete function. However, to achieve optimal rate of convergence of the spectrum, choose $(\bv_h,\eta_h):= (\mathcal{I} \bu_k^{\ast}, \mathcal{R}_h p_k^{\ast})$. Upon employing, the approximation properties of the interpolation operator, we bound the term as follows:
\begin{multline}
A ((\bT-\bT_h) \bu_k, p_k-p_{k,h}), (\bu_k^{\ast}, p_k^{\ast} )-(\CI \bu_k^{\ast}, \mathcal{R}_h p_k^{\ast})) \\
\leq C  \| (\bT-\bT_h) \bu_k \|_{1,h} \| \bu_k^{\ast}-\mathcal{I} \bu_k^{\ast}\|_{1,h} 
+ \|(\bT-\bT_h)\bu_k\|_{1,h} \|p^{\ast}_k-\mathcal{R}_h p_k^{\ast}\|_{0,\Omega} \\
+ \|p_k-p_{k,h}\|_{0,\Omega} \|\bu_k^{\ast}-\mathcal{I} \bu_k^{\ast} \|_{1,h}.
\end{multline}
By employing Lemma~\ref{Int_poly_err}, and spectral convergence of the primal problem, we have 
\begin{multline}
\label{First_term}
A ((\bT-\bT_h) \bu_k, p_k-p_{k,h}), (\bu_k^{\ast}, p_k^{\ast} )-(\CI \bu_k^{\ast}, \mathcal{R}_h p_k^{\ast}))\\
 \leq C h^{\min \{r,k \}+\min \{r^*,k \}} \Big(|\bu_k|_{1+r,\Omega}+|p_k|_{r,\Omega}+|\boldsymbol{f}|_{r-1,\Omega}\Big)  \Big(|\bu^{\ast}_k|_{1+r^*,\Omega}+|p^{\ast}_k|_{r^*,\Omega}\Big).
\end{multline}
By employing the polynomial consistency property of the load term and approximation property of the $\BL^2$ projection operator, we have
\begin{equation}
\label{diff_load}
\begin{split}
c(\bu_k,\bv_h)-c_h(\bu_k,\bv_h)&=\sum_{K \in \CT_h} c^K(\bu_k-\bu_{k,\pi},\CI \bu_k-\bu_{k,\pi}) \\
& \quad +c^K_h(\bu_k-\bu_{k,\pi},\CI \bu_k-\bu_{k,\pi}) \leq C h^{2\min \{r,k \}+2} |\bu_k|_{1+r,\Omega}.
\end{split}
\end{equation}
The difference between continuous and discrete forms can be bounded as follows 
\begin{equation}
\label{conti_dis}
\begin{split}
& A_h((\bT_h \bu_k,p_{k,h}),(\CI \bu_k^{\ast}, \mathcal{R}_h p_k^{\ast}))-A((\bT_h \bu_k,p_{k,h}),(\CI \bu_k^{\ast}, \mathcal{R}_h p_k^{\ast})) \\
& \leq C h^{\min \{r,k \}+\min \{r^*,k \}} \Big (|\bu_k|_{1+r,\Omega}+|p_k|_{r,\Omega}+|\boldsymbol{f}|_{r-1,\Omega} \Big ) \Big(|\bu^{\ast}_k|_{1+r^*,\Omega}+|p^{\ast}_k|_{r^*,\Omega} \Big).
\end{split}
\end{equation}
In the above estimate, we have added and subtracted $\BPIZ \bT_h \bu_k$, and applied the approximation properties of the interpolation operator.  Now, we are in a situation to bound the variational crime associated with the formulation. Recollecting $\calN_h((\mu \bu_k, p_k), \bu_k^{\ast})=0$, we rewrite the term as follows:
\begin{equation}
\begin{split}
\label{VCrime}
\calN_h((\mu \bu_k, p_k),\CI \bu_k^{\ast}) &=\calN_h((\mu \bu_k, p_k), \CI \bu_k^{\ast}-\bu_k^{\ast}) \\
& \leq C h^{\min \{r,k \}} \Big(|\bu_k|_{1+r,\Omega}+|p_k|_{r,\Omega} \Big) 
  \Big ( |\CI \bu_k^{\ast}-\bu_k^{\ast}|_{1,h} \Big) \\
& \leq C h^{\min \{r,k \}+\min \{r^*,k \}} \Big(|\bu_k|_{1+r,\Omega}+|p_k|_{r,\Omega} \Big) |\bu^{\ast}_k|_{1+r^*,\Omega}.
\end{split}
\end{equation}
Upon inserting \eqref{First_term}, \eqref{diff_load}, \eqref{conti_dis}, and \eqref{VCrime} into \eqref{main_Spectral}, we obtain an estimate for the term \mbox{$\langle(\bT-\bT_h) \bu_k, \bu_k^{\ast} \rangle $}, and consequently double order convergence of the spectrum, i.e., \eqref{Spectral_both}. 
\end{proof}

\section{Numerical experiments}
\label{sec:numerical-experiments}
We end our paper reporting some numerical tests to illustrate the performance of our method. The implementation of the method has been developed in a Matlab code for the lowest order case $k=1$. The goal is to assess the performance of the method on different domains and of course, study the presence of spurious eigenvalues.  After computing the eigenvalues,  the rates of convergence  are calculated by using a least-square fitting. More precisely, if $\lambda_h$ is a discrete complex eigenvalue, then the rate of convergence $\alpha$ is calculated by extrapolation with the least square  fitting
\begin{equation}
\label{eq:fitting}
\lambda_{h}\approx \lambda_{\text{extr}} + Ch^{\alpha},
\end{equation}
where $\lambda_{\text{extr}}$ is the extrapolated eigenvalue given by the fitting. 

For the tests we consider the following families of polygonal meshes which satisfy the assumptions  \textbf{A1} and  \textbf{A2}  (see Figure \ref{fig:mesh}):
\begin{itemize}
\item $\CT_h^1$: trapezoidal meshes;
\item $\CT_h^2$: squares meshes;
\item $\CT_h^3$: structured hexagonal meshes made of convex hexagons; 
\item$\CT_h^4$: non-structured Voronoi meshes.
\end{itemize}

From now and on, the refinement parameter $N$, used to label each mesh, represents the number of elements intersecting each edge and its relation with the meshsize is $h=1/N$.
\begin{figure}[h!]
\begin{center}
\begin{minipage}{5.2cm}
\centering\includegraphics[height=5.1cm, width=5.1cm]{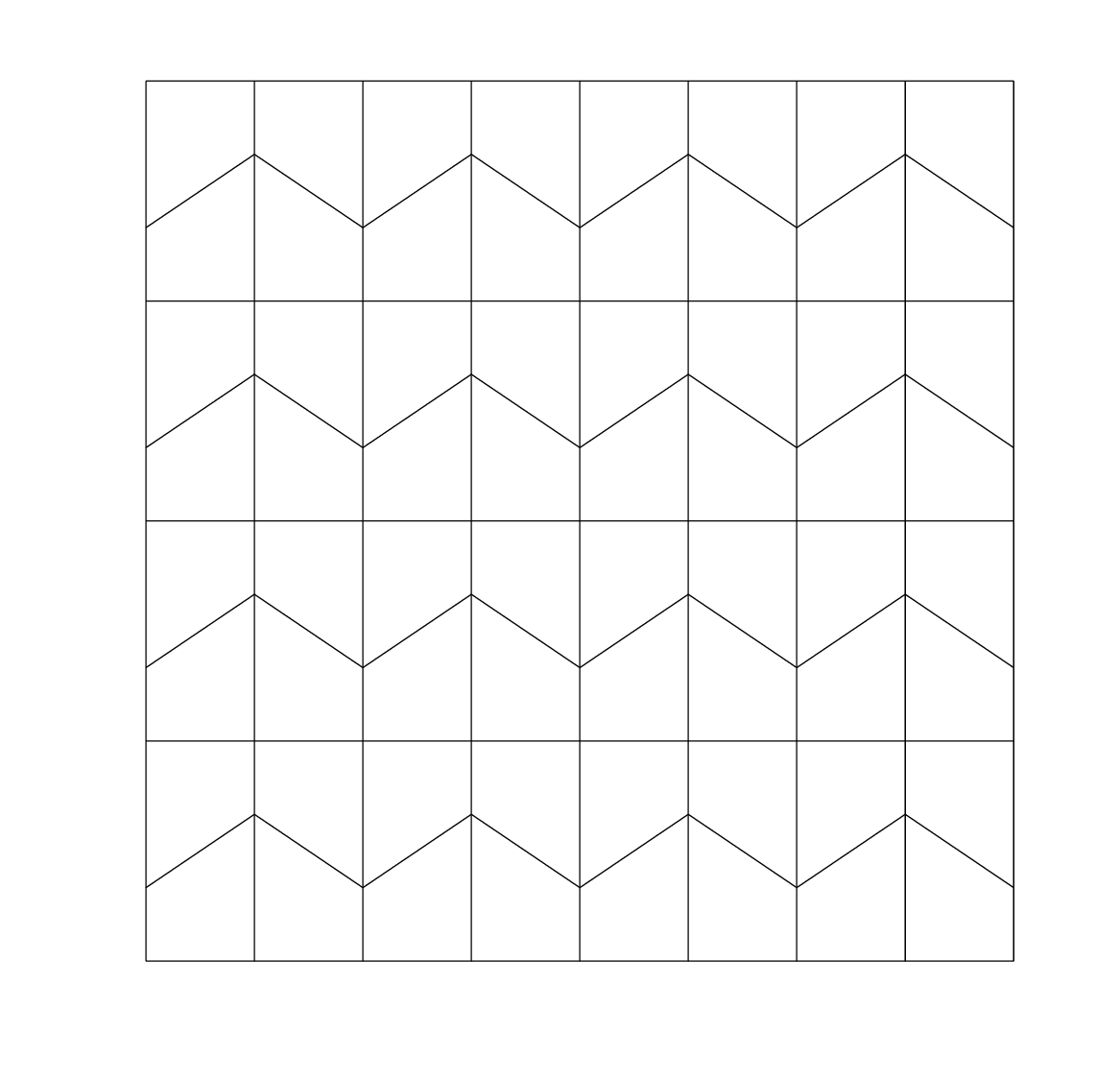}
\end{minipage}
\begin{minipage}{5.2cm}
\centering\includegraphics[height=5.1cm, width=5.1cm]{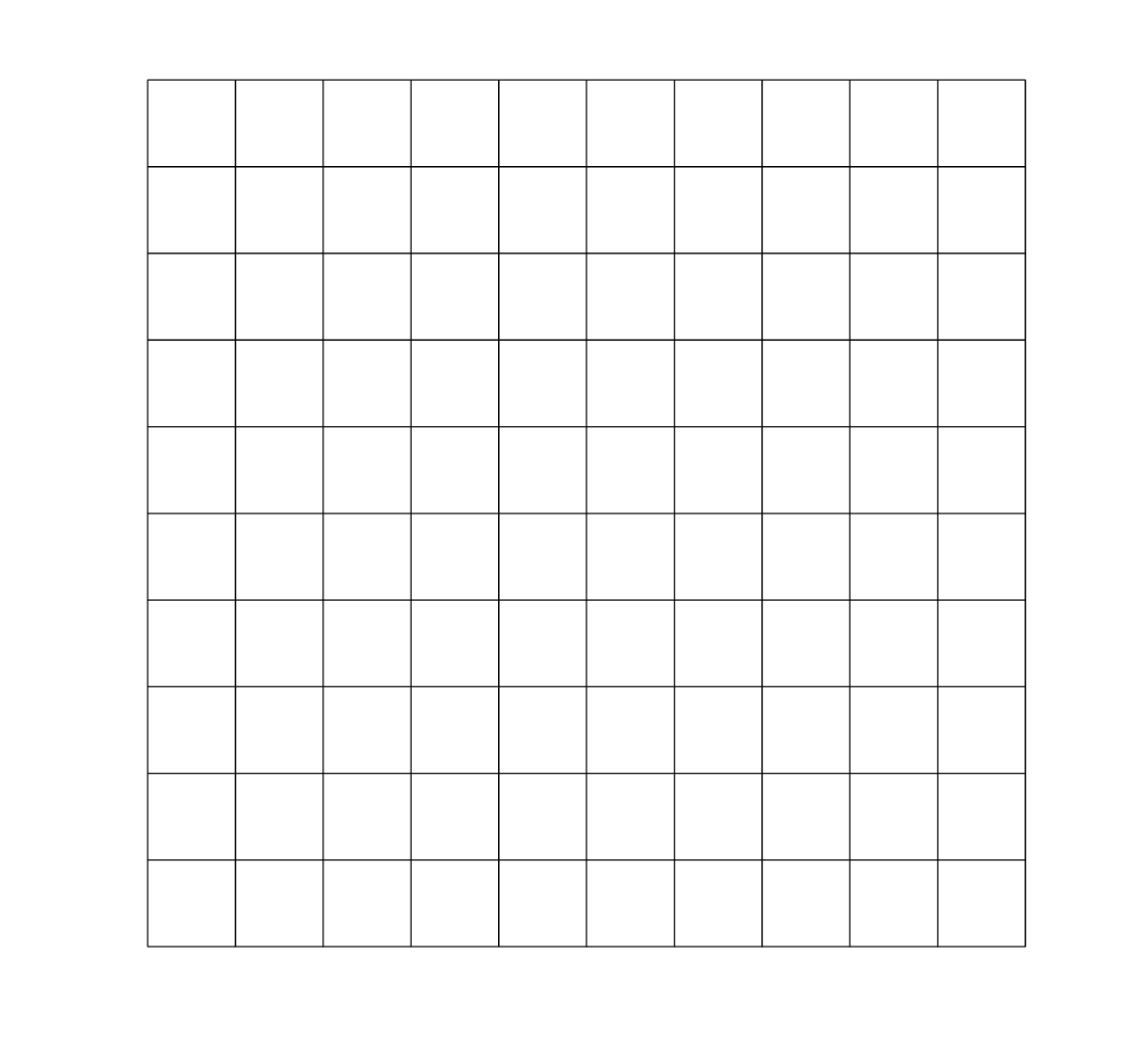}
\end{minipage}\\
\begin{minipage}{5.2cm}
\centering\includegraphics[height=5.1cm, width=5.1cm]{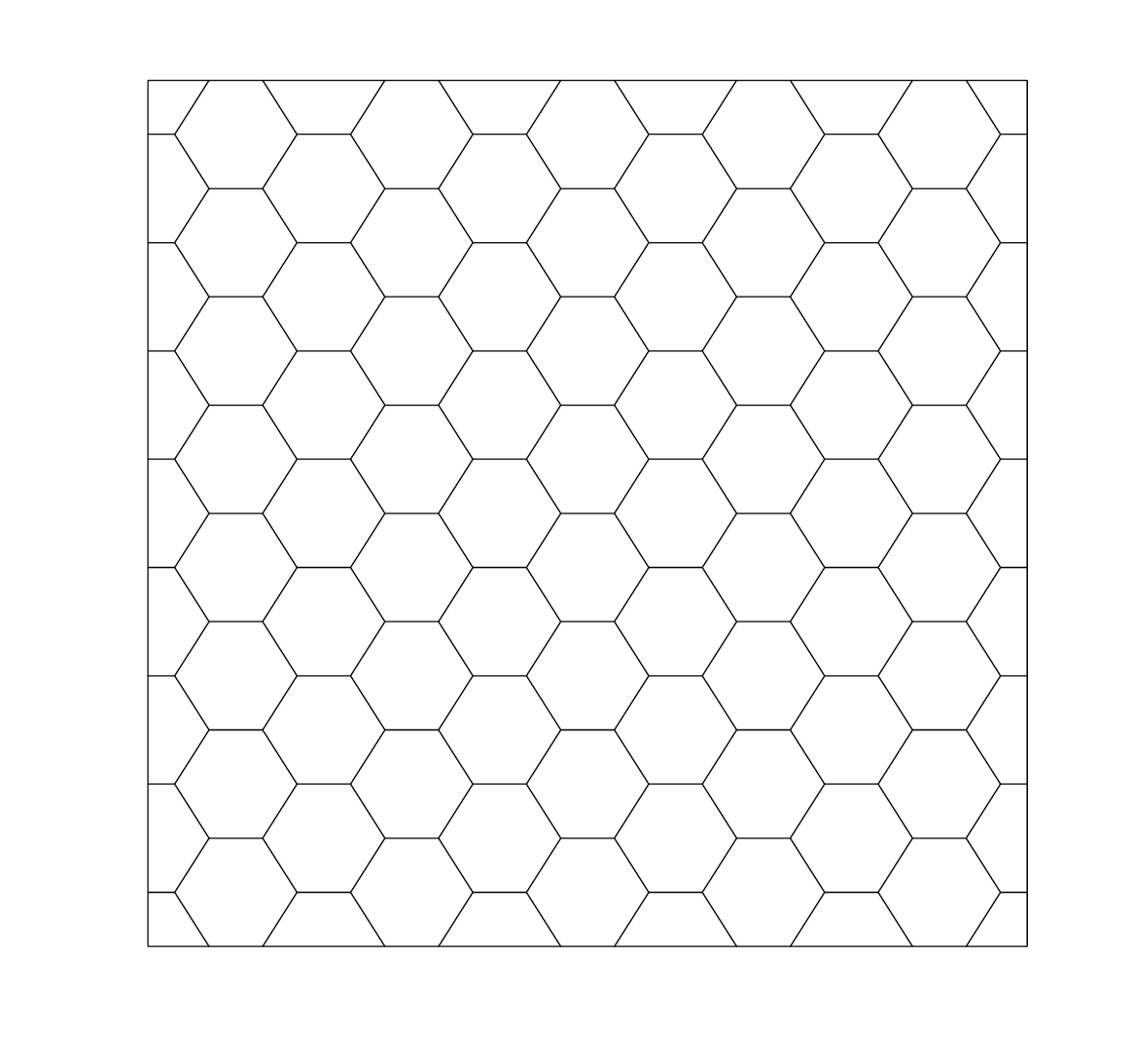}
\end{minipage}
\begin{minipage}{5.2cm}
\centering\includegraphics[height=5.1cm, width=5.1cm]{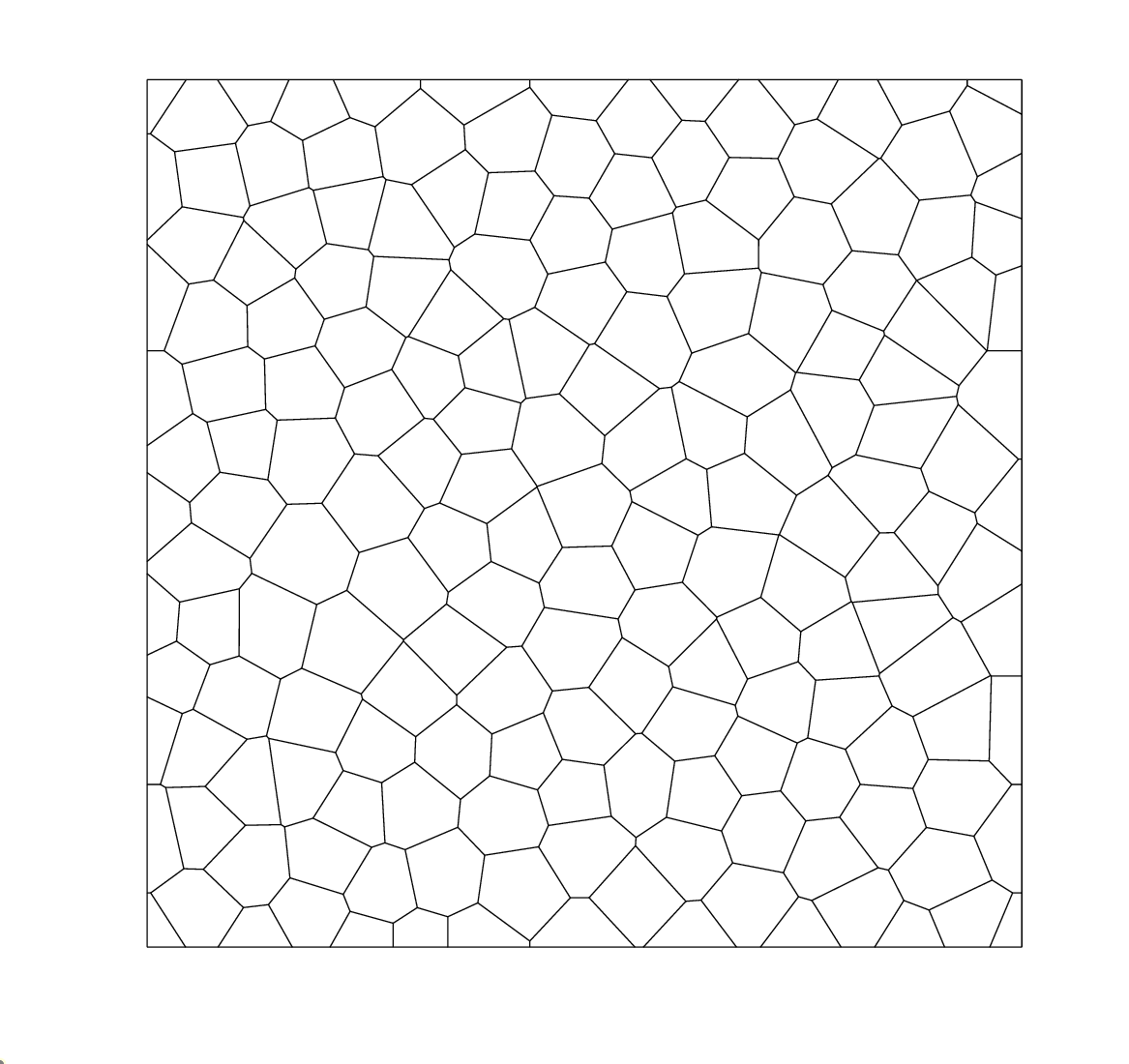}
\end{minipage}
\caption{\label{fig:mesh} Sample meshes: $\CT_h^1$ (top left), $\CT_h^2$ (top right), $\CT_h^3$ (bottom  left), $\CT_h^4$ (bottom right) for $N=8$ and 10.}
\end{center}
\end{figure}
\subsection{Test 1: a square domain}
In this first test, we have taken $\Omega=(-1,1)^2$, $\boldsymbol{\beta}=(1,0)^{\texttt{t}}$. On this type of domain, the eigenfunctions are sufficiently smooth due to the convexity of the square and the null boundary conditions. Hence, an optimal order of convergence is expected with our method. For this test we consider the meshes reported in Figure \ref{fig:mesh}. The results are contained in Table \ref{tabla1} where in the column "Order" we report the computed order of convergence for the eigenvalues, which has been obtained with the  least square fitting \eqref{eq:fitting},
together with extrapolated values that we report on the column  "Extr."
 \begin{table}[h]
\caption{ The lowest computed eigenvalues $\l_{h,i}$, $1\leq i\leq 4$ on different meshes.}
\label{tabla4}
\begin{center}
\begin{tabular}{|c|c|c|c|c|c|c|c|c|} \hline
$\mathcal{T}_h$ &$\l_{h,i}$ & $N = 16$ & $N = 32$& $N = 64$ & $N = 128$  &Order & Extr. &\cite{lepe2023finite}\\ \hline 
 & $\l_{1,h}$ & 13.5455  & 13.5931 &  13.6054 &  13.6085  &  1.95   &13.6097&13.6096\\
$\mathcal{T}_h^1$&$\l_{2,h}$ &   22.9603&   23.0917 &  23.1204  & 23.1274   & 2.17 &  23.1291&23.1297\\
&$\l_{3,h}$  &  23.2729  & 23.3893 &  23.4147 &  23.4209  &  2.17 &  23.4223& 23.4230\\
&$\l_{4,h}$ &   31.7714 &  32.1695  & 32.2658 &  32.2900  &  2.04 &  32.2973&32.2981\\\hline 
&$\l_{h,1}$ &13.5670   &13.5990  & 13.6069  & 13.6089   & 2.00 &  13.6096&13.6096\\
$\mathcal{T}_h^2$&$\l_{h,2}$  & 22.9501  & 23.0917  & 23.1206  & 23.1275 &   2.26&   23.1289&23.1297\\
&$\l_{h,3}$  & 23.2825 &  23.3948 &  23.4163  & 23.4213 &   2.35&   23.4221& 234230\\
&$\l_{h,4}$   &31.8671 &  32.1979 &  32.2735  & 32.2920 &   2.11&   32.2971&32.2981\\\hline
&$\l_{h,1}$ &13.6980  & 13.6318  & 13.6151  & 13.6110   & 1.99 &  13.6095& 13.6096\\
$\mathcal{T}_h^3$&$\l_{h,2}$ &   23.3644  & 23.1976 &  23.1472  & 23.1341   & 1.77   &23.1277& 23.1297\\
&$\l_{h,3}$ &   23.7112 &  23.4960&   23.4411 &  23.4275  &  1.98  & 23.4227&23.4230\\
&$\l_{h,4}$ &   32.8415&   32.4460  & 32.3354 &  32.3074  &  1.86  & 32.2951&32.2981\\\hline
&$\l_{h,1}$ &13.6935  & 13.6276  & 13.6135   & 13.6106 &2.23 &  13.6097&13.6096\\
$\mathcal{T}_h^4$&$\l_{h,2}$ &   23.3782  & 23.1945   &23.1443   &23.1334  &1.92  & 23.1280&23.1297\\
&$\l_{h,3}$ &   23.6837 &  23.4885&   23.4379  & 23.4268 & 1.98&   23.4219&23.4230\\
&$\l_{h,4}$ &   32.7775 &  32.4220  & 32.3255 & 32.3051  & 1.94  & 32.2951&32.2981\\\hline
  \end{tabular}
\end{center}
\label{tabla1}
\end{table}
\begin{figure}[h]
	\begin{center}
\centering\includegraphics[height=5.1cm, width=5.1cm]{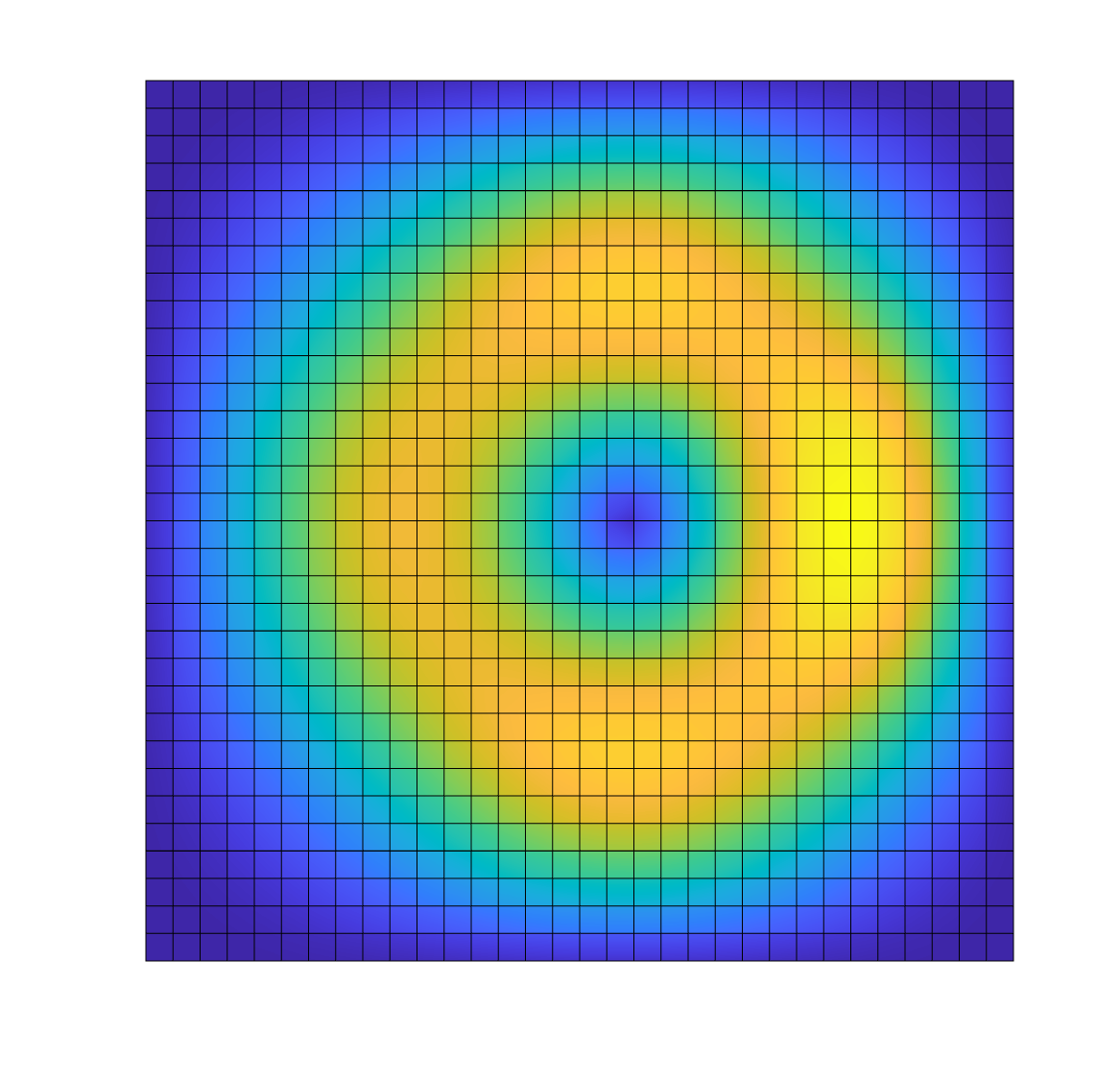}
			\centering\includegraphics[height=5.1cm, width=5.1cm]{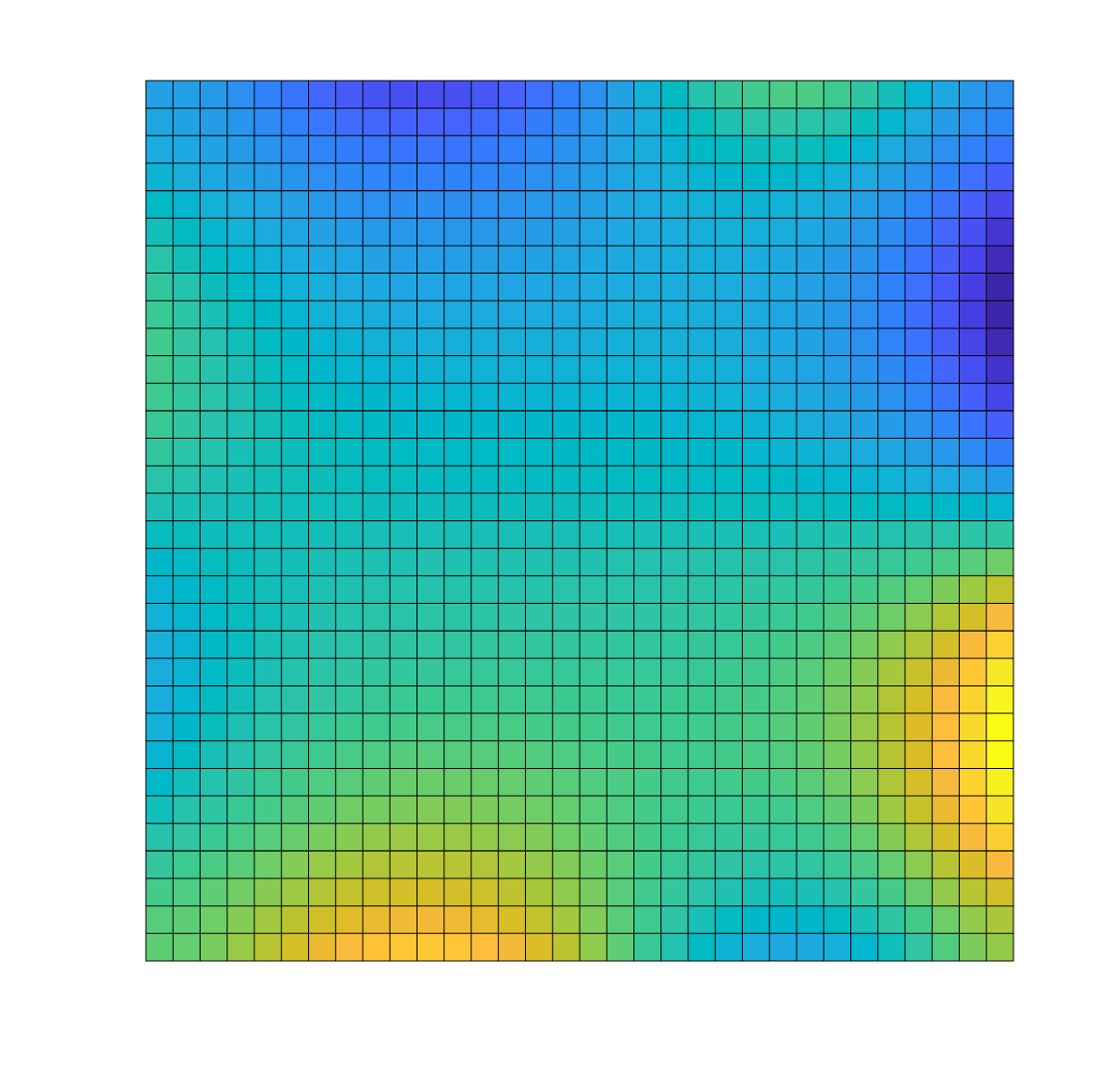}
			\centering\includegraphics[height=5.1cm, width=5.1cm]{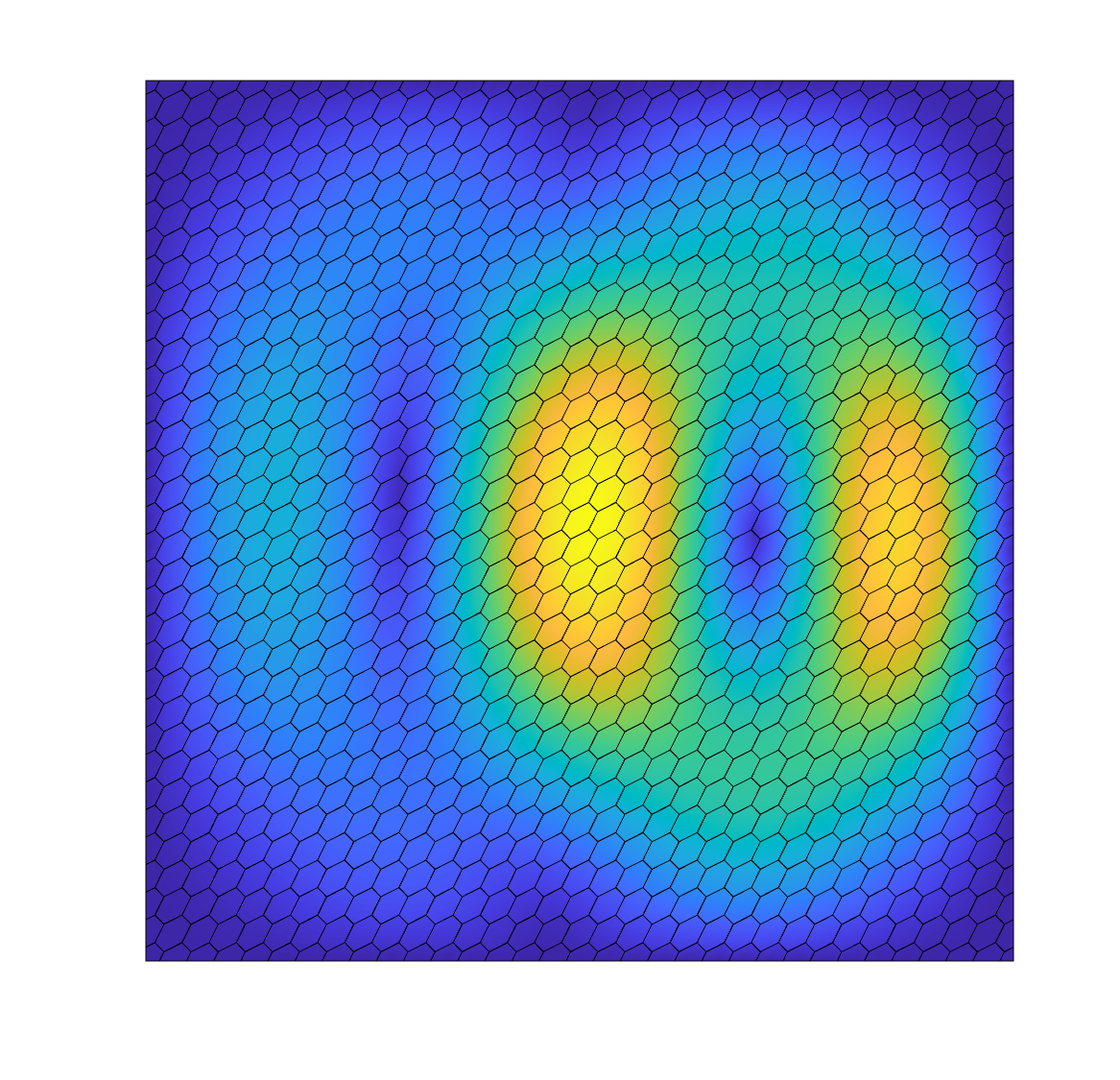}
			\centering\includegraphics[height=5.1cm, width=5.1cm]{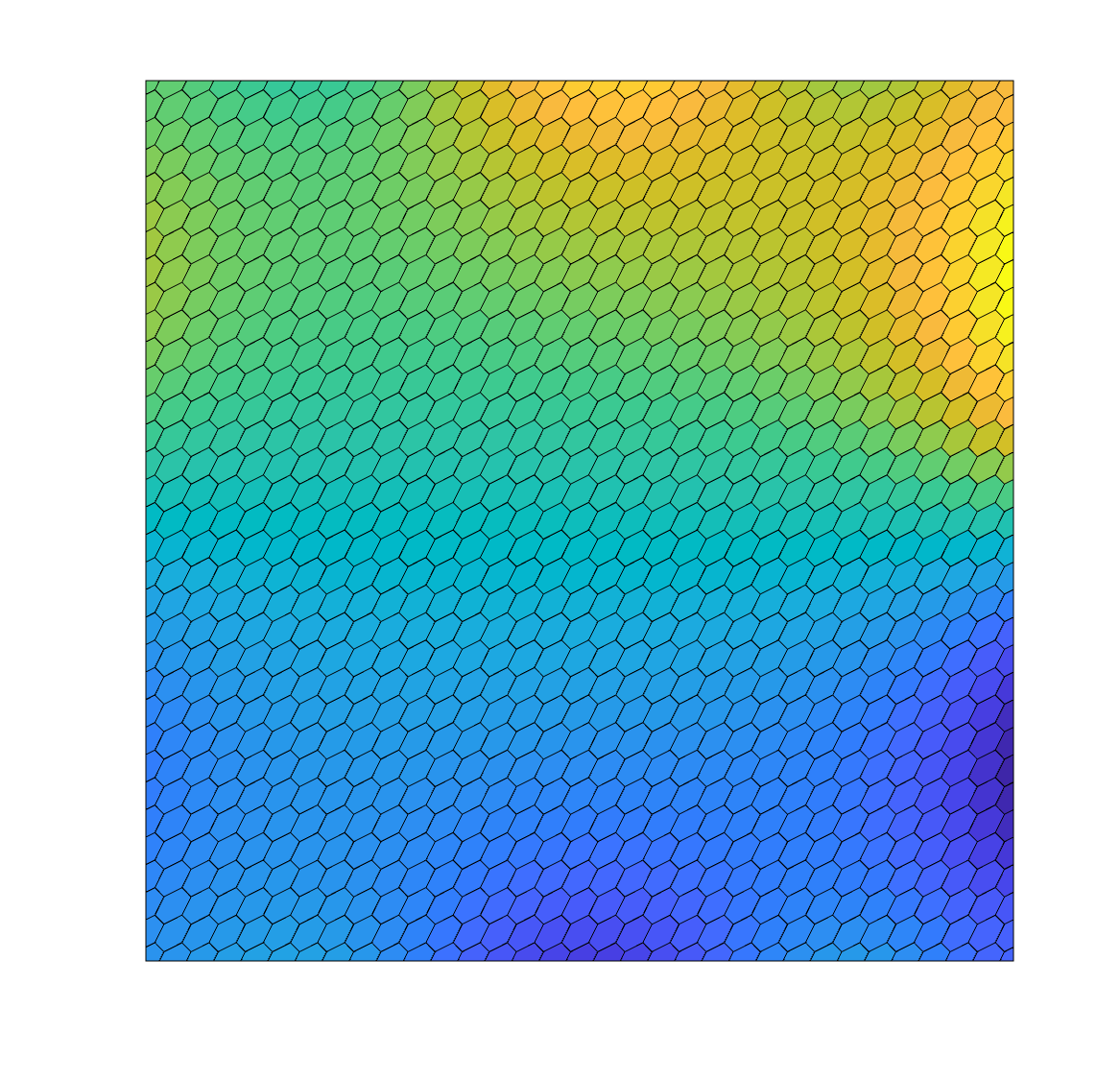}
			\centering\includegraphics[height=5.1cm, width=5.1cm]{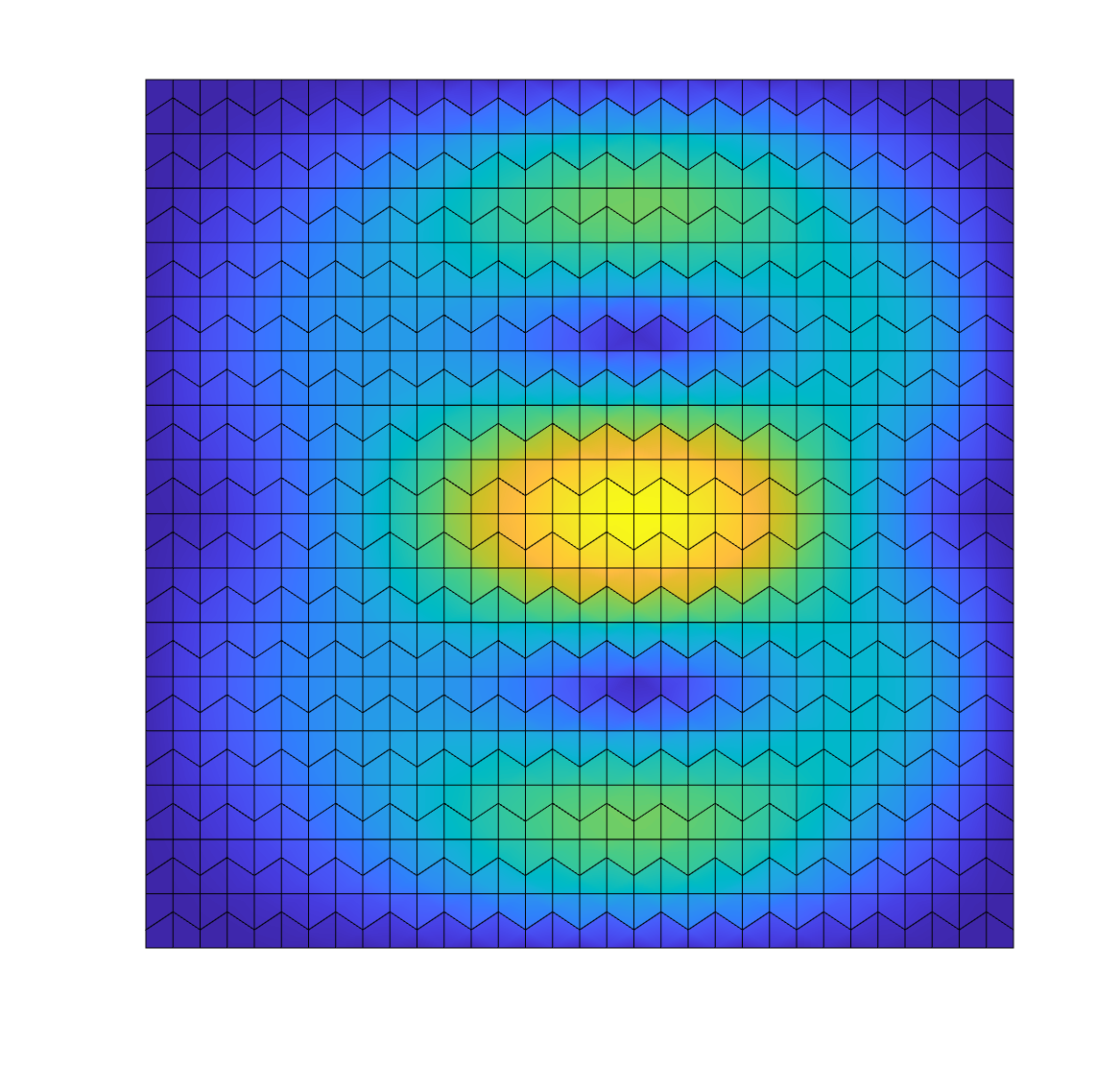}
			\centering\includegraphics[height=5.1cm, width=5.1cm]{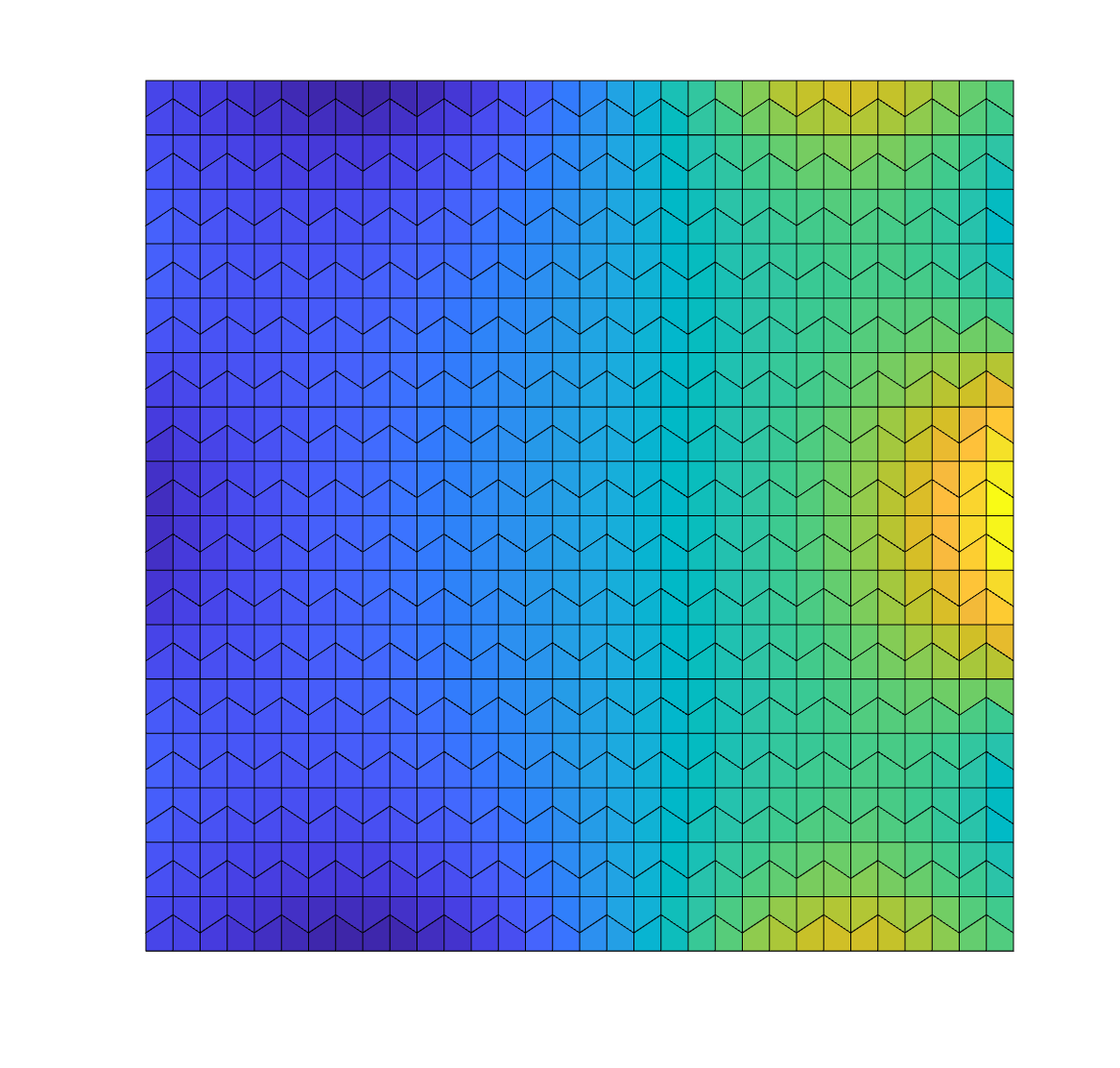}
		\caption{First, second and third magnitude of the eigenfunctions in the square together with the associated pressures: first column $u_{1,h}$, $u_{2,h}$ and $u_{3,h}$ ;second column: $p_{1,h}$, $p_{2,h}$ and $p_{3,h}$; for different family of meshes.}
		\label{fig:p&u_L}
	\end{center}
\end{figure}

\subsection{Test case 2: $\L$ shaped domain} In this example, we consider the non-convex L-shaped domain defined as $\Omega_{L}:=(-1,1) \times (-1,1) \setminus[-1,0] \times [-1,0]$(Figure~\ref{fig:mesh:L}). The eigenfunctions have singularity at $(0,0)$ therefore the convergence order of the corresponding eigenvalues is not optimal. According to the regularity of the eigenfunctions, the rate of convergence  $r$ for the eigenvalues is such that  $1.7\leq r\leq 2$. In Table~\ref{table_Lshp}, we display the results for the model problem. In Figures~\ref{fig:L:shp}, we have dissected the first three discrete velocity and pressure fields. Table~\ref{table_Lshp}'s results demonstrate that the approach provides the anticipated convergence behavior in the eigenvalue approximation. Because of the geometrical singularity of the re-entrant angle, the eigenfunction associated with the first eigenvalue is not sufficiently smooth when compared to the eigenfunctions of the other eigenvalues. The order of convergence for the first computed eigenvalue reflects this fact. 

\begin{figure}[h!]
\begin{center}
\begin{minipage}{5.2cm}
\centering\includegraphics[height=5.1cm, width=5.1cm]{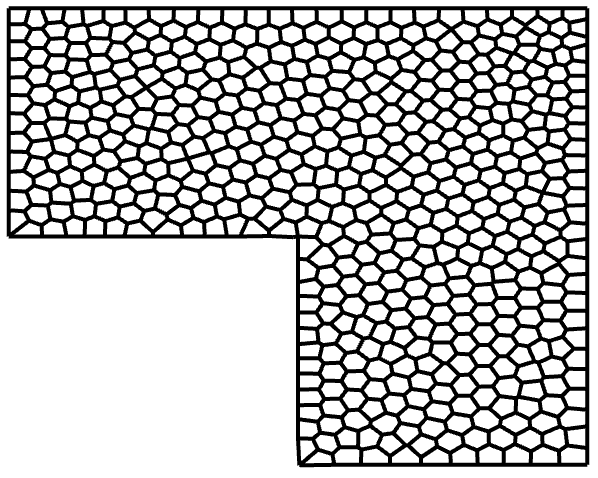}
\end{minipage}
\begin{minipage}{5.2cm}
\centering\includegraphics[height=5.1cm, width=5.1cm]{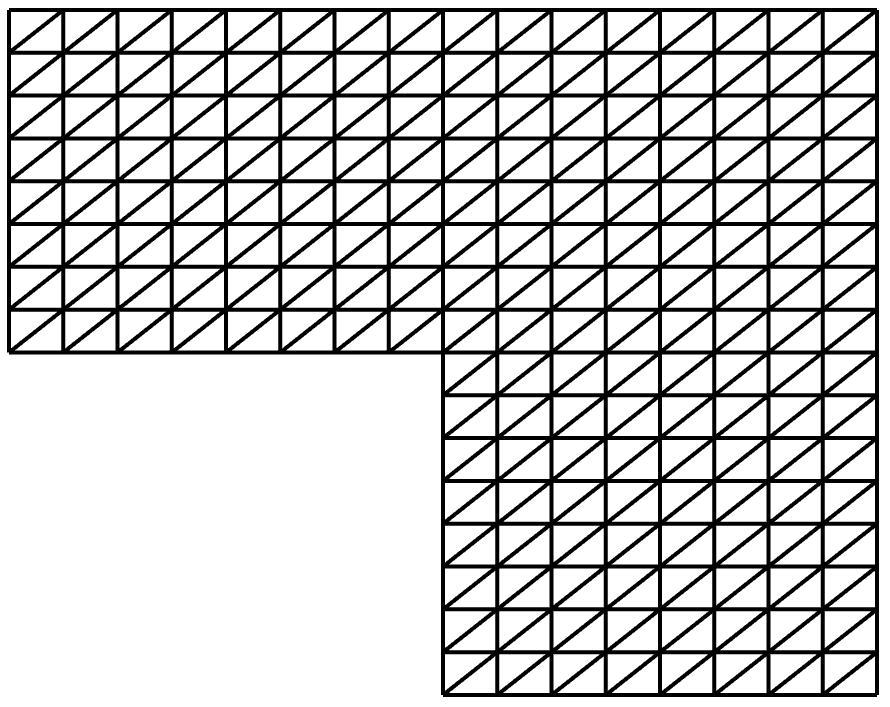}
\end{minipage}
\caption{\label{fig:mesh:L} Sample meshes: $\CT_h^5$ ( left panel), $\CT_h^6$ ( right panel) for $N=8$ }
\end{center}
\end{figure}

\begin{table}[h]
\caption{ The lowest computed eigenvalues $\l_{h,i}$, $1\leq i\leq 4$ on different meshes.}
\label{table_Lshp}
\begin{center}
\begin{tabular}{|c|c|c|c|c|c|c|c|c|} \hline
$\mathcal{T}_h$ &$\l_{h,i}$ & $N = 16$ & $N = 32$& $N = 64$ & $N = 128$  &Order & Extr. &\cite{lepe2023finite}\\ \hline 
 & $\l_{1,h}$ & 31.6764  & 32.5080 &  32.8513 &  32.8855  &  1.65   &32.8949&33.0306\\
$\mathcal{T}_h^5$&$\l_{2,h}$ &   36.6099&   36.9845 &  37.0997  & 37.1058   & 2.02 &  37.1073 &37.1106\\
&$\l_{3,h}$  &  41.8939  & 42.2468 &  42.3768 &  42.3878 &  1.79 &  42.3901& 42.4023\\
&$\l_{4,h}$ &   48.7401 &  49.1200  & 49.2219 &  49.2247  &  2.19 &  49.2264&49.2552\\\hline 
&$\l_{h,1}$ &31.2535  &32.3647  & 32.7931 & 32.8151 & 1.76 &  32.8303&33.0306\\
$\mathcal{T}_h^6$&$\l_{h,2}$  & 36.1669  & 36.8918 & 37.0938  & 37.1058 &   2.13&   37.1066&37.1106\\
&$\l_{h,3}$  & 41.8756 &  42.2558 &  42.3880 & 42.3978 &   1.86&   42.4000& 42.4023\\
&$\l_{h,4}$   &49.4014 &  49.2980 &  49.2609 & 49.2577 &   1.82&   49.2572&49.2552\\\hline
  \end{tabular}
\end{center}
\end{table}

\begin{figure}[h]
	\begin{center}
\centering\includegraphics[height=5.1cm, width=5.1cm]{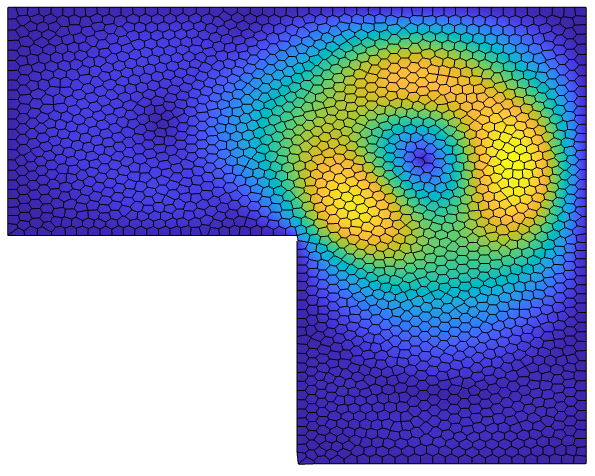}
			\centering\includegraphics[height=5.1cm, width=5.1cm]{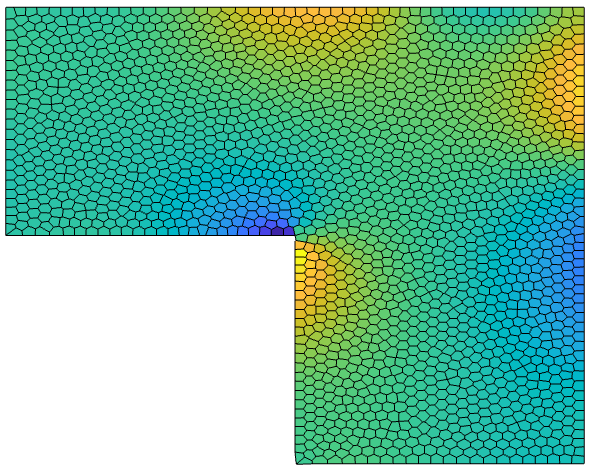}
			\centering\includegraphics[height=5.1cm, width=5.1cm]{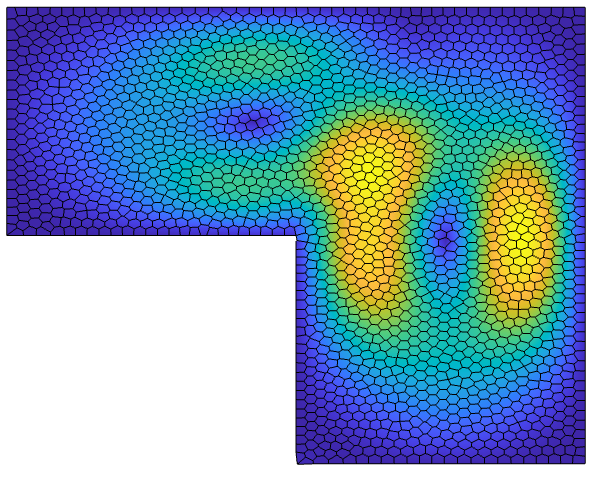}
			\centering\includegraphics[height=5.1cm, width=5.1cm]{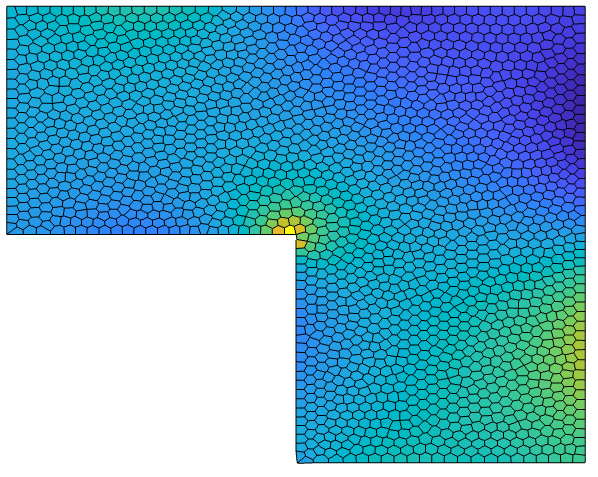}
			\centering\includegraphics[height=5.1cm, width=5.1cm]{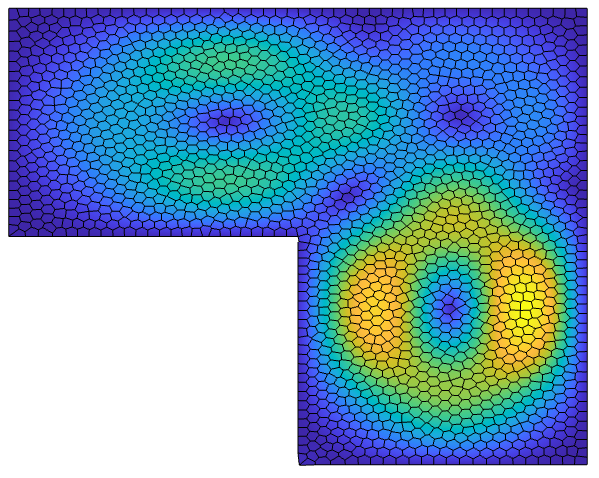}
			\centering\includegraphics[height=5.1cm, width=5.1cm]{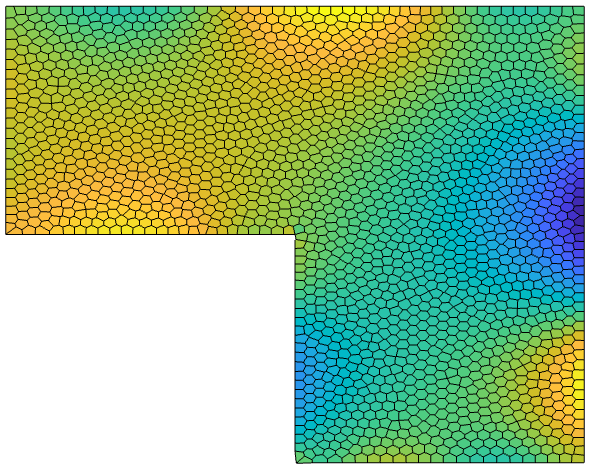}
		\caption{First, second, and third magnitude of the eigenfunctions in the nonconvex $L$ domain together with the associated pressures: first column $u_{1,h}$, $u_{2,h}$ and $u_{3,h}$ ;second column: $p_{1,h}$, $p_{2,h}$ and $p_{3,h}$; for different family of meshes.}
		\label{fig:L:shp}
	\end{center}
\end{figure}
\subsection{Spurious analysis}
The aim of this test is to analyze numerically the influence of the stabilization parameter on the computation of the spectrum. It is well know that if this parameter is not correctly chosen, 
may appear spurious eigenvalues. We refer to \cite{MR4253143,MR4229296,MR4050542} where the VEM reports this phenomenon. 
It is widely recognized that certain domain configurations, specifically dependent on convexity and boundary conditions, can lead to the emergence of spurious eigenvalues when stabilized methods are used, in contrast to implementations in domains with null Dirichlet boundary conditions. We refer to the reader to \cite{MR4077220,MR4229296} where this is discussed. Hence, for this experiment we consider the following problem: Given a domain $\O\subset\mathbb{R}^2$, let us assume that its boundary $\partial\O$ is such that $\partial\O:=\Gamma_D\cup\Gamma_N$ where $|\Gamma_D|>0$. 
\begin{equation}\label{def:oseen-eigenvalue_mixed_boundary}
\left\{
\begin{array}{rcll}
-\nu\Delta \bu + (\boldsymbol{\beta}\cdot\nabla)\bu + \nabla p&=&\lambda\bu&\text{in}\,\O,\\
\div \bu&=&0&\text{in}\,\O,\\
\bu &=&\boldsymbol{0}&\text{on}\,\Gamma_D,\\
(\nu\nabla\bu-p\boldsymbol{I})\cdot\boldsymbol{n} &=&\boldsymbol{0}&\text{on}\,\Gamma_N,
\end{array}
\right.
\end{equation}
where $\boldsymbol{I}\in\mathbb{C}^{d\times d}$ is the identity matrix. Clearly from \eqref{def:oseen-eigenvalue_mixed_boundary} a part of the boundary $\partial\Omega$ changes from Dirichlet to Neumann leading to a different configuration from problem \eqref{def:oseen-eigenvalue} and hence, the stabilization term may introduce spurious eigenvalues that cannot being observed on a clamped domain. Let us remark that the strategy to identify if an eigenvalue is spurious or not, is to check the shape of the associated eigenfunction. In particular, for the computational tests we have considered $\O:=(0,1)^2$ and $\boldsymbol{\beta}:=(1,0)^{\texttt{t}}$ as convective term.


\begin{figure}[h]
	\begin{center}
			\centering\includegraphics[height=5.1cm, width=5.1cm]{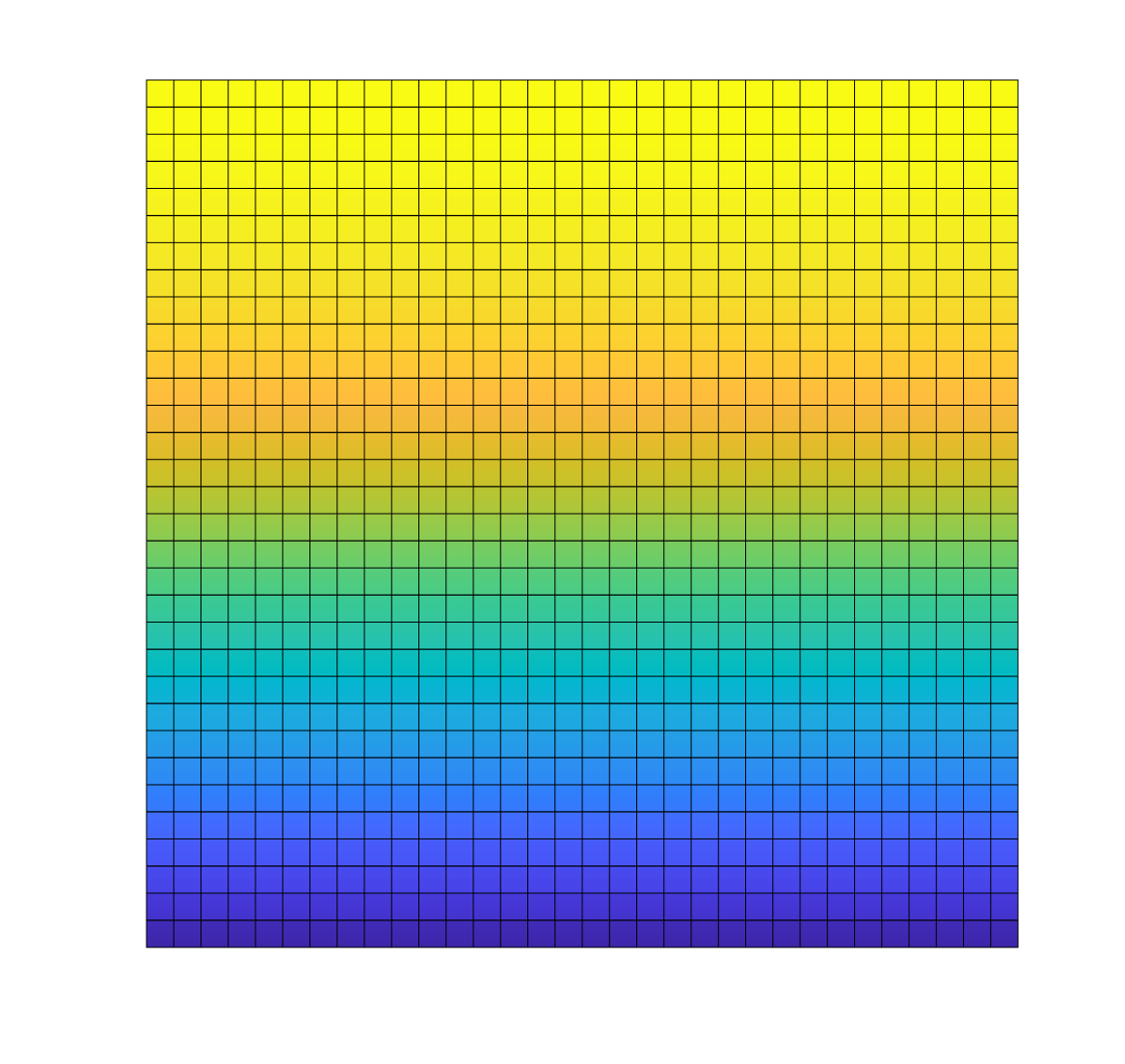}
			\centering\includegraphics[height=5.1cm, width=5.1cm]{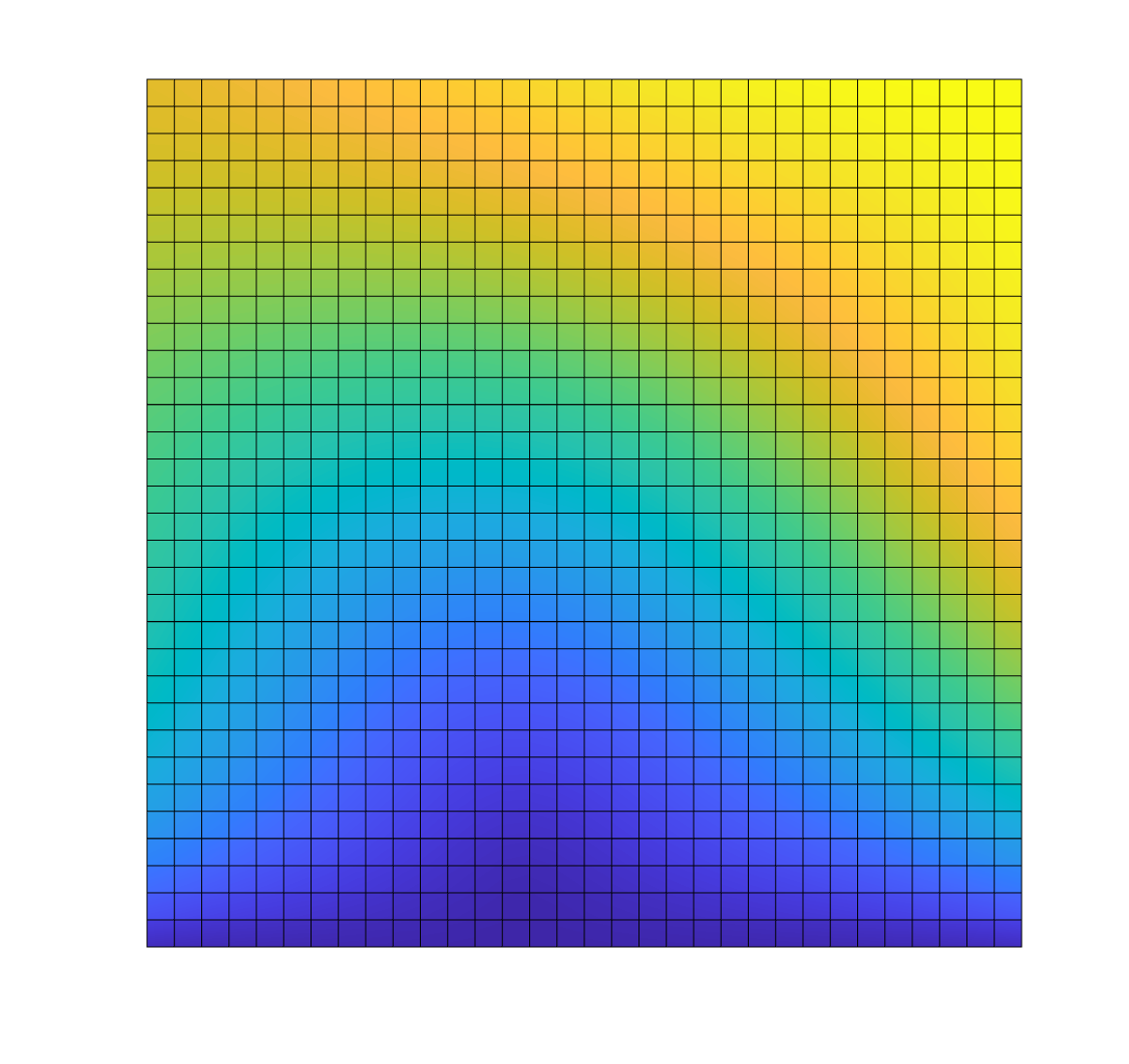}
			\centering\includegraphics[height=5.1cm, width=5.1cm]{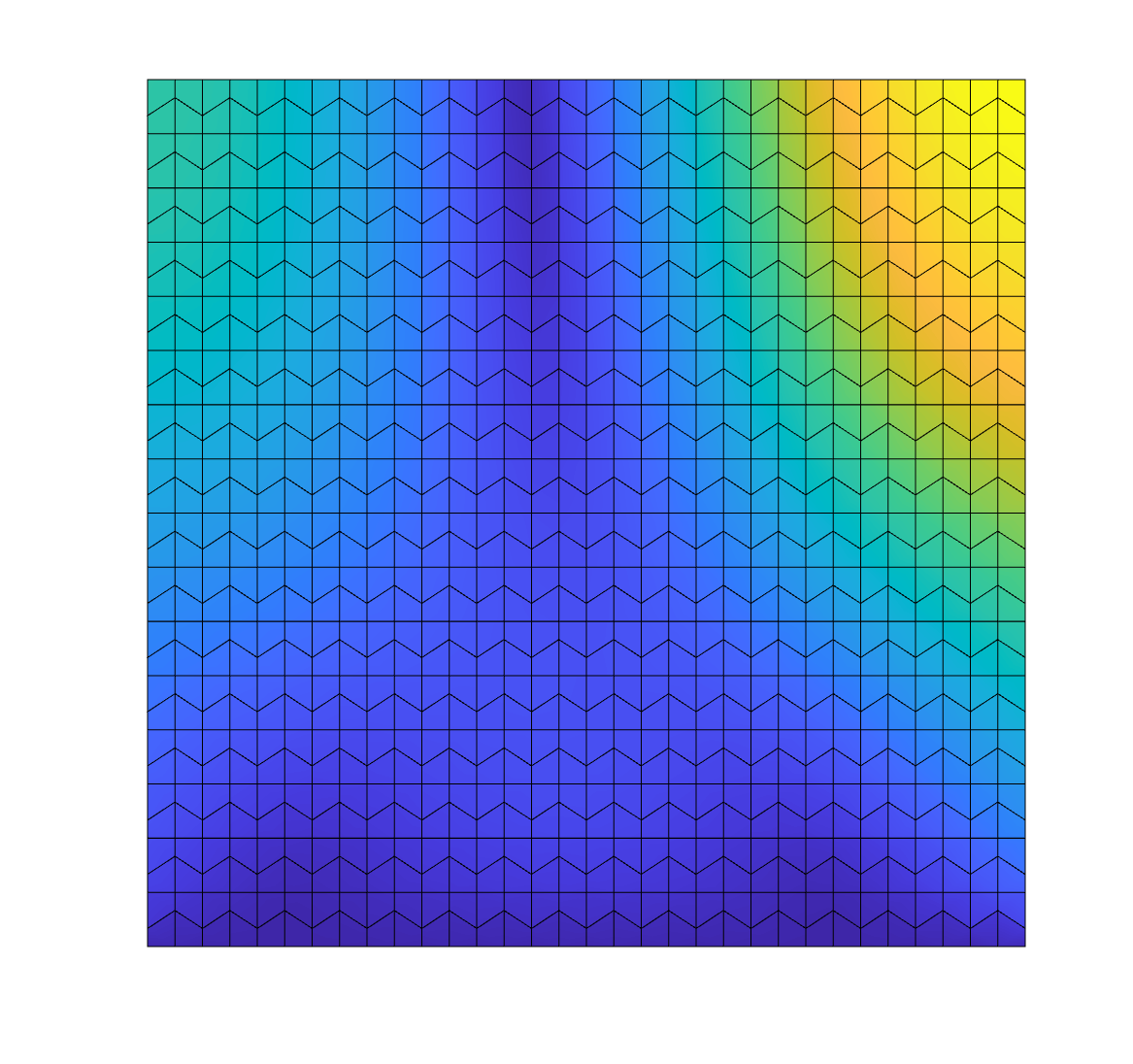}
			\centering\includegraphics[height=5.1cm, width=5.1cm]{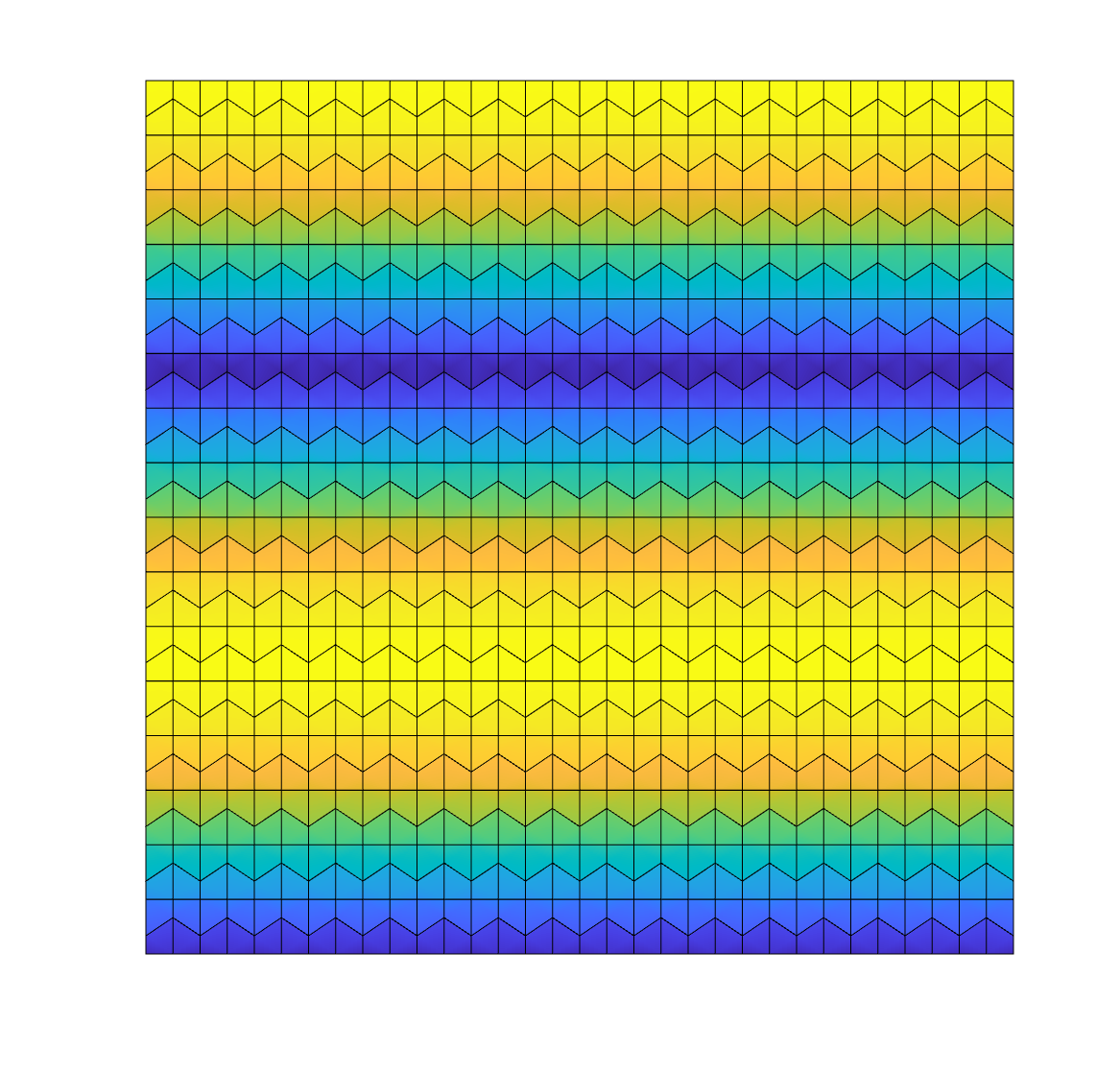}
		\caption{First, second, and third magnitude of the eigenfunctions with $N=32$, for different family of meshes.}
		\label{fig:spuroustest}
	\end{center}
\end{figure}

In Tables \ref{tablaespureos1} and \ref{tablaespureos2} we report the computed results for quadrilateral and Voronoi meshes, respectively.
From Table \ref{tablaespureos1} we observe that when the stabilization parameter $\alpha_E$ is small, more precisely, is such that $\alpha_E<1$, an important amount of spurious  eigenvalues arise on the computed spectrum which start to vanish when $\alpha_E$ increases. This phenomenon is clear for both families of meshes $\CT_h^1$ and $\CT_h^2$. For other families of polygonal meshes the results are similar. Also in Figure \ref{fig:p&u_L} we present plots of the magnitude of the computed  velocity and pressure, with different meshes.

\begin{table}[H]
\label{tablaespureos1}
\begin{center}
\caption{Computed eigenvalues for different values of $\alpha_E$ with $\mathcal{T}_{h}^{1}$.}
\begin{tabular}{|c|c|c|c|c|c|c|} \hline
 $\alpha_E$=1/32 & $\alpha_E$=1/16 & $\alpha_E$=1/4 & $\alpha_E$=1 & $\alpha_E$=4 & $\alpha_E$=16 & $\alpha_E$=32 \\ \cline{1-7}
\fbox{1.4756} & 2.0870 &   2.4106  &  2.4592   & 2.4699   & 2.4725&2.4729\\
 \fbox{1.6460}&  \fbox{2.9541} &   5.0781 &   5.8418  &  6.1009  &  6.1942&  6.2204\\
\fbox{1.7314}&   \fbox{3.4238}&  12.2493 &  14.9763  & 15.2397 &  15.3516& 15.3869\\
\fbox{1.7403}&    \fbox{3.4620} &  12.9070  & 21.1375  & 22.3902 &  22.6216& 22.6584\\
\fbox{1.7434}&    \fbox{3.4755} &   \fbox{13.4713}  & 24.3622  & 26.5618 &  27.0429 & 27.1458\\
 \fbox{1.7461}&   \fbox{3.4866} &   \fbox{13.5881}  & 37.6233  & 43.4899 &  44.4647& 44.6536\\
 \fbox{1.7465}&  \fbox{3.4883} &   \fbox{13.7754}  & 40.5498  & 46.3123 &  47.5366& 47.8232\\
 \fbox{1.7476}&    \fbox{3.4931}  &  \fbox{13.8329}  & 44.8864  & 62.6882 &  64.8430& 65.1451\\
  \fbox{1.7476}&    \fbox{3.4931}  &  \fbox{13.9038}  & 45.6918  & 62.8106 &  65.2323 & 65.6622\\
\fbox{1.7482}&   \fbox{3.4954}  &  \fbox{13.9206}  & 51.1740  & 73.0533  & 74.6701 & 75.0219
 \\ \hline
\end{tabular}
\end{center}
\end{table}

\begin{table}[H]
\label{tablaespureos2}
\begin{center}
\caption{Computed eigenvalues for different values of $\alpha_E$ with $\mathcal{T}_{h}^{2}$.}
\begin{tabular}{|c|c|c|c|c|c|c|} \hline
 $\alpha_E$=1/32 & $\alpha_E$=1/16 & $\alpha_E$=1/4 & $\alpha_E$=1 & $\alpha_E$=4 & $\alpha_E$=16& $\alpha_E$=32 \\ \cline{1-7}
1.3079& 1.9108 &   2.3682  &  2.4508    &2.4693  &  2.4738&2.4746 \\
 \fbox{1.4751}&   2.6176&    4.7627 &   5.7418   & 6.1175  &  6.2326&6.2538\\
 \fbox{1.5773}  &  \fbox{3.1053}&   10.8813  & 14.9485 &  15.2728 &  15.3987&15.4251\\
 \fbox{1.5888}  &   \fbox{3.1537} &  11.6653  & 20.2761  & 22.3258  & 22.7300& 22.7935\\
 \fbox{1.5929}   &  \fbox{3.1711} &   \fbox{12.2935}  & 23.3574  & 26.5470  & 27.1798&27.2809\\
  \fbox{1.5965}   &  \fbox{3.1857} &   \fbox{12.5435}  & 36.2960  & 43.3638  & 44.5662&44.7522\\
 \fbox{1.5970}    & \fbox{3.1879} &   \fbox{12.5978} &  38.9726  & 46.2787  & 47.8972&48.1768\\
 \fbox{1.5985}    & \fbox{3.1940} &   \fbox{12.6964} &  40.1479  & 61.8863  & 65.7328&66.2699\\
 \fbox{1.5986}    & \fbox{3.1946} &   \fbox{12.7105} &  41.7956  & 62.5039  & 66.1776&  66.7132\\
   \fbox{1.5993}  & \fbox{3.1973}  &  \fbox{12.7546}  & 47.2934  & 73.3563  & 75.1252&75.4144 
 \\ \hline
\end{tabular}
\end{center}
\end{table}

    The natural question now is if the refinement of the meshes causes some behavior on the spurious eigenvalues. To observe this, in Table \ref{tab21}
    we report the computed eigenvalues for $\alpha_E = 1/16$ and different refinements of the meshes $\CT_h^1$ and $\CT_h^2$.

\begin{table}[H]
\begin{center}
\caption{\label{tab21} First ten approximated eigenvalues for $\mathcal{T}_{h}^{1}$, $\mathcal{T}_{h}^{2}$ and $\alpha_E = 1/16$.}
\scalebox{0.85}{\begin{tabular}{|c|c|c|c|c|c|c|c|c|c|} \hline
\multicolumn{5}{ |c| }{$\mathcal{T}_{h}^{1}$} &\multicolumn{4}{ |c| }{$\mathcal{T}_{h}^{2}$}\\\hline
$\l_{i,h}$ & $N = 8$ & $N = 16$ & $N = 32$ & $N = 64$ &$N = 8$ & $N = 16$ & $N = 32$ & $N = 64$ \\ \hline \hline 
 $\l_{1,h}$ & 2.0870  &  2.4062  &  2.4536    &2.4640 &  1.9108 & 2.3625 & 2.4434 & 2.4675 \\
 $\l_{2,h}$ &  \fbox{2.9541} &   5.0980   & 5.9016    &6.1662 &2.6176 & 4.7627 &5.7403&6.2711\\
 $\l_{3,h}$ &  \fbox{3.4238}&    \fbox{12.1729}  & 15.0548  & 15.3446& \fbox{3.1053} & 10.7987 &14.9670&15.4816 \\
 $\l_{4,h}$ &  \fbox{3.4620}  & 12.8841 &  20.7115 &  21.9155&\fbox{3.1537} & \fbox{11.6268}& 19.7656&22.2157 \\
 $\l_{5,h}$ & \fbox{3.4755}&    \fbox{13.5330} &  24.3679  & 26.5839& \fbox{3.1711}&  \fbox{12.2229} &23.1339&27.1272 \\
 $\l_{6,h}$ & \fbox{3.4866}  & 13.5547  & 36.9583  & 42.3002& \fbox{3.1857}&  \fbox{12.5104} &35.4604&43.3846 \\
 $\l_{7,h}$ &  \fbox{3.4883} &   \fbox{13.7505} &  40.8357 &  46.9367&\fbox{3.1879}&  \fbox{12.5338} & \fbox{38.5668}& 48.4105 \\
 $\l_{8,h}$ &  \fbox{3.4931} &   \fbox{13.7849} &  43.3386 &  59.0853&  \fbox{3.1940}&  \fbox{12.6514} & \fbox{38.8406}&61.7552\\
 $\l_{9,h}$ &   \fbox{3.4931}   & \fbox{13.8772} &  45.3771  & 61.8600& \fbox{3.1946}&  \fbox{12.6648} &41.0988&64.6454\\
 $\l_{10,h}$ &  \fbox{3.4954} &   \fbox{13.8772}&   50.2525 &  73.6216&   \fbox{3.1973} & \fbox{12.7087} &45.7664&75.3587\\ \hline
\end{tabular}}
\end{center}
\end{table}
    
   Table \ref{tab21} reveals that a refinement strategy is capable to avoid the spurious eigenvalues from the spectrum. This is an important fact that confirms the good properties of the NCVEM on our eigenvalue context. In fact, we observe that when $\alpha_E = 1/16$ is considered, the spectrum gets cleaner when the mesh is refined. Moreover, this test suggests that $\alpha_E = 1$ is a suitable value to be considered for the approximation as in, for instance, \cite{MR3959470}.
   
   \subsection{Effect of the stabilization}\label{new:experiment}
We now present an experiment where we consider different stabilization parameters for the right-hand side. The idea is to observe the accuracy and the possible spurious eigenvalues in the computation of the spectrum. In \cite{MR4177014}, the authors raise the issue of selecting appropriate parameters for eigenvalue computations, particularly when the discretization method, such as VEM and NCVEM, requires fine-tuning of these parameters. This experiment considers the scalar stabilization for the left-hand side symmetric term and right-hand side term. More precisely, we consider the stabilized bilinear forms for $a_{h,\text{sym}}^K( \cdot, \cdot)$ and $\widetilde{c}_{h}^K(\cdot, \cdot)$:
\begin{equation}
\begin{split}
S^K(\bw_h,\bv_h)&= \alpha_K h_k^{d-2} \sum_{z=1}^{N^{\text{dof}}} \text{dof}_z( \bw_h) \text{dof}_z(\bv_h) \\
\widetilde{S}^K(\bw_h,\bv_h)&= \beta_K h_k^{d} \sum_{z=1}^{N^{\text{dof}}} \text{dof}_z(\bw_h) \text{dof}_z(\bv_h),
\end{split}
\end{equation}
where 
\begin{equation}
\widetilde{c}_{h}^K(\bw_h, \bv_h):=c^K(\BPIZ \bw_h,\BPIZ \bv_h)+ \widetilde{S}^K(\bw_h-\BPIZ \bw_h,\bv_h-\BPIZ \bv_h)
\end{equation}
This test consider the stabilization parameters
 $$\alpha_K, \beta_K \in \{ 0, 10^{-3},10^{-2}, 10^{-1}, 1, 10, 10^2, 10^3 \}.$$  
 As in the previous tests, we have considered $k=1$. Several comparisons of the eigenvalue behavior for different mesh sizes and stabilization parameters are presented in Figures~ \ref{alphadiv} to \ref{betaMass2}. In Figure~\ref{alphadiv}, we observe that the method can become unstable when removing the stabilization $\alpha_K$. The most accurate results occur near $\alpha_K=1$. An over prediction of the eigenvalues is observed for the coarse mesh when using high scaling values, but improves when the mesh is refined in Figure~\ref{alphadiv2}. This behavior of 
 the scheme is independent of choice of polygonal meshes since we have observed similar results in left and right panels  of Figure~\ref{alphadiv} where we have used quad ($\CT_h^2$) and Voronoi ($\CT_h^4$)    meshes respectively. In Figure~\ref{alphadiv}, and Figure~\ref{alphadiv2}, we have considered $\beta_K=0$.
In Figure~\ref{betaMass1}, and Figure~\ref{betaMass2},  we have the results from assuming $\alpha_K=1$ for the left-hand side, and changing $\beta_K$. Here, we have plotted each of the first ten eigenvalues (including their multiplicity) for each of the stabilization parameter choices. By comparing the results between the two meshes, we can observe that the eigenvalues do not change significantly for small values of mass matrix stabilization $\beta_K$. In fact, the accuracy improves for smaller mesh sizes. We support these results with those portrayed in Figure~\ref{betaMass2}, where we observe a comparison of the eigenvalue behavior for different mesh sizes and two choices stabilization parameters. It is clear that the discrete eigenvalues converges to the exact values, independent of $\beta_K$.
In summary, the results are in good agreement with the findings in \cite[Section 3]{MR4177014}, which show that small or nonexistent stabilization for the right-hand side is the safest choice, while the left-hand side behaves more accurately with default stabilization and near-one scaling. Note that we have also obtained a spurious spectrum for large values of $\beta_K$. Hence, the non-conforming scheme is significantly affected by spurious eigenvalues arising from the high stabilization of the right-hand side. This result is similar with those reported in \cite{MR4050542} where the authors present a test to study the robustness of the VEM to approximate the eigenvalues for the spectral elasticity problem. 

\begin{figure}[h]
	\begin{center}
			\centering\includegraphics[height=5.5cm, width=5.5cm]{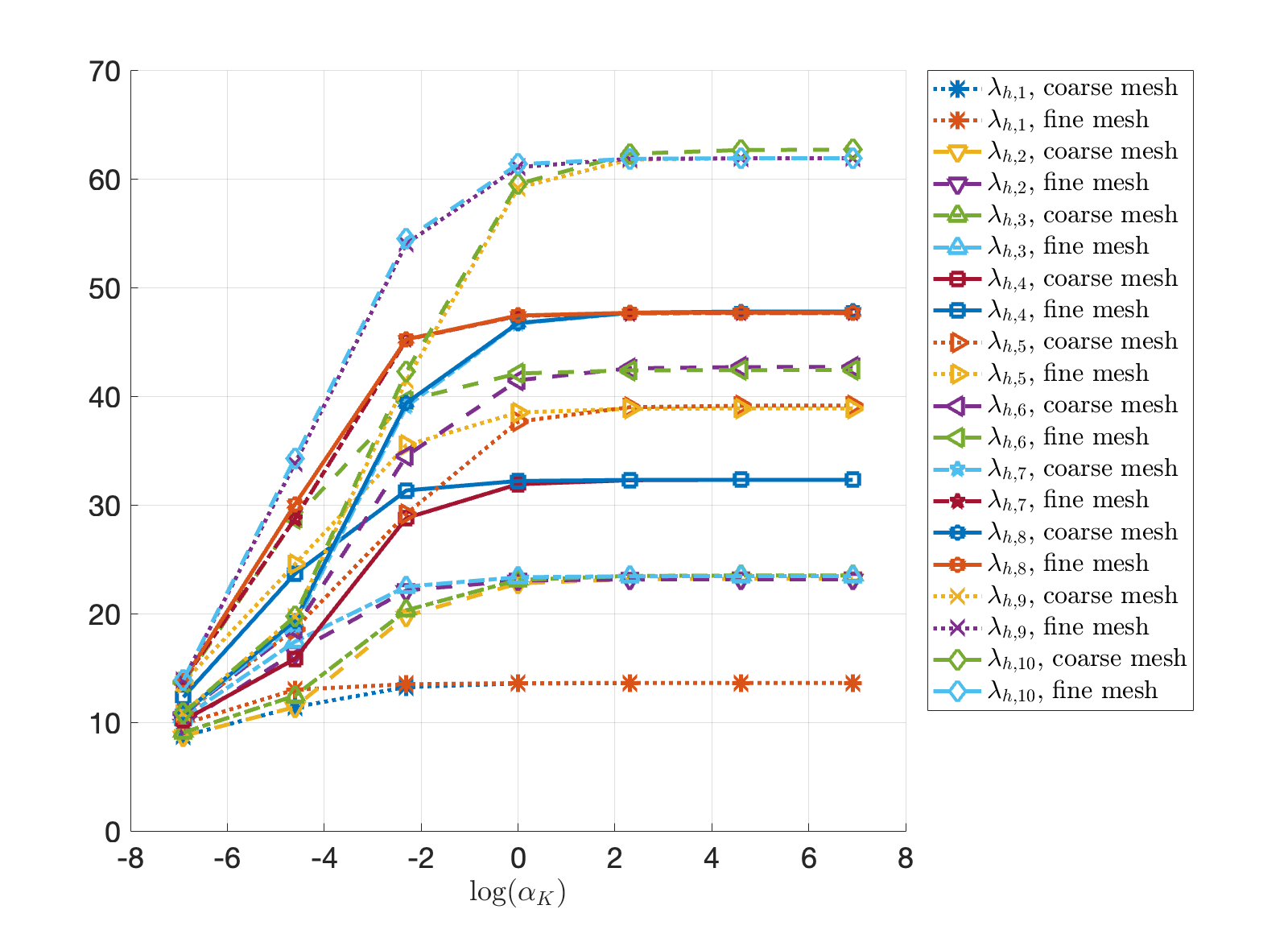}
			\centering\includegraphics[height=5.5cm, width=5.5cm]{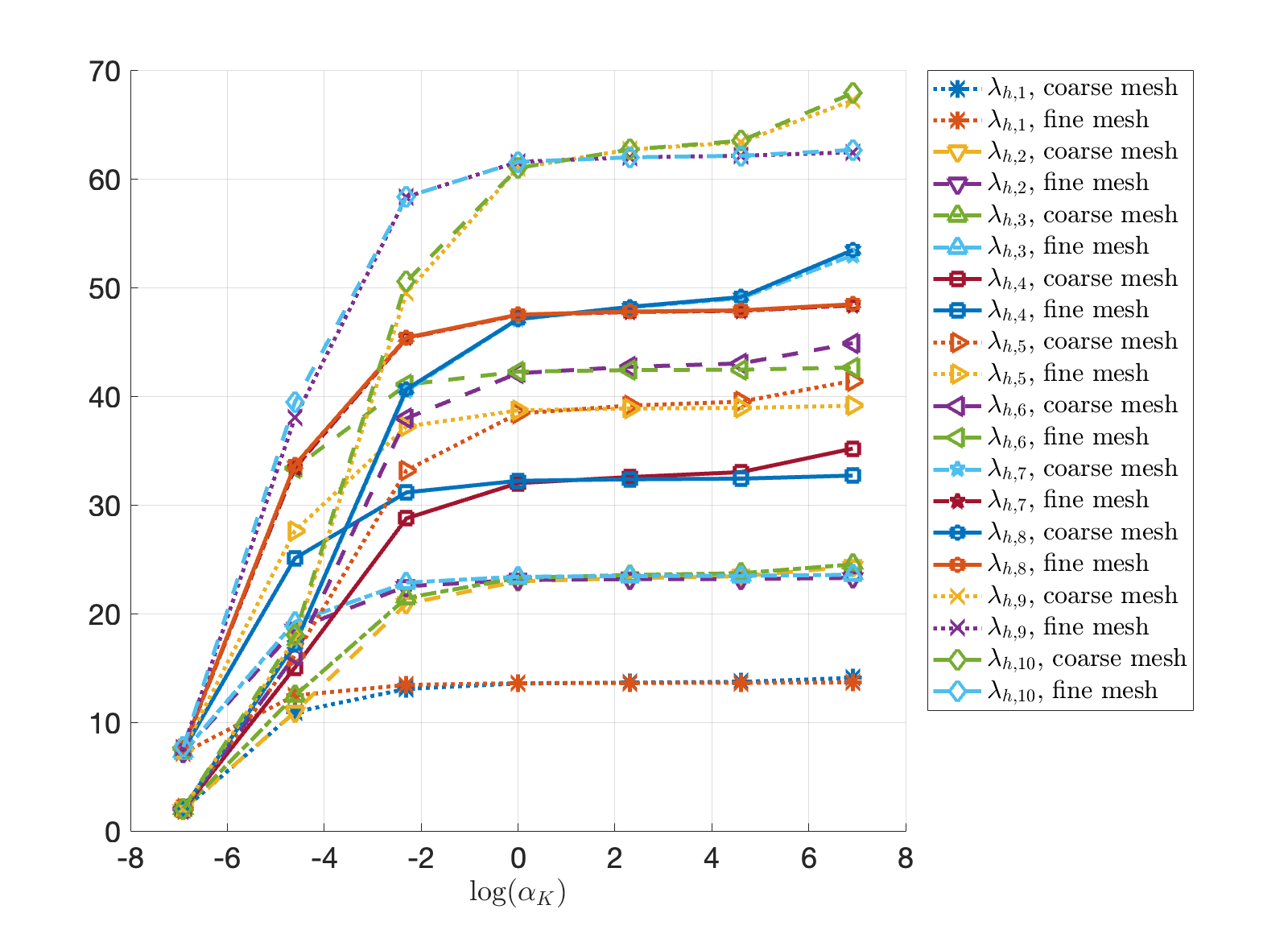}
		\caption{ Comparison between the computed eigenvalues for a coarse mesh ($h=1/32$) and a fine mesh ($h=1/64$) with respect to different choices of $\alpha_K$ . The domain is $\Omega=(-1,1)^2$ discretized using $\CT_h^2$ (left panel), and $\CT_h^4$ (right panel). The eigenvalues are plotted with their corresponding multiplicity.}
		\label{alphadiv}
	\end{center}
\end{figure}

\begin{figure}[h]
	\begin{center}
			\centering\includegraphics[height=5.5cm, width=5.5cm]{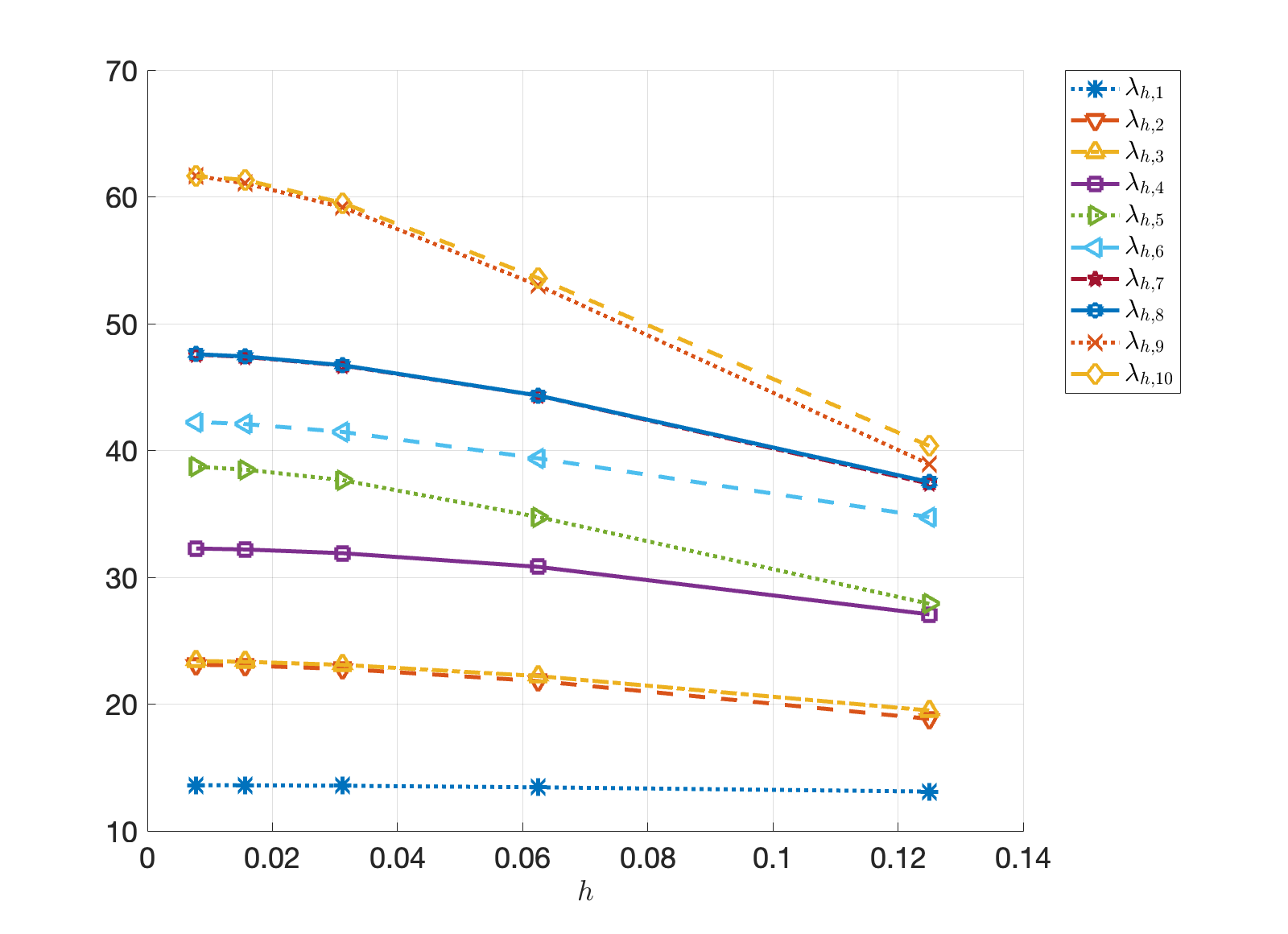}
			\centering\includegraphics[height=5.5cm, width=5.5cm]{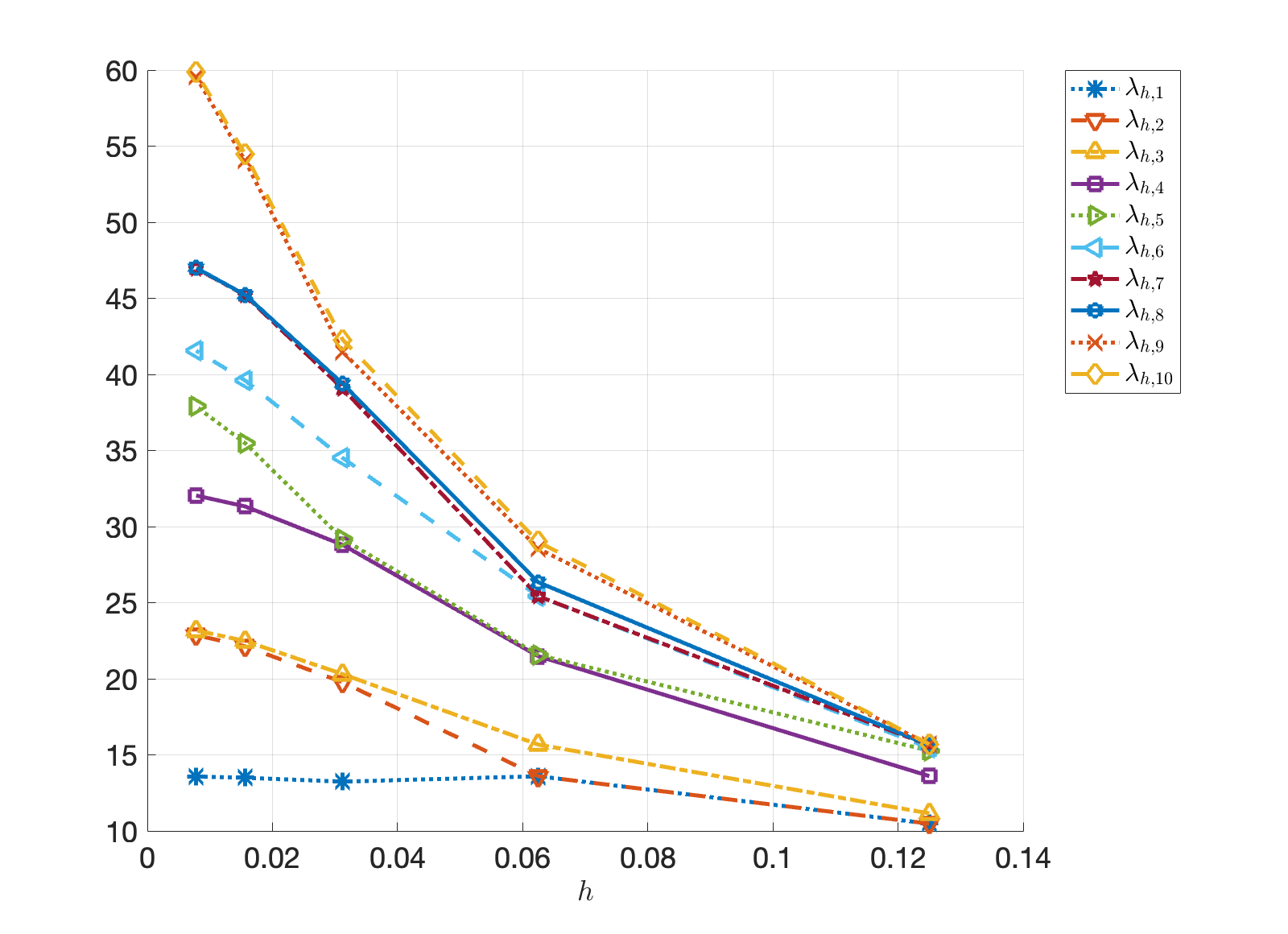}
		\caption{Comparison of the accuracy of the non-conforming virtual element scheme for $\beta_K=0$ and different values of $\alpha_K$ on the square domain $\Omega=(-1,1)^2$  using $\CT_h^2$. The eigenvalues are plotted with their corresponding multiplicity. We have picked out $\alpha_K=1$ on left panel and $\alpha_K=1/10$ on right panel}
		\label{alphadiv2}
	\end{center}
\end{figure}

\begin{figure}[h]
	\begin{center}
			\centering\includegraphics[height=5.5cm, width=5.5cm]{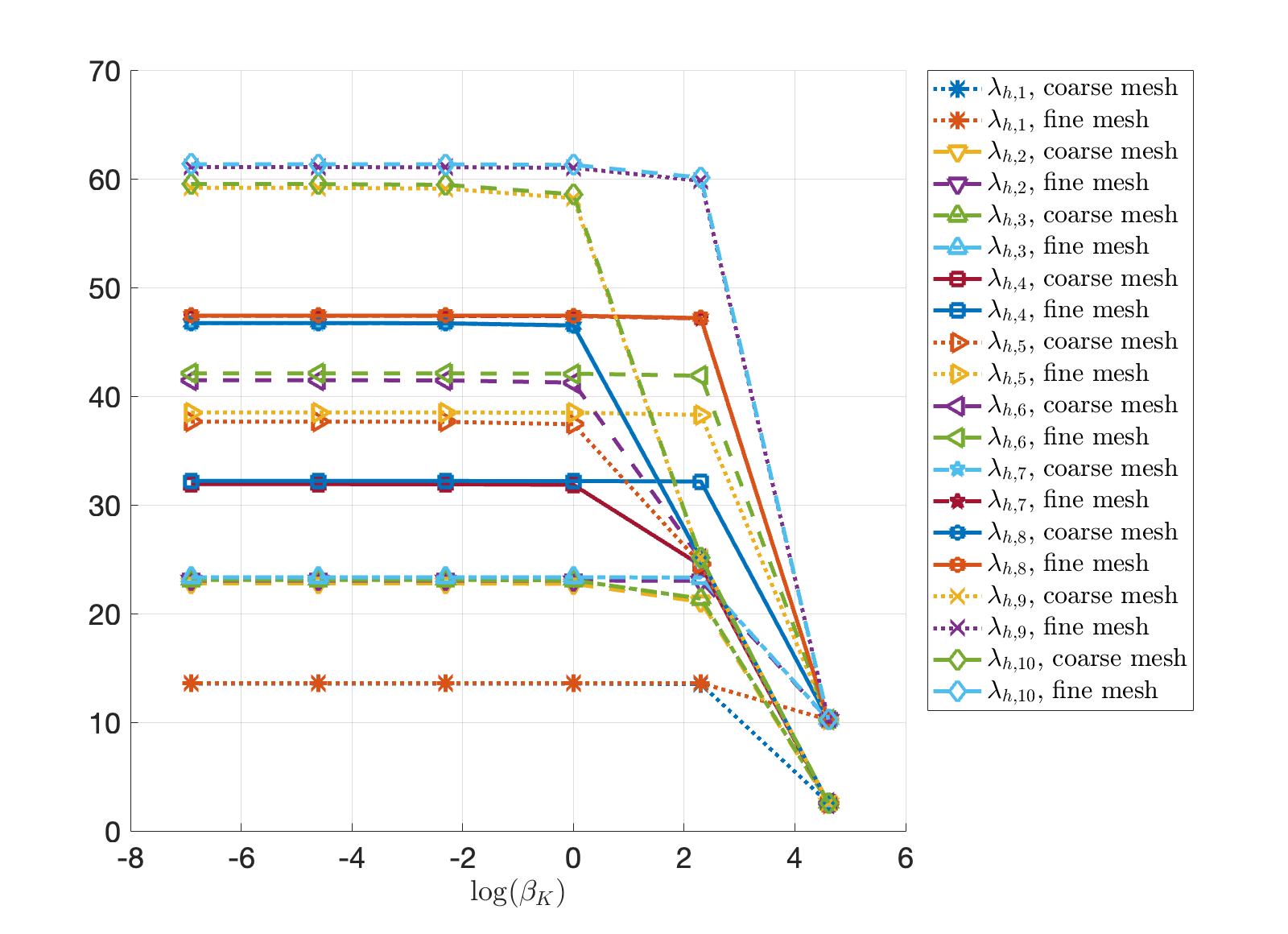}
			\centering\includegraphics[height=5.5cm, width=5.5cm]{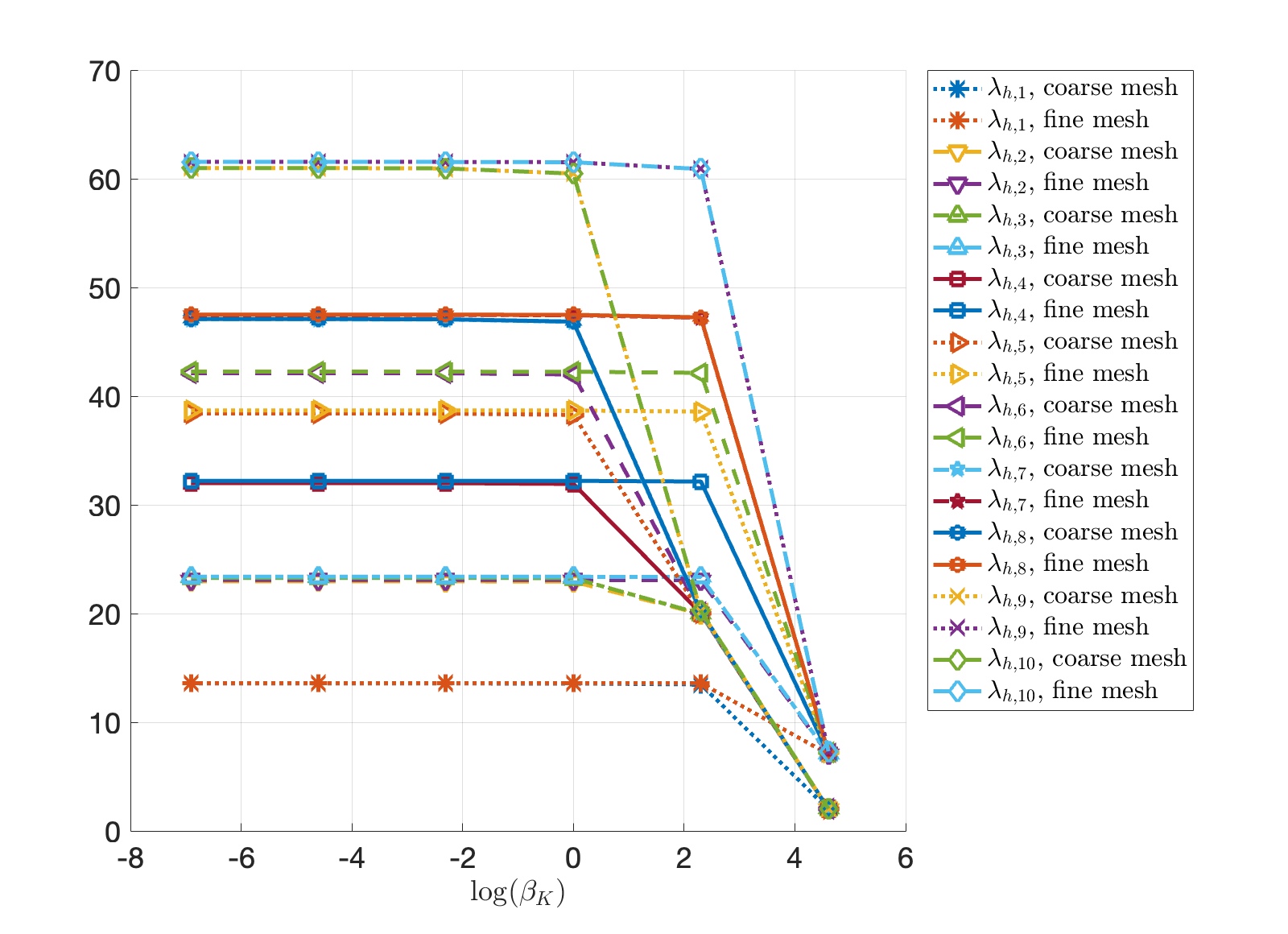}
		\caption{Comparison between the computed eigenvalues for a coarse mesh ($h=1/32$) and a fine mesh ($h=1/64$) with respect to different choices of $\beta_K$, and $\alpha_K=1$. The domain is $\Omega=(-1,1)^2$ discretized using $\CT_h^2$ (left panel), and $\CT_h^4$ (right panel). The eigenvalues are plotted with their corresponding multiplicity.}
		\label{betaMass1}
	\end{center}
\end{figure}

\begin{figure}[h]
	\begin{center}
			\centering\includegraphics[height=5.5cm, width=5.5cm]{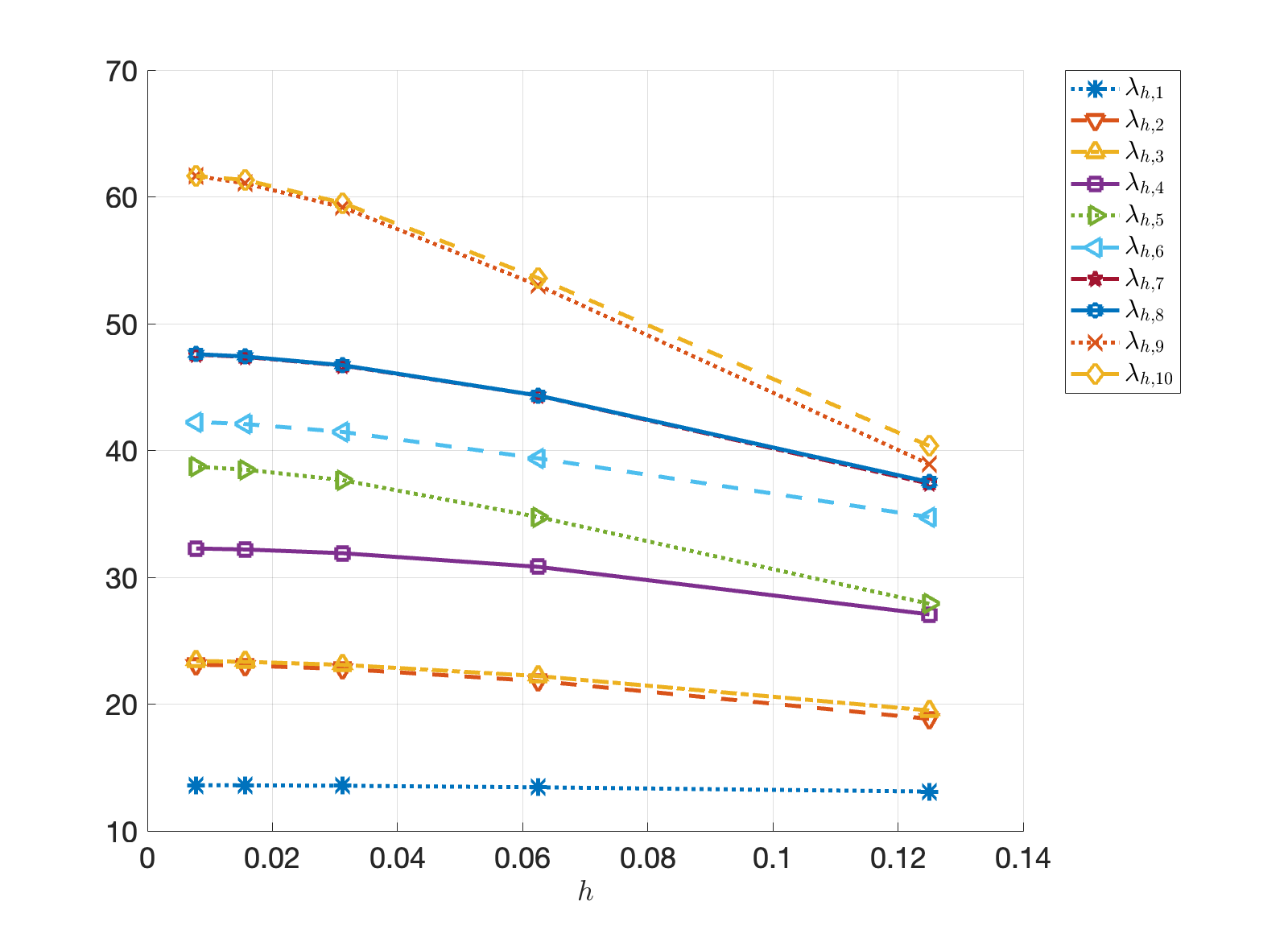}
			\centering\includegraphics[height=5.5cm, width=5.5cm]{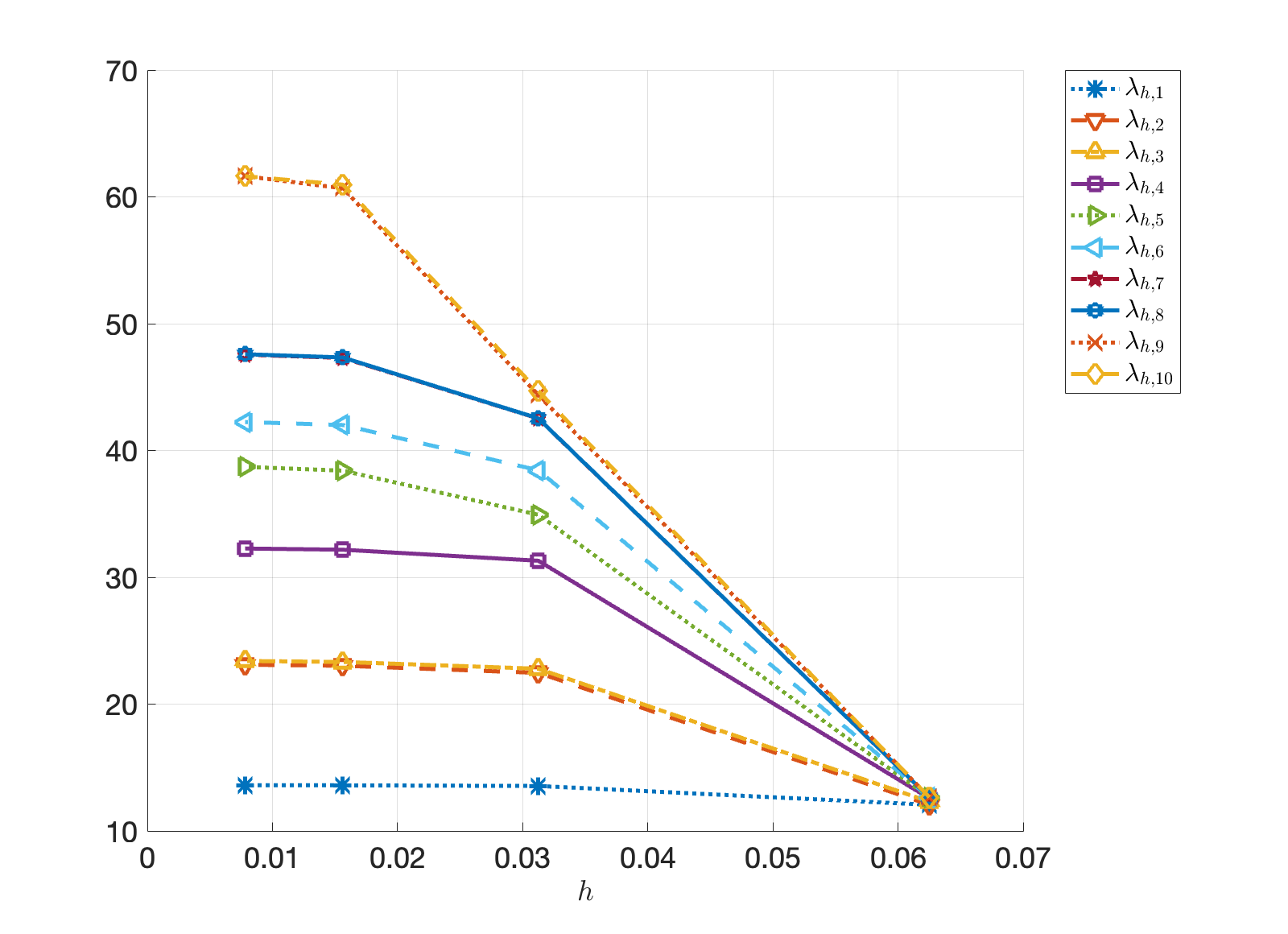}
		\caption{Comparison of the accuracy of the non-conforming virtual element scheme for $\alpha_K=1$ and different values of $\beta_K$ on the square domain $\Omega=(-1,1)^2$  using $\CT_h^2$. The eigenvalues are plotted with their corresponding multiplicity. We have picked out $\beta_K=0$ on left panel and $\beta_K=10$ on right panel}
		\label{betaMass2}
	\end{center}
\end{figure}

\section{Conclusion}
For the nonsymmetric Oseen eigenvalue problem, we have presented a divergence-free, arbitrary-order accurate, nonconforming virtual element approach that applies to highly generic shaped polygonal domains.  We performed a convergence study of the eigenfunctions using a solution operator on the continuous space. In addition, we utilized the idea of compact operators to define the discrete operator associated to the discrete problem and demonstrate the convergence of the approach. In the end, we were able to retrieve the double order of convergence of the eigenvalues by taking use of the extra regularity of the eigenfunctions. Our next area of interest will be a continuation of the analysis with minimum regularity of the eigenfunctions.

\bibliographystyle{siamplain}
\bibliography{oseen-eigenvalue}
\end{document}